\documentclass[11pt,reqno]{amsart}
\usepackage{hyperref}

\allowdisplaybreaks[4]
\usepackage{amssymb}
\usepackage{mathrsfs}
\usepackage{amsthm}
\usepackage{amsmath}
\numberwithin{equation}{section}
\usepackage{bm}
\usepackage{bbm}
\usepackage{graphicx}
\usepackage{booktabs}
\usepackage{float}
\usepackage{mathrsfs}
\usepackage{subfig}
\usepackage{geometry}
\usepackage{color}
\theoremstyle{plain}
\newtheorem{lemma}{Lemma}[section]
\newtheorem{theorem}[lemma]{Theorem}

\newtheorem{prop}[lemma]{Proposition}

\newtheorem{assp}{Assumption}
\newtheorem{rem}{Remark}

\newcommand{\E}{\mathbb{E}}
\newcommand{\PP}{\mathbb{P}}
\newcommand{\RR}{\mathbb{R}}

\def\e{\varepsilon}
\def\<{\langle}
\def\>{\rangle}
\def\l{\lambda}
\def\a{\alpha}
\def\nn{\nonumber}

\def\bc{\begin{center}}       \def\ec{\end{center}}
\def\ba{\begin{array}}        \def\ea{\end{array}}
\def\be{\begin{equation}}     \def\ee{\end{equation}}
\def\bea{\begin{eqnarray}}    \def\eea{\end{eqnarray}}
\def\beaa{\begin{eqnarray*}}  \def\eeaa{\end{eqnarray*}}

\begin{document}
\title[Probabilistic limit behaviors of  numerical discretizations]{Probabilistic limit behaviors of  numerical discretizations for time-homogeneous Markov processes}
\author{Chuchu Chen, Tonghe Dang, Jialin Hong, Guoting Song}
\address{LSEC, ICMSEC,  Academy of Mathematics and Systems Science, Chinese Academy of Sciences, Beijing 100190, China,
\and 
School of Mathematical Sciences, University of Chinese Academy of Sciences, Beijing 100049, China}
\email{chenchuchu@lsec.cc.ac.cn; dangth@lsec.cc.ac.cn; hjl@lsec.cc.ac.cn;
songgt606@lsec.cc.ac.cn}
\thanks{This work is funded by the National key R$\&$D Program of China under Grant (No. 2020YFA0713701), National Natural Science Foundation of China (No. 12031020, No. 11971470, No. 12022118, No. 11871068), and by Youth Innovation Promotion Association CAS, China.}
\begin{abstract}

In order to give quantitative estimates for approximating the ergodic limit, we investigate  probabilistic limit behaviors of time-averaging estimators of  numerical discretizations for a class of time-homogeneous Markov processes, by studying the corresponding strong law of large numbers and the central limit theorem. Verifiable   general sufficient conditions are proposed to ensure these limit behaviors, which are related to the properties of strong mixing and strong convergence for numerical discretizations of Markov processes. Our results hold for test functionals with lower regularity compared with existing results, and the analysis does not require the existence of the Poisson equation associated with the underlying Markov process. Notably, our results are applicable  to numerical discretizations for a large class of stochastic systems, including stochastic ordinary differential equations, infinite dimensional stochastic evolution equations, and stochastic functional differential equations.
\end{abstract}

\maketitle
\section{Introduction}\label{Tntr}
For solutions of a large class of autonomous stochastic differential equations, including  solutions of stochastic ordinary  differential equations (SODEs) and  infinite dimensional  stochastic evolution equations (SEEs), and functional solutions of stochastic functional differential equations (SFDEs), ``memorylessness'' is a universal phenomenon  and  usually characterized by the Markov property (see e.g. \cite{AAPergodic, Markovbook}). One of
the fundamental problems in the Markov process theory is the probabilistic limit   behaviors of the time-average 
$\frac1T\int_0^Tf(X_t)\mathrm dt$ as $T\to\infty$ for the Markov process $\{X_t\}_{t\ge 0}$,  where $f$ belongs to a class of test functionals. This problem is closely  related to the invariant measure $\mu$ of the Markov process (see e.g. \cite{AAPvariance}). It is well known that for the time-homogeneous Markov process with strong mixing, the time-average $\frac1T\int_0^Tf(X_t)\mathrm dt$ converges to the ergodic limit $\mu(f):=\int f\mathrm d\mu,$ and the normalized fluctuations around $\mu(f)$ can be
described by a centered Gaussian random variable, i.e., $\{X_t\}_{t\ge 0}$ fulfills the strong law of large numbers (LLN)
\begin{align*}
\lim_{T\to\infty}\frac1T\int_0^Tf(X_t)\mathrm dt=\mu(f)\quad \text{a.s.},
\end{align*}
and the central limit theorem (CLT)
\begin{align*}
\lim_{T\to\infty}\frac{1}{\sqrt{T}}\int_0^T(f(X_t)-\mu(f))\mathrm dt=\mathscr N(0,v^2)\quad \text{in distribution}
\end{align*}
with some variance $v^2$; see e.g. \cite{clt12,  Markovbook, Shirikyan}. In many circumstances, it is  unavoidable to approximate the ergodic limit by using the time-averaging estimator $\frac1k\sum_{i=0}^{k-1}f(Y_i)$ of a numerical discretization $\{Y_i\}_{i\in\mathbb N}$ for the Markov process.
Do numerical discretizations actually inherit  probabilistic limit behaviors in particular the strong LLN and the CLT? 
How to give quantitative estimates for the approximation of the ergodic limit?

There have been some works on the study of the strong LLN and the CLT of  numerical discretizations for stochastic systems. For example,  for SODEs with globally Lipschitz continuous coefficients, 
 the error between the time-averaging estimator  for Euler-type methods  and the invariant measure is estimated  in the almost sure sense, which implies the strong LLN (see e.g. \cite{Brehier16, siam10});  authors in \cite{AAP12} obtain  the CLT for the Euler--Maruyama (EM) scheme with decreasing step of a wide class of Brownian ergodic diffusions; 
authors in \cite{cltsde} prove the CLT and the self-normalized Cram\'er-type moderate deviation of the time-averaging estimator for the  EM method. For SODEs with non-globally Lipschitz continuous coefficients, the author in \cite{jdc2023} 
establishes the strong LLN and
the CLT of the backward EM method (BEM) by investigating the boundedness of the $p$th ($p > 2$) moment of the BEM method in the infinite time
horizon and the regularity of the solution to the Poisson equation. For the case of infinite dimensional SEEs, a first attempt is \cite{dang2023}, where the strong LLN and the  CLT for approximating the ergodic limit via a full discretization are established for parabolic stochastic partial differential equations.
More recently,  for the  stochastic
reaction-diffusion equation near sharp interface limit, authors in \cite{cui2023weak} establish a CLT of 
the temporal numerical approximation   by studying  properties of the corresponding Poisson equation. 
A tool used in the above mentioned papers to study the limit behaviors of numerical discretizations is the Poisson equation associated with the original stochastic system. The analysis generally requires that the solution of the Poisson equation possesses certain smoothness and boundedness. 
For some stochastic process, e.g. the functional solution of the SFDE, the structure of the Poisson equation is too complicated  to derive the required properties.
Nevertheless, for the general Markov process, it is not known yet if there are some criteria to judge the inheritance of probabilistic  limit behaviors by a numerical discretization of the process.

Focusing on these problems, 
 this paper proposes verifiable general sufficient conditions to ensure the strong LLN and the CLT
for numerical discretizations of time-homogeneous Markov processes.
 The conditions are related to  properties of strong mixing and strong convergence of numerical discretizations for  Markov processes.
 It is shown that once the numerical discretizations satisfy these conditions, the strong LLN and the CLT hold, i.e.,
\begin{align*}
\lim_{|\Delta|\to0}\lim_{k\to\infty}\frac{1}{k}\sum_{i=0}^{k-1}f(Y^{x^h,\Delta}_{t_i})= \mu(f)\quad\mbox{a.s.},
\end{align*}
and
\begin{align*}
\lim_{\tau\to0}\frac{1}{\sqrt {k\tau}}\sum_{i=0}^{k-1}\big(f(Y^{x^h,\Delta}_{t_i})-\mu(f)\big)\tau= \mathscr N(0,v^2)\quad\mbox{in distribution},
\end{align*}
where  parameters $k$ and $\Delta$ satisfy some relation in the CLT, and the limit variance $v^2=2\mu((f-\mu(f))\int_0^{\infty}(P_tf-\mu(f))\mathrm dt)$ coincides with the one of the underlying  Markov process. 
 We remark that in the strong LLN and the CLT, the test functional  $f$ is of  polynomial growth  and has certain weighted H\"older continuity.  
See Theorems \ref{LLNle} and \ref{CLTle} for details. 
%
%
There are two distinct advantages of our results. On  one hand, our approach  does not require the specific differential structure of the Markov process, and 
the proposed conditions for probabilistic limit behaviors of both the Markov process and its numerical discretizations are general and verifiable. 
This makes the results flexible and easy to apply to a large variety of stochastic systems and  corresponding numerical discretizations.
On the other hand, compared with the existing results, 
our results hold for test functionals with lower regularity.


For the proof of the CLT, a
 prerequisite
 is to construct  a suitable discrete martingale term such that it contains the essential contribution for the convergence of the normalized time-averaging estimator. 
As mentioned before, in the existing works, the discrete martingale term is extracted by means of the associated Poisson equation. For the general Markov process, the technique based on the Poisson equation is not applicable 
due to that the differential structure of the Markov process may be unknown.
To solve this problem, 
we construct a new discrete martingale term 
by fully utilizing the Markov property and strong mixing  of numerical discretizations of Markov processes. Some properties of the discrete martingale that support the essential contribution for the time-averaging estimator are fully analyzed; see Section \ref{property_mar} for details. 
With this in hand, and combining the full utilization   of the strong mixing of numerical discretizations and the detailed analysis of the internal relation between parameters $k$ and  $\Delta,$   the CLT of numerical discretizations is finally obtained with the
limit variance coinciding with the one of the underlying Markov process. We would like to mention in  Remark \ref{rmk3.2} that the strong LLN and the CLT still hold with certain  trade-off between the regularity of test functionals and the convergence order of numerical discretizations. Precisely, if the strong convergence of a numerical discretization is replaced by the weak convergence, 
our results hold for some class of 
test functionals.

As expected, 
our results are well applicable for numerical discretizations of a wide variety of stochastic systems, which  is supported by 
 discussions of 
SODEs, infinite dimensional  SEEs, and SFDEs.  
First, as a test case for our results, we consider SODEs with superlinearly growing coefficients and the corresponding  BEM method, which have been well studied in the literature. We show that the strong LLN and the CLT hold for both the exact solution and its numerical discretization for test functionals with lower regularity. This weakens the condition on test functionals of existing results (where, at least, the first and the second derivatives of test functionals  need to be bounded). Moreover, the trade-off between the regularity of test functionals and the convergence order of the numerical discretization is also discussed. 
Similar results can be derived for the infinite dimensional SEE case, which is our second application. In particular, for the stochastic Allen--Cahn equation whose coefficient grows superlinearly, our results imply  the strong LLN and the CLT for such equation 
and its full discretization, 
which has not been reported before to our knowledge.  
Last but not least, we consider the SFDE case which is  one of the stochastic systems that motivates this study. By verifying the proposed conditions,
we obtain the strong LLN and the CLT of the  EM method of the SFDE for the first time.


The paper is organized as follows. In the next section, some preliminaries are introduced, and  the probabilistic limit behaviors including the strong LLN and the CLT of the time-homogeneous Markov processes are presented, whose proofs are given in Appendix \ref{appen_exact}. Section \ref{sec_main} presents our main results, 
including sufficient conditions ensuring the strong LLN and the CLT of numerical discretizations as well as the applications to stochastic systems. Some verifications of conditions in the application are given in Appendices \ref{spdeorder} and \ref{sfdeproof}. 
  Section \ref{sec_LLN} and Section \ref{sec_CLT} are devoted to presenting proofs of the strong LLN and the CLT of numerical discretizations,  respectively, 
based on the convergence of the numerical invariant measure and the decomposition of the normalized time-averaging estimator.
The proofs of properties of the conditional variance of the martingale difference sequence are given in Section \ref{proof_prop}.

\section{Preliminaries}
In this section, we give some preliminaries  for the study of the $E$-valued time-homogeneous Markov process $\{X^x_t\}_{t\ge 0}$.
Throughout this article, we use $K$ to denote a generic positive constant  independent of the initial datum and the step-size $\Delta$, which may take different values at  different appearances. Let $\lfloor a \rfloor$ denote the integer part of a real number $a$ and $\lceil a\rceil$ denote the minimal integer greater than or equal to $a$. Let $a\vee{}b:=\max\{a,b\}$ and $a\wedge{}b:=\min\{a,b\}$ for real numbers $a$ and $b$. Denote by 
 $\stackrel{d}{\longrightarrow}$, $\stackrel{\mathbb P}{\longrightarrow}$, and $\stackrel{a.s.}{\longrightarrow}$ the convergence in distribution,  in probability, and  in the almost sure sense, respectively. 

\subsection{Settings}

Let $(E, \|\cdot\|)$ be a real-valued  separable Banach space and $\mathfrak B(E)$ be the Borel $\sigma$-algebra. 
Denote by $\mathcal P(E)$ the family of probability measures  on $(E,\mathfrak B(E))$ and by $\nu(f):=\int_{E}f(u)\mathrm d\nu(u)$ the integral of the measurable functional $f$ with respect to $\nu\in\mathcal P(E).$  
For fixed positive constants $p\geq1$ and $\gamma\in(0,1]$, define the quasi-metric  $d_{p,\gamma}$  on $E$  by
\begin{align}\label{5.1}
d_{p,\gamma}(u_1,u_2):=(1\wedge\|u_1-u_2\|^{\gamma})(1+\|u_1\|^p+\|u_2\|^p)^{\frac{1}{2}}.
\end{align}
Then we introduce a test functional space related to this quasi-metric.  Define $\mathcal C_{p,\gamma}:=\mathcal C_{p,\gamma}(E;\mathbb R)$  the set of continuous functions  on $E$ endowed with the following norm
\begin{align}\label{5.2}
\|f\|_{p,\gamma}:=\sup_{u\in E}\frac{|f(u)|}{1+\|u\|^{\frac{p}{2}}}+\sup_{\substack{u_1,u_2\in E\\u_1\neq u_2}}\frac{|f(u_1)-f(u_2)|}{d_{p,\gamma}(u_1,u_2)}.
\end{align}
It is worthwhile pointing out that $\mathcal C_{p_1,\gamma}\subset \mathcal C_{p_2,\gamma}$ for $p_2\geq p_1\geq1$ and  
$\mathcal C_{p,\gamma_2}\subset \mathcal C_{p,\gamma_1}$ for $0<\gamma_1\leq \gamma_2\leq1$, see e.g. \cite{Shirikyan} for a more general form of the class of spaces. 
The Wasserstein quasi-distance  induced by the quasi-metric is defined by 
\begin{align}\label{5.3}
\mathbb W_{p,\gamma}(\nu_1,\nu_2):=\inf_{\pi\in\Pi(\nu_1,\nu_2) }\int_{E\times E}d_{p,\gamma}(u_1,u_2)\pi(\mathrm du_1,\mathrm du_2),
\end{align}
where  $\nu_1,\nu_2\in
 \mathcal P_{p,\gamma}(E):=\{\nu\in\mathcal P(E):\int_{E\times E}d_{p,\gamma}(u_1,u_2)\mathrm d\nu(u_1)\mathrm d\nu(u_2)<\infty\}$ and $\Pi(\nu_1,\nu_2)$ denotes  the collection of all probability measures on $E\times E$ with
marginal measures $\nu_1$ and $\nu_2$.

Note  that  
 when $\nu_1(\|\cdot\|^{p})<\infty $ and $\nu_2(\|\cdot\|^{p})<\infty $,
 \begin{align*}
\mathbb W_{p,\gamma}(\nu_1,\nu_2)\nn
\leq&{} \big(1+\nu_1(\|\cdot\|^{p})+\nu_2(\|\cdot\|^{p})\big)^{\frac{1}{2}} \Big(\inf_{\pi\in\Pi(\nu_1,\nu_2) }\int_{E\times E}1\wedge\|u_1-u_2\|^{2\gamma}\pi(\mathrm du_1,\mathrm du_2)\Big)^{\frac{1}{2}}\nn\\
\leq&{} \big(1+\nu_1(\|\cdot\|^{p})+\nu_2(\|\cdot\|^{p})\big)^{\frac{1}{2}} \Big(\inf_{\pi\in\Pi(\nu_1,\nu_2) }\int_{E\times E}1\wedge\|u_1-u_2\|^{2}\pi(\mathrm du_1,\mathrm du_2)\Big)^{\frac{\gamma}{2}}\nn\\
= &{}\big(1+\nu_1(\|\cdot\|^{p})+\nu_2(\|\cdot\|^{p})\big)^{\frac{1}{2}}(\mathbb W_{2}(\nu_1,\nu_2))^{\gamma},
\end{align*}
where  \eqref{5.1} and the H\"older inequality are used, and $\mathbb{W}_p$ is the bounded-Wasserstein distance defined by 
\begin{align*}
\mathbb{W}_p(\nu_1, \nu_2):=\Big(\inf_{\pi\in\Pi(\nu_1, \nu_2)}\int_{E\times E}1\wedge\|u_1-u_2\|^p\pi(\mathrm du_1,\mathrm du_2)\Big)^{\frac{1}{p}}.
\end{align*}
This leads to that for any $f\in\mathcal C_{p,\gamma}$,
\begin{align}\label{5.4}
|\nu_1(f)-\nu_2(f)|&=\inf_{\pi\in \Pi(\nu_1,\nu_2)}\Big|\int_{E\times E}(f(u_1)-f(u_2))\pi(\mathrm du_1,\mathrm du_2)\Big|\nn\\
&\leq
\|f\|_{p,\gamma}\mathbb W_{p,\gamma}(\nu_1,\nu_2)\leq \|f\|_{p,\gamma} \big(1+\nu_1(\|\cdot\|^{p})+\nu_2(\|\cdot\|^{p})\big)^{\frac{1}{2}}(\mathbb W_{2}(\nu_1,\nu_2))^{\gamma}.
\end{align}

Let $\{X_t^{x}\}_{t\geq0}$ be a family of $E$-valued  time-homogeneous  Markov processes with the deterministic initial value $X_0^{x}=x\in E$ on a filtered probability space $( \Omega,\mathcal{F},\{\mathcal F_t\}_{t\ge 0},\PP )$,
 in which the corresponding expectation is denoted by  $\E$. 
 Let $\Delta$ denote the uniform step-size,
 where
 $\Delta=\tau$ if the temporal semi-discretization is considered for the finite or infinite dimensional case,
 and $\Delta=(h,\tau)^{\top}$ if the spatial-temporal full discretization is considered for the infinite dimensional case. Here,  $h$ and $\tau$ represent step-sizes in spatial and temporal directions, respectively. To unify the notation, we always use 
 $\Delta=(h,\tau)^{\top}\in[0,1]\times(0,1]$,  where $h=0$ when only the temporal semi-discretization is considered.  
 Denote by $E^h(\subset E)$ the state space of  discretizations with norm $\|\cdot\|_{h}$ satisfying  $\|u^h\|_{h}=\|u^h\|$ for $u^h\in E^h$.
 We note that $E^h$ is of finite dimensional when the spatial-temporal full discretization is considered.
 When only the temporal semi-discretization is considered for the finite or the infinite dimensional space $E$, we let $E^0:=E$. When the initial datum $x\in E^h,$ we always take $x^h=x.$
Denote $t_{k}:=k\tau$ for $k\in\mathbb N$.
For each given $\Delta$, let $\{Y^{x^h,\Delta}_{t_k}\}_{k\in\mathbb N}$ denote the $E^h$-valued discretization  of $\{X_t^{x}\}_{ t\geq0}$ satisfying $Y^{x^h,\Delta}_{0}=x^{h}$,
and define the continuous version  by
$$Y^{x^h,\Delta}_{t}=\sum_{k=0}^{\infty}Y^{x^h,\Delta}_{t_k} \textbf 1_{[t_k, t_{k+1})}(t).$$

Denote by $\mu_t^{x}$ and $\mu_t^{x^h,\Delta}$ the probability measures generated by  $X_t^{x}$ and $Y_t^{x^h,\Delta}$, respectively, i.e., for any $A\in\mathfrak{B}(E)$,
$
\mu_t^{x}(A)={}\PP\big\{\omega\in\Omega: X_{t}^{x}\in A\big\},\;
\mu_t^{x^h,\Delta}(A)={}\PP\big\{\omega\in\Omega: Y_{t}^{x^h,\Delta}\in A\big\}.
$
Denote by $\mathcal B_b:=\mathcal B_b(E;\mathbb R)$ (resp. $\mathcal C_b:=\mathcal C_b(E;\mathbb R)$) the family of bounded Borel measurable  (resp.  bounded and continuous) functions on $E$. When there is no confusion, we also denote $\mathcal B_b(E^h;\mathbb R),\mathcal C_b(E^h;\mathbb R),\mathcal C_{p,\gamma}(E^h;\mathbb R)$ by $\mathcal B_b,\,\mathcal C_b,\,\mathcal C_{p,\gamma},$ respectively. 

Define the linear operators $P_{t}$  and  $P_{t}^{\Delta}$ generated by $X^{\cdot}_{t}$ and $Y^{\cdot,\Delta}_{t}$, respectively, as 
\begin{align*}
&P_{t}:  \mathcal B_b\rightarrow  \mathcal B_b,~~~  P_{t} f(x):=\E\big(f(X^{x}_t)\big)=\int_{E} f(u)\mathrm d\mu^{x}_t(u) \quad\forall\; x\in E,\nn\\ 
&P_{t}^{\Delta}: \mathcal B_b\rightarrow  \mathcal B_b,~~~  P_{t}^{\Delta} f(x^h):=\E\big(f(Y^{x^h,\Delta}_t)\big)=\int_{E} f(u)\mathrm d\mu^{x^h,\Delta}_t(u)\quad\forall\; x\in E^h.
\end{align*}
According to the Markov property of  $X^{x}_t$, one deduces that $P_t$ is a Markov semigroup. 
If $X^{x}_t$ (resp. $Y^{x^h,\Delta}_t$) admits a unique invariant measure $\mu$ (resp. $\mu^{\Delta}$) satisfying $\mu(\|\cdot\|^{l})<\infty$ (resp. $\mu^{\Delta}(\|\cdot\|^{l})<\infty$) for some constant $l>0$, it follows from \cite[Remark 3.2.3]{DaPrato1996} that $P_{t}$ (resp. $P^{\Delta}_t$) is uniquely extendible to a bounded linear operator on $L^l(E,\mu)$ which is still denoted by $P_{t}$ (resp. $P^{\Delta}_{t}$). 

\subsection{Probabilistic limit behaviors of time-homogeneous Markov processes}
The following assumption is introduced for the time-homogeneous Markov process to ensure  the existence and uniqueness of the invariant measure as well as the probabilistic limit behaviors including the strong LLN and the CLT. 

 \begin{assp}\label{a1}
 Assume that the time-homogeneous  Markov process $\{X^x_t\}_{t\ge 0}$ satisfies the following conditions.
\begin{itemize}
\item[\textup{(i)}] There exist  constants $r\ge 2,\tilde r\ge 1$, and $L_1>0$ such that for any $x\in E$,
\begin{align*}
&\sup_{t\geq0}\E [\|X^{x}_t\|^{r}]\leq L_1(1+\|x\|^{\tilde rr}).
\end{align*}
\item[\textup{(ii)}] There exist constants $\gamma_1\in(0,1]$, $\beta\in[0,r-1],\, L_2>0$,  and a  continuous function $\rho:[0,+\infty)\to[0,+\infty)$  with
$
\int_{0}^{\infty} \rho^{\gamma_1}(t)\mathrm dt<\infty
$
such that for any $x,y\in E,$
\begin{align*}
\big(\E[\|X^{x}_{t}-X^{y}_{t}\|^2]\big)^{\frac{1}{2}}\leq L_2\|x-y\|(1+\|x\|^{\beta}+\|y\|^{\beta})\rho(t).
\end{align*}
\end{itemize}
 \end{assp} 
The probabilistic limit behaviors including the LLN and the CLT of Markov processes have been well-studied in the monograph \cite{Markovbook}. 
 For our convenience, we list them in the following proposition  and postpone   
the  proofs to  Appendix \ref{appen_exact} to make the article self-contained.

 \begin{prop}\label{Limit_exact}
 Let Assumption {\textup{\ref{a1}}} hold,  $\gamma\in[\gamma_1,1]$, and $p$ satisfy $2\tilde r\big(p\tilde r+(2+3\beta)\gamma\big)\leq r.$ Then $\{X^x_t\}_{t\ge 0}$ admits a unique invariant measure $\mu\in\mathcal P(E),$ and fulfills the strong LLN and the CLT: for any $f\in\mathcal C_{p,\gamma},$
 \begin{align*}
&\frac1T\int_0^Tf(X^x_t)\mathrm dt\overset{a.s.}\longrightarrow\mu(f)\text{ as }T\to\infty,\\
 &\frac{1}{\sqrt{T}}\int_0^T(f(X^x_t)-\mu(f))\mathrm dt\overset{d}\longrightarrow\mathscr N(0,v^2) \quad\text{ as }T\to\infty,
 \end{align*}
 where $v^2:=2\mu\big((f-\mu(f))\int_{0}^{\infty}(P_{t}f-\mu(f))\mathrm dt\big)$. 
 \end{prop}

\section{Main results and applications}\label{sec_main}
 In this section, we propose verifiable general sufficient conditions to ensure the strong LLN and the CLT of numerical discretizations for time-homogeneous Markov processes. The applications to SODEs, infinite dimensional  SEEs, and SFDEs are also presented. 
 \subsection{Sufficient conditions and main results}
 Sufficient conditions ensuring the strong LLN and the CLT of numerical discretizations are proposed, which are related to the strong   mixing and strong convergence of numerical discretizations. 
 \begin{assp}\label{a2}
Assume that there exists $\tilde\Delta=(\tilde h,\tilde \tau)^{\top}\in[0,1]\times(0,1]$ such that for any  $h\in[0,\tilde h]$, $\tau\in(0,\tilde\tau]$, and  $x^h\in E^h$, $\{Y^{x^h,\Delta}_{t_k}\}_{k\in\mathbb N}$ is time-homogeneous  Markovian and  satisfies  the following conditions.
\begin{itemize}
\item[\textup{(i)}] There exist  constants $q\ge 2,\tilde q\ge 1$, and $L_3>0$ such that
\begin{align*}
&\sup_{k\geq0}\E [\|Y^{x^h,\Delta}_{t_{k}}\|^{q}]\leq L_3 (1+\|x^h\|^{\tilde qq}).
\end{align*}
\item[\textup{(ii)}] There exist constants $\gamma_2\in(0,1]$, $\kappa\in[0,q-1], L_4>0$, and a    function $\rho^{\Delta}:[0,+\infty)\to[0,+\infty)$ with
$$
\sup_{\{h\in[0,\tilde h],\tau\in(0,\tilde\tau]\}}\tau\sum_{k=0}^{\infty} \big(\rho^{\Delta}(t_k)\big)^{\gamma_2}<\infty
$$
such that for any $x^h, y^h\in E^h$,
\begin{align*}
\big(\E[\|Y^{x^h,\Delta}_{t_k}-Y^{y^h,\Delta}_{t_k}\|^2]\big)^{\frac{1}{2}}\leq L_4\|x^h-y^h\|(1+\|x^h\|^{\kappa}+\|y^h\|^{\kappa})\rho^{\Delta}(t_k).
\end{align*}
\end{itemize}
 \end{assp} 

Below, when there is no confusion, we always assume that $h\in [0,\tilde h]$ and $\tau\in(0,\tilde{\tau}].$ And constants $r$ and $q$ in Assumptions \ref{a1} (\romannumeral1) and \ref{a2} (\romannumeral1), respectively, are always assumed to be large enough to meet the need.

\begin{assp}\label{a3}
Assume that there exist $\a=(\a_1,\a_2)^{\top}\in\mathbb R_+^2 $ and  $L_5>0$ such that for any $x\in E$, 
 $$\sup_{t\geq 0}\big(\E[\|X^{x}_{t}-Y^{x^h,\Delta}_{t}\|^2]\big)^{\frac{1}{2}}\leq L_5(1+\|x\|^{\tilde r\vee \tilde q})|\Delta^{\a}|.$$
where
$\Delta^{\a}:=(h^{\a_1},\tau^{\a_2})^{\top}$ and $|\Delta^{\a}|:=h^{\a_1}+\tau^{\a_2}$.
\end{assp}

 Denote the time-averaging estimator  of the numerical discretization by 
$$
\frac 1kS_{k}^{x^h,\Delta}:=\frac 1k\sum_{i=0}^{k-1}f(Y^{x^h,\Delta}_{t_i})
$$ for $f\in\mathcal C_{p,\gamma}$ with suitable parameters $p\ge 1$ and $\gamma\in(0,1].$ Then we present the strong LLN of the time-averaging estimator.
\begin{theorem}\label{LLNle}
Let Assumptions \textup{\ref{a1}--\ref{a3}} hold, $\gamma\in[\gamma_2,1],$
and  $p\ge 1$ satisfy  $p(1+\tilde q)+2(1+\kappa)\gamma\leq q$.  
Then for any $f\in\mathcal C_{p,\gamma},$ we have 
$$\lim\limits_{|\Delta|\rightarrow 0}\lim\limits_{k\to\infty}\frac{1}{k}S_{k}^{x^h,\Delta}=\mu(f) \quad \mbox{a.s.}$$
\end{theorem}
The following CLT characterizes the probabilistic limit behavior  of  
$\frac{1}{\sqrt {k\tau}} \sum_{i=0}^{k-1}\big(f(Y^{x^h,\Delta}_{t_i})-\mu(f)\big)\tau$, 
for $f\in\mathcal C_{p,\gamma}$ with   suitable parameters $p\geq 1$ and $\gamma\in(0,1]$. Recall that $v^2=2\mu\big((f-\mu(f))\int_{0}^{\infty}(P_{t}f-\mu(f))\mathrm dt\big)$ is given  in Proposition \ref{Limit_exact}.

\begin{theorem}\label{CLTle}
Let Assumptions \textup{\ref{a1}--\ref{a3}} hold,  $\gamma\in[\gamma_1\vee\gamma_2,1]$,  $k:=\lceil\tau^{-1-2\lambda}\rceil$ with  $\lambda\in(0,\alpha_2\gamma)$, and  $h:=\tau^{\alpha_2/\alpha_1}$ when $h\neq0$. Let $p\ge 1$ satisfy 
\begin{align}\label{condip}
\tilde q^2\Big(3\vee(\frac{1}{\lambda}+1)\Big)\Big(p(\tilde q\vee\tilde r)+\big(3+4(\kappa\vee\beta)\big)\gamma\Big)\leq q\wedge r,
\end{align}
and  $\e\ll\l$  with 
$\sqrt \e-\frac{\e}{\l}-\frac{\e\sqrt \e}{\l}>0$ satisfy  
\begin{align}\label{rho_remainder}
\lim_{\e\to0}\lim_{\tau\to0}\tau^{\e}\sum_{i=\lceil\tau^{-1-\frac{\e}{4}}\rceil}^{k}(\rho^{\Delta}(t_i))^{\gamma}=0.
\end{align} 
Then for any $f\in\mathcal C_{p,\gamma},$ 
\begin{align*}
\frac{1}{\sqrt {k\tau}}\sum_{i=0}^{k-1}\big(f(Y^{x^h,\Delta}_{t_i})-\mu(f)\big)\tau\stackrel{d}{\longrightarrow}\mathscr N(0,v^2)\quad \text{as }\tau\rightarrow0.
\end{align*}
\end{theorem}

\begin{rem}
\begin{itemize}
\item[(\romannumeral1)] If Assumption \ref{a1} \textup{(\romannumeral1)} is replaced by: for any $r\ge2,$ there exist $\tilde r\ge 1$ and $L_1>0$ such that $\sup_{t\ge 0}\mathbb E[\|X^x_t\|^r]\leq L_1(1+\|x\|^{\tilde rr}),$ and Assumption \ref{a2} \textup{(\romannumeral1)} is replaced by: for any $q\ge2,$ there exist $\tilde q\ge 1$ and $L_3>0$ such that $\sup_{k\ge 0}\mathbb E[\|Y^{x^h,\Delta}_{t_k}\|^q]\leq L_3(1+\|x^h\|^{\tilde qq}),$
then the restriction \eqref{condip} on $p$ can be  eliminated. 
\item[(\romannumeral2)] When $P^{\Delta}_{t_k}$ is exponential mixing, i.e., $\rho^{\Delta}(t_k)=e^{-ct_k}$ with some $c>0,$ then  \eqref{rho_remainder} is satisfied naturally.
\end{itemize}
\end{rem}

\begin{rem}\label{rmk3.2}
Our main results in Theorems \ref{LLNle} and  \ref{CLTle}   are proved  with a strong convergence condition but for  test functionals  with lower regularity, i.e.,  $f\in \mathcal C_{p,\gamma}$. They still hold with certain  trade-off between the  convergence condition of numerical discretizations (i.e., Assumption \ref{a3}) and the regularity of test functionals. 
Precisely,  if Assumption \ref{a3} is weaken as:   
 \begin{align}\label{weak_eq}
 \sup_{t\ge 0}|\mathbb E[f(X^{x}_{t})-f(Y^{x^h,\Delta}_{t})]|\leq K|\Delta^{\alpha}|
 \end{align} for  test functionals belonging to 
some space $\mathfrak C$,
then Theorems \ref{LLNle} and \ref{CLTle} still hold for  $f\in \mathcal C_{p,\gamma}\cap \mathfrak C$.

The main difference in  proofs of  this remark and Theorem \ref{CLTle}  lies in the estimation of terms  $|P_tf(u)-P^{\Delta}_tf(u)|,\,u\in E^h$  and $|\mu(f)-\mu^{\Delta}(f)|$ (see e.g.
terms \eqref{t3.1} and \eqref{add_eq2}), where   $\mu^{\Delta}$ is the  invariant measure of the numerical discretization 
(see Proposition \ref{prop_IM}). 
 Under the assumption \eqref{weak_eq}, we have that for $u\in E^h,$
\begin{align*}
|P_{t}f(u)-P^{\Delta}_{t}f(u)|\leq \sup_{t\ge 0}|\mathbb E[f(X^{u}_{t})-f(Y^{u,\Delta}_{t})]|\leq K|\Delta^{\alpha}|,
\end{align*}
and 
 $|\mu(f)-\mu^{\Delta}(f)|$ can be estimated as
\begin{align*}
&\quad |\mu(f)-\mu^{\Delta}(f)|\leq |P_{t_k}f(u)-\mu(f)|
+
|P^{\Delta}_{t_k}f(u)-\mu^{\Delta}(f)|+\sup_{k\ge 0}|P_{t_k}f(u)-P^{\Delta}_{t_k}f(u)|\\
&\leq K\|f\|_{p,\gamma}(1+\|u\|^{\frac{p\tilde r}{2}+(1+\beta)\gamma})\rho^{\gamma}(t_k)
+K\|f\|_{p,\gamma}(1+\|u\|^{\frac{p\tilde q}{2}+(1+\kappa)\gamma})(\rho^{\Delta}(t_k))^{\gamma}
 +K|\Delta^{\alpha}|\\
 &\leq K|\Delta^{\alpha}|,
\end{align*}
where we used  \eqref{Xmix} and  \eqref{trmix1}, and in the last step we
let $k\to\infty$.


\end{rem}

\subsection{Applications}\label{sec_ex}
In this subsection, we present applications of Theorems \ref{LLNle}  and \ref{CLTle} to numerical discretizations of stochastic systems, including  SODEs, infinite dimensional SEEs, and SFDEs. 
Denote by $\mathcal C^k_b:=\mathcal C^k_b(E;\mathbb R)$ (resp. $\widetilde {\mathcal C}^k_b:=\widetilde {\mathcal C}^k_b(E;\mathbb R)$) the bounded continuous functions (resp. continuous functions) that are 
continuously differentiable with bounded derivatives up to order $k.$
\subsubsection{SODEs}
In this subsection, we consider the case of  SODEs that has been extensively studied in the literature, see \cite{Brehier16, jdc2023, cltsde, siam10, AAP12} for instance.
Precisely, we focus on verifying that the BEM method satisfies Assumptions \ref{a2}--\ref{a3} of this paper, and exploring the trade-off between the 
regularity of test functionals and the convergence of the numerical discretization.   

The equation has the following form
\begin{align}\label{ex_sode}\mathrm d X^x_t=b(X^x_t)\mathrm dt+\sigma (X^x_t)\mathrm dW(t),\quad X^x_0=x,
\end{align} where $\{W(t),t\ge 0\}$ is a $D$-dimensional standard Brownian motion on a filtered probability space $(\Omega,\mathcal F,\{\mathcal F_t\}_{t\ge 0},\mathbb P)$, the superlinearly growing coefficient $b:\mathbb R^d\to\mathbb R^d$ satisfies the dissipative condition, and $\sigma:\mathbb R^d\to\mathbb R^{d\times D}$ is Lipschitz continuous.  
The BEM method of \eqref{ex_sode} reads as $$Y^{x,\Delta}_{t_{i+1}}=Y^{x,\Delta}_{t_i}+b(Y^{x,\Delta}_{t_{i+1}})\tau+\sigma(Y^{x,\Delta}_{t_{i}})\delta W_i,\quad \delta W_i=W(t_{i+1})-W(t_i).$$

We first consider \eqref{ex_sode} with coefficients  satisfying conditions in \cite{jdc2023}, from which we obtain that 
\begin{itemize}
\item[(Ea)]
 Assumption \ref{a1} (\romannumeral1) is satisfied with $\tilde r=1$ and any $r\ge 2;$ 
 \item[(Eb)] Assumption \ref{a1} (\romannumeral2) is satisfied with $\beta=0$,  $\rho(t)=e^{-ct}$ for some $c>0$, and any $\gamma_1\in (0,1]$; 
\end{itemize}
see  \cite[Proposition 2.2]{jdc2023} for details. Thus by virtue of Proposition \ref{Limit_exact}, the exact solution $\{X^x_t\}_{t\ge 0}$ admits a unique invariant measure $\mu$ and fulfills the strong LLN and the CLT. 

For the BEM method, 
\begin{itemize}
\item[(Na)] Assumption \ref{a2} (\romannumeral1) is satisfied with $\tilde q=1$ and any  $q\ge 2;$ 
\item[(Nb)] Assumption \ref{a2} (\romannumeral2) is satisfied with $\kappa=0,\rho^{\Delta}(t)=e^{-ct}$ for some $c>0,$ and any $\gamma_2\in(0,1];$
\item[(Nc)] Assumption \ref{a3} is satisfied with 
$|\Delta^{\a}|=\tau^{\frac{1}{2}}$; 
\end{itemize}
see  \cite[Propositions 2.2 and 3.1]{jdc2023} for details. Thus by Theorem \ref{LLNle}, for any $p\ge 1,$  $\gamma\in(0,1],$ we have
\begin{align*}
\lim_{\tau\to0}\lim_{k\to\infty}\frac 1kS^{x,\Delta}_k=\mu(f)\quad \text{a.s.}\quad \forall\; f\in\mathcal C_{p,\gamma}.
\end{align*}
By Theorem \ref{CLTle}, we have that for any $p\ge 1,\,\gamma\in(0,1]$, $k=\lceil \tau^{-1-2\lambda}\rceil$ with  $\lambda\in(0,\frac12\gamma)$, 
\begin{align}\label{clt_sode}
\frac{1}{\sqrt{k\tau}}\sum_{i=0}^{k-1}\big(f(Y^{x,\Delta}_{t_i})-\mu(f)\big)\tau\overset d\longrightarrow\mathscr{N}(0,v^2)\quad\text{ as }\tau\to0\quad \forall\; f\in\mathcal C_{p,\gamma}.
\end{align}  

As we stated in Remark \ref{rmk3.2}, when the weak convergence of the BEM method is used, 
Theorems \ref{LLNle} and \ref{CLTle} still hold for some class of test functionals. To simplify the presentation,  we use the weak error result given in \cite{Luyulan}, where  the time-independent weak error analysis for the BEM method for the SODE with piecewise continuous arguments is  investigated. By means of the results for the case without memory,  one derives that (Ea) is replaced by 
\begin{itemize}
\item[(Ea')]
 Assumption \ref{a1} (\romannumeral1) is satisfied with $\tilde r=1$ and some $r\ge 2;$ 
\end{itemize}
see  \cite[Lemma 2.4]{Luyulan} for details. Thus by virtue of Proposition \ref{Limit_exact}, the exact solution $\{X^x_t\}_{t\ge 0}$ admits a unique invariant measure $\mu$ and fulfills the strong LLN and the CLT. 

For the BEM method, (Na) and (Nc) are replaced respectively by
\begin{itemize}
\item[(Na')] Assumption \ref{a2} (\romannumeral1) is satisfied with $\tilde q=1$ and some  $q\ge 2;$ 
\item[(Nc')] Remark \ref{rmk3.2} is satisfied with $\sup_{t\ge 0}|\mathbb E[f(X^x_{t})-f(Y^{x,\Delta}_{t})]|\leq K\tau$ for $f\in\mathcal C^3_b$; 
\end{itemize}
see  \cite[Eq. (13) and Theorem 4.9]{Luyulan} for details. Thus by Theorem \ref{LLNle} and the fact that $\mathcal C^3_b\subset \mathcal C_{p,\gamma}$, we have 
\begin{align*}
\lim_{\tau\to0}\lim_{k\to\infty}\frac 1kS^{x,\Delta}_k=\mu(f)\quad \text{a.s.}\quad \forall\; f\in\mathcal C_{b}^3.
\end{align*}
By Theorem \ref{CLTle}, we have that 
for $k=\lceil \tau^{-1-2\lambda}\rceil$ with  $\lambda\in(0,1)$, 
\begin{align}\label{clt_sode}
\frac{1}{\sqrt{k\tau}}\sum_{i=0}^{k-1}\big(f(Y^{x,\Delta}_{t_i})-\mu(f)\big)\tau\overset d\longrightarrow\mathscr{N}(0,v^2)\quad\text{ as }\tau\to0\quad \forall\; f\in\mathcal C_b^3.
\end{align}  
Note that the solution of the Poisson equation $\mathcal L\varphi=f-\mu(f)$ is $\varphi=-\int_0^{\infty}(P_tf-\mu(f))\mathrm dt,$ where $\mathcal L$ is the generator of  \eqref{ex_sode}. When $f\in\mathcal C_b^3,$ properties of the Poisson equation are well-studied, see e.g. \cite{siam10}. 
It follows from  $\varphi\mathcal L\varphi=\frac12 \mathcal L\varphi^2-\frac12 \|\sigma^{\mathrm T}\nabla \varphi\|^2$ and $\mu(\mathcal L\varphi^2)=0$ that $v^2=-2\mu(\varphi\mathcal L\varphi)=\mu(\|\sigma^{\top}\nabla \varphi\|^2)$, which is consistent with that of \cite[Theorems 3.2 and 3.4]{jdc2023}.

\subsubsection{Infinite dimensional SEEs}
In this subsection, we consider the case of the infinite dimensional SEEs, which has been studied in \cite{AAPergodic, dang2023, cui2021, Kruse}. We focus on verifying that full discretizations  satisfy  Assumptions \ref{a2}--\ref{a3} 
and investigating the trade-off between the regularity of test functionals and the convergence of the numerical discretization.

Consider the following infinite dimensional SEE  on $E:=L^2((0,1);\mathbb R)$
\begin{align}\label{ex_spde}
\mathrm dX^x_t=AX^x_t\mathrm dt+F(X^x_t)\mathrm dt+\mathrm dW(t), \quad X^x_0=x,
\end{align}
where $A:\textrm{Dom}(A)\subset E\to E$ is the Dirichlet Laplacian with homogeneous Dirichlet boundary conditions, and  $\{W(t),t\ge 0\}$ is a generalized $Q$-Wiener process on a filtered probability space $(\Omega,\mathcal F,\{\mathcal F_t\}_{t\ge 0},\mathbb P)$ satisfying the usual regularity condition $\|(-A)^{\frac{\beta_1-1}{2}}Q^{\frac12}\|_{\mathcal L_2}<\infty$ for $\beta_1\in(0,1]$ with $\|\cdot\|_{\mathcal L_2}$ being the Hibert--Schmidt operator norm, and $F$ is globally Lipschitz and satisfies the dissipative condition.  
For this equation,
\begin{itemize}
\item[(Ea)] Assumption \ref{a1} (\romannumeral1) is satisfied with $\tilde r=1$ and any $r\ge 2;$
\item[(Eb)] Assumption \ref{a1} (\romannumeral2) is satisfied with $\beta=0,$  $\rho(t)=e^{-ct}$ for some $c>0$, and any $\gamma_1\in(0,1];$
\end{itemize}
see \cite[Proposition 2.1]{dang2023} for details. 
Thus, by Proposition \ref{Limit_exact}, the exact solution 
$\{X^x_t\}_{t\ge 0}$ admits a unique invariant measure $\mu$ and fulfills the strong LLN and the CLT.

Consider the full discretization $\{Y^{x^{h},\Delta}_{t_i}\}_{i\in\mathbb N}$, whose spatial direction is based on the spectral Galerkin method and temporal direction is the exponential Euler method,
\begin{align*}
Y^{x^h,\Delta}_{t_{k+1}}=S^h(\tau)(Y^{x^h,\Delta}_{t_k}+P^hF(Y^{x^h,\Delta}_{t_k})\tau+\delta W^h_k),\quad Y^{x^h,\Delta}_0=x^h,
\end{align*}
where $h:=\frac{1}{N}$ with $N\in\mathbb N_+$,  $P^h$ is the spectral Galerkin projection, $\delta W_k^h=P^h(W(t_{k+1})-W(t_k))$, and $ S^h(\tau):=e^{A^h\tau}$ with $A^h:=P^hA.$
For this numerical discretization, 
\begin{itemize}
\item[(Na)] Assumption \ref{a2} (\romannumeral1) is satisfied with $\tilde q=1$ and any $q\ge 2$;
\item[(Nb)] Assumption \ref{a2} (\romannumeral2) is satisfied with $\kappa=0$, $\rho^{\Delta}(t)=e^{-ct}$ for some $c>0$, and any $\gamma_2\in(0,1];$
\item[(Nc)] Assumption \ref{a3} holds with $|\Delta^{\alpha}|=\tau^{\frac{\beta_1}{2}}+h^{\beta_1};$
\end{itemize}
see  \cite[Propositions 2.3--2.5]{dang2023}.
It follows from Theorem \ref{LLNle} that the strong LLN holds: for any $p\ge1,\,\gamma\in(0,1],$
\begin{align}\label{lln_spde1}
\lim_{|\Delta|\to 0}\lim_{k\to\infty}\frac1k S^{x^h,\Delta}_k=\mu(f)\quad \text{a.s.}\quad \forall\; f\in\mathcal C_{p,\gamma}.
\end{align}
Combining Theorem \ref{CLTle}, we have that for $h=\tau^{\alpha_2/\alpha_1}=\tau^{\frac12}$, any $p\ge 1,\,\gamma\in(0,1],$ $k=\lceil \tau^{-1-2\lambda}\rceil$ with $\lambda\in(0,\frac{\beta_1}{2}\gamma )$,
 \begin{align}\label{clt_spde1}
\frac{1}{\sqrt{k\tau}}\sum_{i=0}^{k-1}\big(f(Y^{x^h,\Delta}_{t_i})-\mu(f)\big)\tau\overset d\longrightarrow\mathscr{N}(0,v^2)\quad\text{ as }\tau\to0\quad \forall\; f\in\mathcal C_{p,\gamma}.
\end{align}  

As we stated in Remark \ref{rmk3.2}, when we use the weak convergence of the full discretization, Theorems \ref{LLNle} and \ref{CLTle} hold for some class of test functionals.  
It follows  from \cite[Remark 2.6]{dang2023} that (Nc) can be replaced as
\begin{itemize}
\item[(Nc')]
 Remark \ref{rmk3.2} is satisfied with $\sup_{t\ge0}|\mathbb E[f(X^x_{t})-f(Y^{x^h,\Delta}_{t})]|\leq K(\tau^{\beta_1-}+h^{2\beta_1-})$ for $f\in\widetilde{\mathcal C}^2_b.$
\end{itemize}
Thus, we obtain that for $k=\lceil \tau^{-1-2\lambda}\rceil$ with $\lambda\in(0,\beta_1)$,
\begin{align}\label{clt_spde2}
\frac{1}{\sqrt{k\tau}}\sum_{i=0}^{k-1}\big(f(Y^{x^h,\Delta}_{t_i})-\mu(f)\big)\tau\overset d\longrightarrow\mathscr{N}(0,v^2)\quad\text{ as }\tau\to0\quad \forall\; f\in\widetilde {\mathcal C}_{b}^2\cap\mathcal C_{p,\gamma}.
\end{align}  
Here, $\widetilde {\mathcal C}_{b}^2\cap\mathcal C_{p,\gamma}=\widetilde{\mathcal C}_{b}^2$ when $p\ge 2$.
This weaken the condition $f\in\mathcal C^4_b$ in \cite[Theorem 3.1]{dang2023} to $f\in\widetilde {\mathcal C}^2_b.$

Moreover, for infinite dimensional SEEs with superlinearly growing coefficients, for instance the stochastic Allen--Cahn equation, results in \cite{dang2023} are not applicable due to that  coefficients do not satisfy the globally Lipschitz condition.  
By the use of our main results, 
 probabilistic limit behaviors including the strong LLN and the CLT of numerical discretizations of the stochastic Allen--Cahn equation 
can also be obtained. To be specific, 
consider \eqref{ex_spde} with $F$ satisfying  the dissipative condition $\langle X-Y,F(X)-F(Y)\rangle\leq \lambda_F\|X-Y\|^2$ with $\lambda_F<\lambda_1,$ where $\lambda_1$ is the smallest eigenvalue of $-A,$ e.g., $F(X)(\xi)=f(X(\xi))=X(\xi)-X(\xi)^3$. 
Similarly to the proof of \cite[Lemma 2]{cui2021}, we can obtain that $$\sup_{t\ge 0}\mathbb E[\|X^x_t\|^r]\leq K(1+\|x\|^{\tilde r  r})\quad \forall\; r\ge 2$$ with some fixed constant $\tilde r.$
And it follows from the It\^o formula that for any $p\ge 2,$
\begin{align*}
\mathbb E[\|X_t^{x_1}-X_t^{x_2}\|^p]\leq \mathbb E[\|x_1-x_2\|^p]e^{-p(\lambda_1-\lambda_F)t}.
\end{align*}
Hence, we derive that 
\begin{itemize}
\item[(SEa)] Assumption \ref{a1} is satisfied with any $r\ge 1$ and  some fixed constant $\tilde r\ge 1;$
\item[(SEb)] Assumption \ref{a1} is satisfied with $\beta=0,$ $\rho(t)=e^{-ct}$ with some $c>0$, and any $\gamma_1\in(0,1]$.
\end{itemize}
Thus by Proposition \ref{Limit_exact}, the exact solution 
$\{X^x_t\}_{t\ge 0}$ admits a unique invariant measure $\mu$ and fulfills the strong LLN and the CLT.

For the numerical method considered in \cite{cui2021}, i.e., the full discretization with spectral Galerkin method in space   and BEM method in time, 
\begin{align}\label{scheme_spde}
Y^{x^h,\Delta}_{t_{k+1}}=Y^{x^h,\Delta}_{t_k}+A^hY^{x^h,\Delta}_{t_{k+1}}\tau+P^hF(Y^{x^h,\Delta}_{t_{k+1}})\tau+\delta W^h_k,\quad Y^{x^h,\Delta}_0=x^h,
\end{align}
we can also derive $\mathbb E[\|Y^{x_1^{h},\Delta}_{t_{k+1}}-Y^{x_2^{h},\Delta}_{t_{k+1}}\|^p]\leq \mathbb E[\|x_1-x_2\|^p]e^{-p(\lambda_1-\lambda_F)t_{k+1}}$. We claim that the time-independent strong convergence of the full discretization is as follow 
\begin{align}\label{ex_spde_order}
\sup_{t\ge 0}\|X^x_{t}-Y^{x^h,\Delta}_{t}\|_{L^2(\Omega;E)}\leq K(\tau^{\frac{\beta_1}{2}}+h^{\beta_1});
\end{align}
see Appendix \ref{spdeorder} for the proof.  Hence,
\begin{itemize}
\item[(SNa)] Assumption \ref{a2} (\romannumeral1) is satisfied with any $q\ge 1$ and some fixed $\tilde q\ge 1$ (see \cite[Lemma 4]{cui2021} for details);
\item[(SNb)] Assumption \ref{a2} (\romannumeral2) is satisfied with $\kappa=0,$ $\rho^{\Delta}(t)=e^{-ct}$ with some $c>0$, and any $\gamma_2\in(0,1];$
\item[(SNc)] Assumption \ref{a3} is satisfied with $|\Delta^{\alpha}|=\tau^{\frac{\beta_1}{2}}+h^{\beta_1}.$
\end{itemize}
Applying Theorems \ref{LLNle} and \ref{CLTle} gives that the numerical discretization \eqref{scheme_spde} satisfies \eqref{lln_spde1} and \eqref{clt_spde1}, respectively. 

By using the weak error of the numerical discretization \eqref{scheme_spde}, (SNc) can be replaced by
\begin{itemize}
\item[(SNc')] Remark \ref{rmk3.2} is satisfied with $\sup_{t\ge 0}|\mathbb E[f(X^x_{t})-f(Y^{x^h,\Delta}_{t})]|\leq K(\tau^{\beta_1-}+h^{2\beta_1-})$ for $f\in\widetilde{\mathcal C}^2_b;$
\end{itemize}
see \cite[Corollary 1]{cui2021} for details. 
This implies that the numerical discretization \eqref{scheme_spde} satisfies 
 \eqref{clt_spde2} for $k=\lceil \tau^{-1-2\lambda}\rceil$ with $\lambda\in(0,\beta_1).$

\subsubsection{SFDEs}

The probabilistic limit behaviors of numerical discretizations for the functional solution of SFDEs is one of the stochastic systems that motivates this study. This type of equations has been well-studied in the literature; see e.g. \cite{BaoDCDS} and references therein. 
However, to the best of our knowledge, 
 less is known for probabilistic limit behaviors of numerical discretizations of SFDEs.
 Thus, we focus on the probabilistic limit behaviors of the EM method of the SFDE by verifying that  Assumptions \ref{a2}--\ref{a3} are satisfied. Denote by 
$\mathcal C([-\delta_0, 0];~\mathbb{R}^{d})$   the space of all continuous functions
 $\phi(\cdot)$ from $[-\delta_0,0]$ to $\RR^d$ equipped with the  norm $\|\phi\|
=\sup_{-\delta_0\leq\theta\leq0}|\phi(\theta)|$.

The  SFDE on $E=\mathcal C([-\delta_0,0]; \RR^d)$ has the following form
\begin{align}\label{sfde}
\mathrm dX^{x}(t)=b(X^{x}_t)\mathrm dt+\sigma (X^{x}_t)\mathrm dW(t), \quad X^{x}_0= x\in E,
\end{align}
where the delay $\delta_0>0$, $\{W(t),\,t\ge 0\}$ is a $D$-dimensional standard Brownian motion on a filtered probability space $(\Omega,\mathcal F,\{\mathcal F_t\}_{t\ge 0},\mathbb P)$, the initial datum $x(\cdot)$ is H\"older continuous,
the coefficient $b:E \to\mathbb R^d$ is Lipschitz continuous and satisfies the dissipative condition, and  $\sigma: E\to\mathbb R^{d\times D}$ is bounded and  Lipschitz continuous; see Appendix \ref{sfdeproof} for detailed assumptions. 
The functional solution $X^x_t:\theta\mapsto X^x(t+\theta)$ is $E$-valued random variable for $t\geq 0$.
For this equation, 
\begin{itemize}
\item[(Ea)] Assumption \ref{a1} (\romannumeral1) is satisfied with $\tilde r=1$ and any $r\ge 2$ (see \cite[Lemma 3.1]{BaoDCDS} for details);
\item[(Eb)] Assumption \ref{a1} (\romannumeral2) is satisfied with $\beta=0,$  $\rho(t)=e^{-ct}$ for some $c>0$, and any $\gamma_1\in(0,1].$
\end{itemize}
To simplify the representation, we use the result given in  
\cite{BaoNA}, where the approximation of the invariant measure of the EM method of the SFDE
with the Markov chain is investigated. By means of the result for the case without the Markov chain (see \cite[Lemma 2.1]{BaoNA}), we obtain (Eb), whose proof is omitted. 
 Thus, by Proposition \ref{Limit_exact}, the exact functional solution 
$\{X^x_t\}_{t\ge 0}$ admits a unique invariant measure $\mu$ and fulfills the strong LLN and the CLT.

Without loss of generality, we assume
that there exists an integer $N\geq\delta_0$  sufficiently large such that $\tau=\frac{\delta_0}{N}\in(0,1]$.  
Let $t_k=k\tau$ for $k=-N,-N+1,\ldots$ 
The $E$-valued numerical discretization $\{Y^{x^h,\Delta}_{t_k}\}_{k\in\mathbb N}$ based on the EM method associated with \eqref{sfde} is defined by 
\begin{align}
&Y^{x^{h},\Delta}_{t_k}(\theta)=\frac{t_{j+1}-\theta}{\tau}y^{x,\tau}(t_{k+j})
+\frac{\theta-t_{j}}{\tau}y^{x,\tau}(t_{k+j+1}),
~~~\theta\in[t_j,t_{j+1}],~~~j\in\{-N,\ldots,-1\},\label{thL}
\end{align}
where $x^h$ is the linear interpolation of $x(t_{-N}),x(t_{-N+1}),\ldots, x(0)$,
\begin{align}
&y^{x,\tau}(t_{k+1})=y^{x,\tau}(t_k)+b(Y^{x^{h},\Delta}_{t_k})\tau+\sigma(Y^{x^{h},\Delta}_{t_k}) \delta W_k,\quad k\in\mathbb N,\label{EMsfde}
\end{align}
 and $y^{x,\tau}(\theta)=x(\theta)$ for $\theta\in[-\delta_0,0]$.
For the  numerical discretization $\{Y^{x^{h},\Delta}_{t_k}\}_{k\in\mathbb N}$, we  derive that for any $\tau\in(0,\tilde\tau]$ and $q\geq 2$,
\begin{align}\label{sfdebound}
\sup_{k\geq 0}\E[\|Y^{x^{h},\Delta}_{t_{k}}\|^q]\leq K(1+\|x\|^q)
\end{align}
and 
\begin{align}\label{sfdecov}
\sup_{t\geq 0}\mathbb E[\|X^{x}_{t}-Y^{x^{h},\Delta}_{t}\|^2]
\leq K(1+\|x\|^2)\tau;
\end{align}
see Appendix \ref{sfdeproof} for the proofs.
Hence, 
\begin{itemize}
\item[(Na)] Assumption \ref{a2} (\romannumeral1) is satisfied with $\tilde q=1$ and any $q\ge 2$;
\item[(Nb)] Assumption \ref{a2} (\romannumeral2) is satisfied with $\kappa=0$, $\rho^{\Delta}(t)=e^{-ct}$ for some $c>0$, and any $\gamma_2\in(0,1]$ (see \cite[Lemma 4.2]{BaoNA} for the case without the Markov chain);
\item[(Nc)] Assumption \ref{a3} holds with $|\Delta^{\alpha}|=\tau^{\frac{1}{2}}$.
\end{itemize}
It follows from Theorem \ref{LLNle} that the strong LLN holds: for any $p\ge1,\,\gamma\in(0,1],$
\begin{align*}
\lim_{|\Delta|\to 0}\lim_{k\to\infty}\frac1k S^{x^h,\Delta}_k=\mu(f)\quad \text{a.s.}\quad \forall\; f\in\mathcal C_{p,\gamma}.
\end{align*}
Combining Theorem \ref{CLTle}, we have that for any $p\ge 1,\,\gamma\in(0,1],$ $k=\lceil \tau^{-1-2\lambda}\rceil$ with $\lambda\in(0,\frac{1}{2}\gamma)$,
 \begin{align*}
\frac{1}{\sqrt{k\tau}}\sum_{i=0}^{k-1}\big(f(Y^{x^h,\Delta}_{t_i})-\mu(f)\big)\tau\overset d\longrightarrow\mathscr{N}(0,v^2)\quad\text{ as }\tau\to0\quad \forall\; f\in\mathcal C_{p,\gamma}.
\end{align*}  

\section{Proof of the numerical strong LLN}\label{sec_LLN}
This section is devoted to giving the proof of the strong LLN based on the existence and uniqueness as well as the convergence of the numerical invariant measure.

\subsection{Numerical invariant measure}
In this subsection, we show the existence and uniqueness of the numerical invariant measure. Then it is shown that the numerical invariant measure  converges weakly to $\mu$ with a certain order.  
\begin{prop}\label{prop_IM}
 Under Assumption \textup{\ref{a2}}, for 
 each fixed $\Delta,$
$\{Y^{x^h,\Delta}_{t_k}\}_{k\in\mathbb N}$ admits a unique numerical invariant measure $\mu^{\Delta}\in\mathcal P(E^h)$. 
 \end{prop}
\begin{proof}
We split the proof into the following three steps. 

\textit{Step 1: Prove that for each fixed $x^h\in E^h,$ the family $\{\mu_{t_k}^{x^h,\Delta}\}_{k=0}^{\infty}$ is tight, and the  limit is unique  in $(\mathcal P(E^h),\mathbb W_1)$. } 
By virtue of \cite[Lemma 6.14]{Villani}, it suffices to prove that $\{\mu_{t_k}^{x^h,\Delta}\}_{k=0}^{\infty}$ is a Cauchy sequence  in the Polish space $(\mathcal P(E^h), \mathbb{W}_1)$. 
In fact, for any $k,i\in\mathbb N$, applying \cite[Remark 6.5]{Villani} leads to
\begin{align}
&\mathbb{W}_1(\mu_{t_{k+i}}^{x^h,\Delta},\,\mu_{t_k}^{x^h,\Delta})
=\sup_{\|\Psi\|_{Lip}\vee\|\Psi\|_{\infty}\leq1}\Big|\int_{E^h}\Psi(u)\mathrm d\mu_{t_{k+i}}^{x^h,\Delta}(u)-\int_{E^h}\Psi(u)\mathrm d\mu_{t_k}^{x^h,\Delta}(u)\Big|,
\end{align}
where $\Psi:E^h\rightarrow\RR$ is bounded and Lipschitz continuous, and  norms $\|\cdot\|_{Lip}$  and $\|\cdot\|_{\infty}$ are defined by
$\|\Psi\|_{Lip}:=\sup_{\substack{u_1,u_2\in E^h\\u_1\neq u_2}}\frac{|\Psi(u_1)-\Psi(u_2)|}{\|u_1-u_2\|}$ and $\|\Psi\|_{\infty}:=\sup_{u\in E^h}|\Psi(u)|,$ respectively.
By the definition of $\mu_{\cdot}^{x^h,\Delta}$, the time-homogeneous Markov property of $\{Y^{x^h,\Delta}_{t_{k}}\}_{k=0}^{\infty}$, and Assumption \ref{a2},  we deduce that 
\begin{align}
\mathbb{W}_1(\mu_{t_{k+i}}^{x^h,\Delta},\mu_{t_k}^{x^h,\Delta})
=&\sup_{\|\Psi\|_{Lip}\vee\|\Psi\|_{\infty}\leq1}|\E[\Psi(Y^{x^h,\Delta}_{t_{k+i}})]-\E[\Psi(Y^{x^h,\Delta}_{t_{k}})]|\nn\\
=&\sup_{\|\Psi\|_{Lip}\vee\|\Psi\|_{\infty}\leq1}\Big|\E\Big[\E\big[\Psi(Y^{x^h,\Delta}_{t_{k+i}})|\mathcal F_{t_i}\big]\Big]-\E[\Psi(Y^{x^h,\Delta}_{t_{k}})]\Big|\nn\\
=&\sup_{\|\Psi\|_{Lip}\vee\|\Psi\|_{\infty}\leq1}\Big|\E\Big[\E\big[\Psi(Y^{y^h,\Delta}_{t_{k}})-\Psi(Y^{x^h,\Delta}_{t_{k}})\big]\Big|_{y^h=Y^{x^h,\Delta}_{t_{i}}}\Big]\Big|\nn\\
\leq&~K\E\Big[\|Y^{x^h,\Delta}_{t_{i}}-x^h\|\big(1+\|Y^{x^h,\Delta}_{t_{i}}\|^{\kappa}+\|x^h\|^{\kappa}\big)\Big]\rho^{\Delta}(t_k)\nn\\
\leq&~K(1+\E[\|Y^{x^h,\Delta}_{t_{i}}\|^{1+\kappa}]+\|x^h\|^{1+\kappa})\rho^{\Delta}(t_k)\nn\\
\leq&~K(1+\|x^h\|^{\tilde q(1+\kappa)})\rho^{\Delta}(t_k),
\end{align}
which together with $\lim_{k\rightarrow\infty}\rho^{\Delta}(t_k)=0$ finishes the proof of \textit{Step 1}.

\textit{Step 2:  Prove that for each $x^h\in E^h,$ $\{Y^{x^h,\Delta}_{t_k}\}_{k\in\mathbb N}$ admits a unique invariant measure which is denoted by $\mu^{x^h,\Delta}.$}  The uniqueness of the invariant measure  is proved in \textit{Step 1}. For the proof of the existence, based on the Krylov--Bogoliubov theorem 
\cite[Theorem 3.1.1]{DaPrato1996},
it suffices  to show that  $P_{t_k}^{\Delta}$ is Feller, i.e., the mapping
$x^h\rightarrow  P_{t_k}^{\Delta}\Phi(x^h)=\E[\Phi(Y^{x^h,\Delta}_{t_k})]$ 
is bounded and continuous for any $\Phi\in\mathcal C_b$. This can be derived by the use of 
Assumption \ref{a2} (\romannumeral2). 

\textit{Step 3: Prove that for any $x^h\neq y^h,$ $\mu^{x^h,\Delta}=\mu^{y^h,\Delta}(=:\mu^{\Delta}).$} This can be obtained by
\begin{align*}
&\mathbb{W}_1(\mu^{x^h,\Delta},\mu^{y^h,\Delta})\nn\\
\leq&{} \lim_{k\rightarrow\infty} \mathbb{W}_1(\mu^{x^h,\Delta},\mu_{t_{k}}^{x^h,\Delta})+\lim_{k\rightarrow\infty}\mathbb{W}_1(\mu_{t_{k}}^{x^h,\Delta},\mu^{y^h,\Delta}_{t_k})+\lim_{k\rightarrow\infty}\mathbb{W}_1(\mu_{t_{k}}^{y^h,\Delta},\mu^{y^h,\Delta})\nn\\
\leq&{} \lim_{k\rightarrow\infty}\|x^h-y^h\|(1+\|x^h\|^{\kappa}+\|y^h\|^{\kappa})\rho^{\Delta}(t_k)=0.
\end{align*}

Combining \textit{Steps 1--3} finishes the proof. 
\end{proof}

\begin{rem}\label{r2.1}
It follows from Assumption \textup{\ref{a2}}  that 
\begin{align}\mu^{\Delta}\big(\|\cdot\|^q\big)&\leq \lim_{t\rightarrow\infty}\int_{E^h} \|u\|^q\mathrm d\mu^{\textbf 0,\Delta}_t(u)
\leq \sup_{t\geq0}\E [\|Y^{\textbf 0,\Delta}_t\|^q]\leq L_{3}.\label{trmub}
\end{align}
Let $p\in[1,q]$ and $\gamma\in(0,1].$ By \eqref{5.4} and Assumption \textup{\ref{a2}} \textup{(i)}, we obtain that for any $u\in E^h$ and $f\in \mathcal C_{p,\gamma}$,
\begin{align*}
|P^{\Delta}_{t_k}f(u)-\mu^{\Delta}(f)|
\leq&\; K\|f\|_{p,\gamma} (1+\E[\|Y^{u,\Delta}_{t_k}\|^{p}]+\mu^{\Delta}(\|\cdot\|^p))^{\frac{1}{2}}(\mathbb W_{2}(\mu^{u,\Delta}_{t_k},\mu^{\Delta}))^{\gamma}\nn\\
\leq&\; K\|f\|_{p,\gamma} (1+\|u\|^{\frac{p\tilde q}{2}})(\mathbb W_{2}(\mu^{u,\Delta}_{t_k},\mu^{\Delta}))^{\gamma}.
\end{align*}
It follows from $(P^{\Delta}_{t_k})^*\mu^{\Delta}=\mu^{\Delta}$ and  Assumption \textup{\ref{a2}}  that
\begin{align}\label{Wtau}
\mathbb W_{2}(\mu^{u,\Delta}_{t_k},\mu^{\Delta})\leq&~{}\E\Big[\big(\E[\|Y^{u,\Delta}_{t_k}-Y^{y,\Delta}_{t_k}\|^2]\big)^{\frac{1}{2}}\big|_{y=\xi}\Big]\nn\\
\leq&~{}L_4\E\big[\|u-\xi\|(1+\|u\|^{\kappa}+\|\xi\|^{\kappa})\big]\rho^{\Delta}(t_k)\nn\\
\leq&~{}K(1+\|u\|^{1+\kappa})\rho^{\Delta}(t_k),
\end{align}
where $\xi$ is an $E^h$-valued random variable with $\mu^{\Delta}=\PP\circ \xi^{-1}$. Hence, we have 
\begin{align}
|P_{t_k}^{\Delta} f(u)-\mu^{\Delta}(f)|
&\leq K\|f\|_{p,\gamma}(1+\|u\|^{\frac{p\tilde q}{2}+(1+\kappa)\gamma})(\rho^{\Delta}(t_k))^{\gamma},\quad k\in\mathbb N.\label{trmix1}
\end{align}
Similar to the proof of \eqref{trmix1}, we   derive  that for any $u_1,u_2\in E^h,$
\begin{align}
&|P_{t_k}^{\Delta}f(u_1)-P_{t_k}^{\Delta}f(u_2)|\nn\\
&\leq{} K\|f\|_{p,\gamma}\|u_1-u_2\|^{\gamma}(1+\|u_1\|^{\frac{p\tilde q}{2}+\gamma\kappa}+\|u_2\|^{\frac{p\tilde q}{2}+\gamma\kappa})(\rho^{\Delta}(t_k))^{\gamma}.\label{Yattr1}
\end{align}
\end{rem}

\begin{prop}\label{Worder}
Under Assumptions \textup{\ref{a1}--\ref{a3}}, we have 
$
\mathbb W_2(\mu,\mu^{\Delta})\leq K|\Delta^{\a}|.
$
\end{prop}
\begin{proof}
Similar to the proof of \eqref{Wtau}, one deduces  
\begin{align}\label{W}
\mathbb W_{2}(\mu^{x}_t,\mu)
\leq K(1+\|x\|^{1+\beta})\rho(t).
\end{align}
This, along with Assumptions \ref{a1}--\ref{a3} implies that 
\begin{align}\label{th4.2.1}
\mathbb W_2(\mu,\mu^{\Delta})
\leq&~ \mathbb W_2(\mu,\mu^{\bf 0}_{t_k})+\mathbb W_2(\mu^{\bf 0}_{t_k},\mu^{{\bf 0},\Delta}_{t_k})+\mathbb W_2(\mu^{{\bf 0},\Delta}_{t_k},\mu^{\Delta})\nn\\
\leq&~K\rho(t_k)+K\big(\E[\|X^{\bf 0}_{t_{k}}-Y^{{\bf 0},\Delta}_{t_{k}}\|^2]\big)^{\frac{1}{2}}+K\rho^{\Delta}(t_k)\nn\\
\leq&~ K(\rho(t_k)+|\Delta^{\alpha}|+\rho^{\Delta}(t_k)).
\end{align}
Letting $t_k\to\infty$ finishes the proof.
\end{proof}

\subsection{Proof of Theorem \ref{LLNle}} 
In this subsection, we give the proof of the strong LLN of numerical discretizations (i.e., Theorem \ref{LLNle}) based on the 
 following lemma which gives the estimate between  $\frac 1kS^{x^h,\Delta}_k$ and  $\mu^{\Delta}(f)$, whose proof 
is based on  \cite[Proposition 2.6 and Remark 2.7]{Shirikyan} and the strong mixing \eqref{trmix1}.

\begin{lemma}\label{l5.1}
Let Assumption \textup{\ref{a2}} hold, $\gamma\in[\gamma_2,1],$
 and $\varrho\in \mathbb N,\,p
\ge 1$ satisfy 
$\varrho\big(\frac{p}{2}(1+\tilde q)+(1+\kappa)\gamma\big)\leq q.$
Then for any $f\in\mathcal C_{p,\gamma}$, 
$x^h\in E^h$, 
we have
\begin{align}\label{l5.1_eq}
\E \big[\big|\frac{1}{k}S_{k}^{x^h,\Delta}-\mu^{\Delta}(f)\big|^{2\varrho}\big]\leq K\|f\|^{2\varrho}_{p,\gamma}(1+\|x^h\|^{\varrho\tilde q(\frac{p}{2}(1+\tilde q)+(1+\kappa)\gamma)})t_{k}^{-\varrho} \quad \forall \;k\in \mathbb N_+.
\end{align}
\end{lemma}

To proceed, we present  the proof of the strong LLN, i.e., $\frac1k S^{x^h,\Delta}_k\overset{a.s.}\longrightarrow\mu(f)$ as $k\to\infty$ and $|\Delta|\to0.$ 

\begin{proof}[\textbf{Proof of Theorem \ref{LLNle}}.]
For $f\in\mathcal C_{p,\gamma}$, 
applying  \eqref{5.4} and Proposition  \ref{Worder} leads to
\begin{align}\label{t3.1}
 |\mu(f)-\mu^{\Delta}(f)|\leq K\|f\|_{p,\gamma}|\Delta^{\a}|^{\gamma}.
\end{align}
Hence it suffices to prove that   $$\lim_{k\to\infty}\Big|\frac{1}{k}S_{k}^{x^{h},\Delta}-\mu^{\Delta}(f)\Big|=0\quad \text{a.s.}$$  for each fixed step-size $\Delta$ and $f\in\mathcal C_{p,\gamma}$.

For any $\delta\in(0,\frac14)$ and $k\in\mathbb N$ with $k\geq \lceil1/\tau\rceil$, define
\begin{align*}
\mathcal A_{k}^{x^{h}, \Delta}:=\Big\{\omega\in\Omega:\big|\frac{1}{k}S_{k}^{x^{h},\Delta}-\mu^{\Delta}(f)\big|(\omega)>\|f\|_{p,\gamma}t_{k}^{-\delta}\tau^{-\frac{1}{4}}\Big\}.
\end{align*}
Note that the condition $p(1+\tilde q)+2(1+\kappa)\gamma\leq q$ coincides with the one in Lemma \ref{l5.1} with $\varrho=2$. By virtue of  \eqref{l5.1_eq} with $\varrho=2$ one has
$$\E\Big[\big|\frac{1}{k}S_{k}^{x^{h},\Delta}-\mu^{\Delta}(f)\big|^{4}\Big]\leq{}K\|f\|^4_{p,\gamma}(1+\|x^h\|^{\tilde q(p(1+\tilde q)+2(1+\kappa)\gamma)})t_{k}^{-2}.$$
Applying the Chebyshev inequality  yields
\begin{align*}
\PP(\mathcal A_{k}^{x^{h},\Delta})\leq&~\frac{t_{k}^{4\delta}\tau}{\|f\|^{4}_{p,\gamma}} \E\Big[|\frac{1}{k}S_{k}^{x^{h},\Delta}-\mu^{\Delta}(f)|^{4}\Big]\nn\\
\leq&~ K(1+\|x^h\|^{\tilde q(p(1+\tilde q)+2(1+\kappa)\gamma)})t_{k}^{2(2\delta-1)}\tau.
\end{align*}
Then it follows from $2(1-2\delta) >1$ that
\begin{align}\label{th5.1.2}
\sum_{k=\lceil\frac{1}{\tau}\rceil}^{\infty}\PP(\mathcal A_{k}^{x^{h},\Delta})\leq&{} K(1+\|x^h\|^{\tilde q(p(1+\tilde q)+2(1+\kappa)\gamma)})\sum_{k=\lceil\frac{1}{\tau}\rceil}^{\infty}t_{k}^{2(2\delta-1)}\tau\nn\\
\leq &{}K(1+\|x^h\|^{\tilde q(p(1+\tilde q)+2(1+\kappa)\gamma)} ).
\end{align}
Define the 
random variable $ \mathcal K^{x^h,\Delta}$ by
\begin{align*}
\mathcal K^{x^h,\Delta}(\omega):=\inf\Big\{j\in \mathbb N ~\mbox{with}~ j\geq \lceil\frac{1}{\tau}\rceil:  \big|\frac{1}{k}S_{k}^{x^h,\Delta}-\mu^{\Delta}(f)\big|(\omega)\leq \|f\|_{p,\gamma}t_{k}^{-\delta}\tau^{-\frac{1}{4}}~~\forall  k\geq j+1\Big\}.
\end{align*}
Inequality \eqref{th5.1.2}, along with the Borel--Cantelli lemma 
implies that $\mathcal K^{x^h,\Delta}<\infty$ a.s. and
\begin{align}\label{th5.1.5}
\PP\Big\{\omega\in \Omega: \big|\frac{1}{k}S_{k}^{x^h,\Delta}-\mu^{\Delta}(f)\big|(\omega)\leq \|f\|_{p,\gamma}t_{k}^{-\delta}\tau^{-\frac{1}{4}} 
~\mbox{for all} ~k\geq \mathcal K^{x^h,\Delta}(\omega)+1\Big\}=1.
\end{align}
Let $T^{x^h,\Delta}=\tau\mathcal K^{x^h,\Delta}$. For the constant $l>0$ with $l+2(2\delta-1)<-1$,  we derive 
\begin{align}\label{th5.1.3}
\E[(T^{x^h,\Delta})^{l}]
=&{}\sum_{j=\lceil\frac{1}{\tau}\rceil}^{\infty}\E\big[(T^{x^h,\Delta})^{l}\textbf 1_{\{\mathcal K^{x^h,\Delta}=j\}}\big]=\sum_{j=\lceil\frac{1}{\tau}\rceil}^{\infty}(j\tau)^{l}\PP\big(\mathcal K^{x^h,\Delta}=j\big)\nn\\
\leq&~\sum_{j=\lceil\frac{1}{\tau}\rceil}^{\infty}(t_{j})^{l}\PP(\mathcal A_{j}^{x^h,\Delta})
\leq~K(1+\|x^h\|^{\tilde q(p(1+\tilde q)+2(1+\kappa)\gamma)}).
\end{align}
For $\tilde{\delta}>0$ and $N\in\mathbb N_+,$  define
$
\mathcal D_{N}^{x^h}:=\big\{\omega\in\Omega: T^{x^h,\Delta}+1>N^{\tilde \delta}\big\}.
$
When $\tilde{\delta}$ is chosen so that 
$\tilde \delta l>1$, it follows from \eqref{th5.1.3} that 
\begin{align*}
\sum_{N=1}^{\infty}\PP(\mathcal D_{N}^{x^h})
\leq&{}\sum_{N=1}^{\infty}(\frac{1}{N})^{\tilde \delta l}\E[(T^{x^h,\Delta}+1)^{l}]
\leq{}K(1+\|x^h\|^{\tilde q(p(1+\tilde q)+2(1+\kappa)\gamma)})\sum_{N=1}^{\infty}\frac{1}{N^{\tilde \delta l}}\\
\leq&\; K(1+\|x^h\|^{\tilde q(p(1+\tilde q)+2(1+\kappa)\gamma)}).
\end{align*} 
Using the Borel--Cantelli lemma again yields that there exists a random variable $\mathfrak N^{x^h}<\infty$ a.s. such that
\begin{align}\label{th5.1.4}
\PP\Big\{\omega\in\Omega:T^{x^h,\Delta}(\omega)+1\leq N^{\tilde \delta} ~\mbox{for all} ~N\geq \mathfrak N^{x^h}(\omega)\Big\}=1.
\end{align}
Combining \eqref{th5.1.5} and  \eqref{th5.1.4} 
yields that for a.s. $\omega\in\Omega,$ when $k\ge \frac{(\mathfrak N^{x^h}(\omega))^{\tilde{\delta}}}{\tau^{1+\frac{1}{4\delta}}}\ge \frac{T^{x^h,\Delta}(\omega)+1}{\tau}\ge \mathcal K^{x^h,\Delta}(\omega)+1,$
 we have
\begin{align*}
\Big|\frac{1}{k}S_{k}^{x^h,\Delta}-\mu^{\Delta}(f)\Big|(\omega)\leq\|f\|_{p,\gamma} t_{k}^{-\delta}\tau^{-\frac{1}{4}} 
\leq \|f\|_{p,\gamma} (\mathfrak N^{x^h}(\omega))^{-\tilde \delta\delta},
\end{align*}
which  implies $\lim_{k\to\infty}\frac1kS_{k}^{x^h,\Delta}=\mu^{\Delta}(f)$ a.s. for fixed  $\Delta.$ 
 This finishes the proof.
\end{proof}
\section{Proof of the numerical CLT}\label{sec_CLT}
 In this section, we present the proof of the CLT  of  numerical discretizations (i.e., Theorem \ref{CLTle}), based on the decomposition of the normalized time-averaging estimator.
 

\subsection{Some properties of the discrete martingale}\label{property_mar}
In this subsection, we study properties of the discrete martingale $\{\mathscr M^{x^h,\Delta}_{k}\}_{k\in\mathbb N}$ extracted from the normalized time-averaging estimator. Namely, for $k\in\mathbb N,$ define 
\begin{align}\label{mar_eq}
\mathscr M^{x^h,\Delta}_{k}:=&\; \tau\sum_{i=0}^{\infty}\Big(\E\big[f(Y_{t_i}^{x^h,\Delta})|\mathcal F_{t_{k}}\big]-\mu^{\Delta}(f)\Big)-\tau\sum_{i=0}^{\infty}\Big(\E\big[f(Y_{t_i}^{x^h,\Delta})|\mathcal F_{0}\big]-\mu^{\Delta}(f)\Big).
\end{align}
By fully utilizing the Markov property and the strong mixing of numerical discretizations of Markov processes, it can be shown in the following  proposition that $\mathscr M^{x^h,\Delta}_{k}$ is well-defined and is a martingale for $k\in\mathbb N$. 
The proof is  presented in  Section \ref{proof_prop}.
\begin{prop}\label{Prop4.1}
Let Assumption \textup{\ref{a2}} hold, $\gamma\in[\gamma_2,1]$, and  
 $p\ge 1$ satisfy $\frac{p\tilde q}{2}+(1+\kappa)\gamma\leq q$.
Then for any $x^h\in E^h$, $f\in\mathcal C_{p,\gamma}$,  the sequence $\{\mathscr M^{x^h,\Delta}_{k}\}_{k\in\mathbb N}$ is an $\{\mathcal F_{t_k}\}_{k\in \mathbb N}$-adapted martingale with  $\mathscr M^{x^h,\Delta}_{0}=0$. Moreover, we have
\begin{align}\label{p6.1.4}
|\mathscr M^{x^h,\Delta}_{k}|&\leq\tau\Big|\sum_{i=0}^{k-1}\big(f(Y_{t_i}^{x^h,\Delta})-\mu^{\Delta}(f)\big)\Big|\nn\\
&\quad +K\|f\|_{p,\gamma}(1+\|Y_{t_k}^{x^h,\Delta}\|^{\frac{p\tilde q}{2}+(1+\kappa)\gamma}+\|x^h\|^{\frac{p\tilde q}{2}+(1+\kappa)\gamma}).
\end{align}
\end{prop}

Define  the martingale difference sequence   of $\{\mathscr M^{x^h,\Delta}_{k}\}_{k\in\mathbb N}$ by
\begin{align*}
\mathcal Z^{x^h,\Delta}_{k}:=\mathscr M^{x^h,\Delta}_{k}-\mathscr M^{x^h,\Delta}_{k-1}~~~\forall\; k\in\mathbb N_+,
~~~~\mathcal Z^{x^h,\Delta}_{0}=0.
\end{align*}
We call $\E\big[|\mathcal Z^{x^h,\Delta}_{k}|^2|\mathcal F_{t_{k-1}}\big]$  the conditional variance of the martingale difference sequence, whose    properties are presented in 
 following propositions.
 The proofs are postponed to Section \ref{proof_prop}.
\begin{prop}\label{p6.1}
Let Assumption \textup{\ref{a2}} hold and $\gamma\in[\gamma_2,1]$. 
\begin{itemize}
\item[(\romannumeral1)] 
For any constant $c$ satisfying 
$c\big(\frac{p\tilde q}{2}+(1+\kappa)\gamma\big)\leq q$, it holds that
\begin{align}\label{Zbound}
\sup_{h\in[0,\tilde h], \tau\in(0,\tilde \tau]}\sup_{k\geq 0}\E[|\mathcal Z^{x^h,\Delta}_{k}|^{c}]
\leq&~{}K\|f\|^c_{p,\gamma}(1+\|x^h\|^{c \tilde q(\frac{p\tilde q}{2}+(1+\kappa)\gamma)}).
\end{align}
 \item[(\romannumeral2)] Let $p\ge 1$ satisfy $p\tilde q+2(1+\kappa)\gamma\leq q.$ 
Then $H^{\Delta}(Y^{x^h,\Delta}_{t_{k}})=\mathbb E[|\mathcal Z^{x^h,\Delta}_{k+1}|^2|\mathcal F_{t_{k}}]$ for all $k\in\mathbb N$, where  $H^{\Delta}: E^h\to\mathbb R$ defined as
\begin{align}\label{l6.15+}
H^{\Delta}(u)=&-\tau^2|f(u)-\mu^{\Delta}(f)|^2+\tau^2 \E\Big[\Big|\sum_{i=0}^{\infty}\big(P^{\Delta}_{t_i}f(Y^{u,\Delta}_{t_1})-\mu^{\Delta}(f)\big)\Big|^2\Big]\nn\\
&-\tau^2 |\sum_{i=0}^{\infty}\big(P^{\Delta}_{t_i}f(u)-\mu^{\Delta}(f)\big)|^2+2\tau^2\big(f(u)-\mu^{\Delta}(f)\big) \sum_{i=0}^{\infty}\big(P^{\Delta}_{t_i}f(u)-\mu^{\Delta}(f)\big).
\end{align}
\item[(\romannumeral3)] Let 
  $p\geq1 $ satisfy
$2p\tilde q+2(1+2\kappa)\gamma\leq q.$
 Then 
$
H^{\Delta}\in\mathcal C_{2\tilde p_{\gamma},\gamma}$ and $\|H^{\Delta}\|_{2\tilde p_{\gamma},\gamma}\leq K \|f\|_{p,\gamma}^2,
$
where  $\tilde p_{\gamma}:=\tilde q\big(p\tilde q+(2+3\kappa)\gamma\big).$
\end{itemize}
\end{prop}

\begin{prop}\label{p6.2}
Let Assumptions \textup{\ref{a1}--\ref{a3}} hold, $\gamma\in[\gamma_1\vee \gamma_2,1]$,
and $p\ge 1$ satisfy 
$
p(1+\tilde q\vee\tilde r)+2(1+\kappa\vee\beta)\gamma\leq q\wedge r.
$
Then for any $f\in\mathcal C_{p,\gamma}$, 
\begin{align}\label{p6.2.1}
\mu^{\Delta}(H^{\Delta})=-\tau^{2}\mu^{\Delta}\big(|f-\mu^{\Delta}(f)|^2\big)+2\tau\mu^{\Delta}\Big(\big(f-\mu^{\Delta}(f)\big)\sum_{i=0}^{\infty}\big(P^{\Delta}_{t_i}f-\mu^{\Delta}(f)\big)\tau\Big).
\end{align}
Moreover, $|\mu^{\Delta}(H^{\Delta})|\leq K\tau \|f\|_{p,\gamma}^2$ and
$
\lim_{|\Delta|\rightarrow 0}\frac{\mu^{\Delta}(H^{\Delta})}{\tau}=v^2.
$
\end{prop}

\subsection{Proof of Theorem \ref{CLTle}}\label{sub5.2}
The proof of Theorem \ref{CLTle} is based on the decomposition of the normalized time-averaging estimator $\frac{1}{\sqrt {k\tau}} \sum_{i=0}^{k-1}\big(f(Y^{x^h,\Delta}_{t_i})-\mu(f)\big)\tau$ into a martingale term and a negligible term. Namely, 
\begin{align}\label{6.1}
\frac{1}{\sqrt {k\tau}}\sum_{i=0}^{k-1}\big(f(Y^{x^h,\Delta}_{t_i})-\mu(f)\big)\tau =\frac{1}{\sqrt {k\tau}}\mathscr M^{x^h,\Delta}_{k}+\frac{1}{\sqrt {k\tau}}\mathscr R^{x^h,\Delta}_{k},
\end{align}
where $\mathscr M^{x^h,\Delta}_{k}$ is given in \eqref{mar_eq} and 
\begin{align*}
\mathscr R^{x^h,\Delta}_{k}:=&-\tau\sum_{i=k}^{\infty}\Big(\E\big[f(Y_{t_i}^{x^h,\Delta})|\mathcal F_{t_{k}}\big]-\mu^{\Delta}(f)\Big)+\tau\sum_{i=0}^{\infty}\Big(\E\big[f(Y_{t_i}^{x^h,\Delta})|\mathcal F_{0}\big]-\mu^{\Delta}(f)\Big)\\
&+k\tau(\mu^{\Delta}(f)-\mu(f)).
\end{align*}
By virtue of the Slutsky theorem, the idea for proving Theorem \ref{CLTle} is to prove that $\frac{1}{\sqrt{k\tau}}\mathscr M^{x^h,\Delta}_{k}\overset{d}\longrightarrow\mathscr N(0,v^2)$ and $\frac{1}{\sqrt{k\tau}}\mathscr R^{x^h,\Delta}_{k}\overset{\mathbb P}\longrightarrow0.$ 

\begin{proof}
\textit{Step 1: Prove that
$\frac{1}{\sqrt{k\tau}}\mathscr M^{x^h,\Delta}_{k}$ converges in distribution to $\mathscr N(0,v^2) $ as $|\Delta|\rightarrow 0$.} Recalling  $k=\lceil\tau^{-2\l-1}\rceil$, 
 it is equivalent to show that $\tau^{\l}\mathscr M^{x^h,\Delta}_{\lceil\tau^{-2\l-1}\rceil}$ converges in distribution to $\mathscr N(0,v^2) $. To this end, we show that 
 the characteristic function of $\tau^{\l}\mathscr M^{x^h,\Delta}_{\lceil\tau^{-2\l-1}\rceil}$ satisfies
\begin{align}\label{claimM}
\lim_{|\Delta|\rightarrow 0}\E \big[\exp{\big({\bf i}\iota\tau^{\l}\mathscr M^{x^h,\Delta}_{\lceil\tau^{-2\l-1}\rceil}\big)}\big]=e^{-\frac{v^2\iota^2}{2}}\quad\forall\; \iota \in\RR,
\end{align}
where $\bf i$ is the imaginary unit. 
Without loss of generality, we assume that $\mu(f)=0$. Otherwise, we let $\tilde f=f-\mu(f)$ and consider $\tilde f$ instead of $f$.

A direct calculation gives
\begin{align}\label{th6.1.1}
&e^{\frac{v^2\iota^2}{2}}\E\big[\exp{({\bf i} \iota\tau^{\l}\mathscr {M}^{x^h,\Delta}_{k})}\big]-1\nn\\
=&{}\sum_{j=0}^{k-1}\Big(e^{\frac{v^2\iota^2(j+1)}{2k}}\E\big[\exp{({\bf i} \iota\tau^{\l}\mathscr {M}^{x^h,\Delta}_{j+1})}\big]-e^{\frac{v^2\iota^2j}{2k}}\E\big[\exp{({\bf i} \iota\tau^{\l}\mathscr {M}^{x^h,\Delta}_{j})}\big]\Big)\nn\\
=&{}\sum_{j=0}^{k-1}e^{\frac{v^2\iota^2(j+1)}{2k}}\E\Big[\exp{({\bf i} \iota\tau^{\l}\mathscr {M}^{x^h,\Delta}_{j})}\big(\exp{({\bf i} \iota\tau^{\l}\mathcal {Z}^{x^h,\Delta}_{j+1})}-1\big)\Big]\nn\\
&+\sum_{j=0}^{k-1}e^{\frac{v^2\iota^2(j+1)}{2k}}(1-e^{-\frac{v^2\iota^2}{2k}})\E\big[\exp{({\bf i} \iota\tau^{\l}\mathscr {M}^{x^h,\Delta}_{j})}\big]\nn\\
=&{}\sum_{j=0}^{k-1}e^{\frac{v^2\iota^2(j+1)}{2k}}\E\Big[\exp{({\bf i} \iota\tau^{\l}\mathscr {M}^{x^h,\Delta}_{j})}\big(\exp{({\bf i} \iota\tau^{\l}\mathcal {Z}^{x^h,\Delta}_{j+1})}-1+\frac{v^2\iota^2}{2k}\big)\Big]\nn\\
&+\sum_{j=0}^{k-1}e^{\frac{v^2\iota^2(j+1)}{2k}}(1-e^{-\frac{v^2\iota^2}{2k}}-\frac{v^2\iota^2}{2k})\E\big[\exp{({\bf i} \iota\tau^{\l}\mathscr {M}^{x^h,\Delta}_{j})}\big].
\end{align}
Note that for any $\zeta\in\RR\setminus\{0\}$,
$
e^{{\bf i}\zeta}=1+{\bf i}\zeta-\frac{\zeta^2}{2}-\zeta^2Q(\zeta),
$
where
$
Q(\zeta)=\zeta^{-2}\int_{0}^{\zeta}\int_{0}^{\zeta_2}(e^{{\bf i}\zeta_1}-1)\mathrm d\zeta_1\mathrm d\zeta_2
$
and it satisfies 
$
\sup_{\zeta\in\mathbb R}|Q(\zeta)|\leq 1,|Q(\zeta)|=\mathcal O(|\zeta|) \text{ for } \zeta\ll1.
$
This, along with $\E[\mathcal Z^{x^h,\Delta}_{j+1}|\mathcal F_{t_{j}}]=0$ implies that
\begin{align}\label{th6.1.2}
&\E\Big[\exp{({\bf i} \iota\tau^{\l}\mathscr {M}^{x^h,\Delta}_{j})}\big(\exp{({\bf i} \iota\tau^{\l}\mathcal {Z}^{x^h,\Delta}_{j+1})}-1+\frac{v^2\iota^2}{2k}\big)\Big]\nn\\
=&~\E\Big[\exp{({\bf i} \iota\tau^{\l}\mathscr {M}^{x^h,\Delta}_{j})}\E\Big[\Big(-\frac{\iota^2\tau^{2\l}(\mathcal Z^{x^h,\Delta}_{j+1})^2}{2}-\iota^2\tau^{2\l}(\mathcal Z^{x^h,\Delta}_{j+1})^2Q_{j+1}+\frac{v^2\iota^2}{2k}\Big)\Big|\mathcal F_{t_{j}}\Big]\Big]\nn\\
=&~\frac{\iota^2\tau^{2\l}}{2}\E\Big[\exp{({\bf i} \iota\tau^{\l}\mathscr {M}^{x^h,\Delta}_{j})}\big(\frac{v^2}{k\tau^{2\l}}-(\mathcal Z^{x^h,\Delta}_{j+1})^2\big)\Big]\nn\\
&-\iota^2\tau^{2\l}\E\Big[\exp{({\bf i} \iota\tau^{\l}\mathscr {M}^{x^h,\Delta}_{j})}(\mathcal Z^{x^h,\Delta}_{j+1})^2Q_{j+1}\Big],
\end{align}
where 
$Q_{j+1}:=Q(\iota\tau^{\l}\mathcal {Z}^{x^h,\Delta}_{j+1})$.  Plugging \eqref{th6.1.2} into \eqref{th6.1.1} yields
\begin{align}\label{th6.1.3}
e^{\frac{v^2\iota^2}{2}}\E[\exp{({\bf i} \iota\tau^{\l}\mathscr {M}^{x^h,\Delta}_{k})}]-1
= I_1(k)+I_2(k)+I_3(k),
\end{align}
where
\begin{align*}
I_1(k):=&{}\sum_{j=0}^{k-1}e^{\frac{v^2\iota^2(j+1)}{2k}}(1-e^{-\frac{v^2\iota^2}{2k}}-\frac{v^2\iota^2}{2k})\E[\exp{({\bf i} \iota\tau^{\l}\mathscr {M}^{x^h,\Delta}_{j})}],\nn\\
I_2(k):=&{}-\iota^2\tau^{2\l}\sum_{j=0}^{k-1}e^{\frac{v^2\iota^2(j+1)}{2k}}\E\Big[\exp{({\bf i} \iota\tau^{\l}\mathscr {M}^{x^h,\Delta}_{j})}(\mathcal Z^{x^h,\Delta}_{j+1})^2Q_{j+1}\Big],\nn\\
I_3(k):=&{}\frac{\iota^2\tau^{2\l}}{2}\sum_{j=0}^{k-1}e^{\frac{v^2\iota^2(j+1)}{2k}}\E\Big[\exp{({\bf i} \iota\tau^{\l}\mathscr {M}^{x^h,\Delta}_{j})}\big(\frac{v^2}{k\tau^{2\l}}-(\mathcal Z^{x^h,\Delta}_{j+1})^2\big)\Big].
\end{align*}
\textbf{Estimate of the term $I_{1}(k)$.} Applying  the inequality $|1-e^{-\zeta}-\zeta|\leq \frac12\zeta^2$ for $\zeta>0$ gives
\begin{align}\label{th6.1.4}
\lim_{\tau\rightarrow0}|I_1(k)|\leq&{}\lim_{\tau\rightarrow0} e^{\frac{v^2\iota^2}{2}}\sum_{j=0}^{k-1}|1-e^{-\frac{v^2\iota^2}{2k}}-\frac{v^2\iota^2}{2k}|\nn\\
\leq&{} \lim_{\tau\rightarrow0}e^{\frac{v^2\iota^2}{2}}\sum_{j=0}^{k-1}\frac{v^4\iota^4}{8k^2}
\leq \lim_{\tau\rightarrow0}\frac{K}{k}=0.
\end{align}
\textbf{Estimate of  the term $I_{2}(k)$.} It follows from $\sup_{\zeta\in\RR}|Q(\zeta)|\leq 1$ that for any $\epsilon>0$,
\begin{align}\label{th6.1.5}
I_2(k)\leq\iota^2e^{\frac{v^2\iota^2}{2}}\tau^{2\l} \sum_{j=0}^{k-1}\E\big[\mathcal (\mathcal Z^{x^h,\Delta}_{j+1})^2|Q_{j+1}|\big]
\leq I_{2,1}(k)+I_{2,2}(k),
\end{align}
where
\begin{align*}
I_{2,1}(k):=&~\iota^2 e^{\frac{v^2\iota^2}{2}}\tau^{2\l}\sum_{j=0}^{k-1}\E\big[\mathcal (\mathcal Z^{x^h,\Delta}_{j+1})^2\textbf{1}_{\{|\mathcal Z^{x^h,\Delta}_{j+1}|>\epsilon\tau^{-\l}\}}\big],\nn\\
I_{2,2}(k):=&~\iota^2 e^{\frac{v^2\iota^2}{2}}\tau^{2\l}\sum_{j=0}^{k-1}\E\big[\mathcal (\mathcal Z^{x^h,\Delta}_{j+1})^2|Q_{j+1}|\textbf{1}_{\{|\mathcal Z^{x^h,\Delta}_{j+1}|\leq\epsilon\tau^{-\l}\}}\big].
\end{align*}
Using 
 Propositions \ref{p6.1}--\ref{p6.2} leads to  that for any $\e\in(0,\l)$,
\begin{align}\label{th6.1.7+}
I_{2,2}(k)
\leq&~\iota^2 e^{\frac{v^2\iota^2}{2}}\tau^{2\l}\sup_{|\zeta|\leq\iota\epsilon}|Q(\zeta)|\sum_{j=0}^{k-1}\E[(\mathcal Z^{x^h,\Delta}_{j+1})^2]\nn\\
\leq &~\iota^2 e^{\frac{v^2\iota^2}{2}}\sup_{|\zeta|\leq\iota\epsilon}|Q(\zeta)| \tau^{2\l}\Big[\sum_{j=0}^{\lceil\tau^{-1-(\l-\e)}\rceil-1}\E[(\mathcal Z^{x^h,\Delta}_{j+1})^2]\nn\\
&+\sum_{j=\lceil\tau^{-1-(\l-\e)}\rceil}^{k-1}\Big(\E\big[H^{\Delta}(Y^{x^h,\Delta}_{t_j})\big]-\mu^{\Delta}(H^{\Delta})\Big)+k|\mu^{\Delta}(H^{\Delta})|\Big]\nn\\
\leq &~\iota^2 e^{\frac{v^2\iota^2}{2}}\sup_{|\zeta|\leq\iota\epsilon}|Q(\zeta)| \Big[\tau^{2\l}\sum_{j=0}^{\lceil\tau^{-1-(\l-\e)}\rceil-1}\E[(\mathcal Z^{x^h,\Delta}_{j+1})^2]\nn\\
&+\tau^{2\l}\sum_{j=\lceil\tau^{-1-(\l-\e)}\rceil}^{k-1}\big(P^{\Delta}_{t_j}H^{\Delta}(x^h)-\mu^{\Delta}(H^{\Delta})\big)+K\|f\|_{p,\gamma}^2\Big].
\end{align}
Combining $\E\big[|\mathscr M^{x^h,\Delta}_{\lceil\tau^{-1-(\l-\e)}\rceil}|^2\big]
=\sum_{j=0}^{\lceil\tau^{-1-(\l-\e)}\rceil-1}\E[(\mathcal {Z}^{x^h,\Delta}_{j+1})^2]$ 
and  \eqref{p6.1.4} implies that
\begin{align*}
&\sum_{j=0}^{\lceil\tau^{-1-(\l-\e)}\rceil-1}\E[(\mathcal {Z}^{x^h,\Delta}_{j+1})^2]\nn\\
\leq{}&K\tau^{1-(\l-\e)}\sum_{j=0}^{\lceil\tau^{-1-(\l-\e)}\rceil-1}\E\big[|f(Y_{t_j}^{x^h,\Delta})|^2+|\mu^{\Delta}(f)|^2\big]\nn\\
&+K\|f\|^2_{p,\gamma}\E\Big[1+\|Y^{x^h,\Delta}_{t_{\lceil\tau^{-1-(\l-\e)}\rceil}}\|^{p\tilde q+2(1+\kappa)\gamma}+\|x^h\|^{p\tilde q+2(1+\kappa)\gamma}\Big].
\end{align*}
Since the condition  \eqref{condip} implies that
$p\tilde q+2(1+\kappa)\gamma\leq r$, by Assumption \ref{a2} (i) we obtain
 \begin{align}\label{th6.1.14+}
\sum_{j=0}^{\lceil\tau^{-1-(\l-\e)}\rceil-1}\E[(\mathcal {Z}^{x^h,\Delta}_{j+1})^2]\leq K\|f\|^2_{p,\gamma}(1+\|x^h\|^{\tilde q(p\tilde q+2(1+\kappa)\gamma)})\tau^{-2(\l-\e)}.
\end{align}
It follows from \eqref{condip} and Proposition \ref{p6.1} (\romannumeral3) that $H^{\Delta}\in\mathcal C_{2\tilde p_{\gamma},\gamma}$ with $2\tilde p_{\gamma}\leq q.$
Making use of \eqref{trmix1} yields 
\begin{align}\label{th6.1.15+}
&\tau^{2\lambda}\sum_{j=\lceil\tau^{-1-(\l-\e)}\rceil}^{k-1}\big(P^{\Delta}_{t_j}H^{\Delta}(x^h)-\mu^{\Delta}(H^{\Delta})\big)\nn\\
\leq&~K \|H^{\Delta}\|_{2\tilde p_{\gamma},\gamma}(1+\|x^h\|^{\tilde q\tilde p_{\gamma}+(1+\kappa)\gamma})\tau^{2\l}\sum_{j=\lceil\tau^{-1-(\l-\e)}\rceil}^{k-1}\big(\rho^{\Delta}(t_j)\big)^{\gamma}\nn\\
\leq&~K \|f\|_{p,\gamma}^2(1+\|x^h\|^{\tilde q\tilde p_{\gamma}+(1+\kappa)\gamma})\tau^{2\l}\sum_{j=\lceil\tau^{-1-(\l-\e)}\rceil}^{k-1}\big(\rho^{\Delta}(t_j)\big)^{\gamma}.
\end{align}
Inserting \eqref{th6.1.14+} and \eqref{th6.1.15+} into \eqref{th6.1.7+}   leads to
\begin{align}\label{I22+}
I_{2,2}(k)\leq&~ K\iota^2 e^{\frac{v^2\iota^2}{2}}\sup_{|\zeta|\leq\iota\epsilon}|Q(\zeta)|\|f\|_{p,\gamma}^2\Big(1+{\|x^h\|^{\tilde q(p\tilde q+2(1+\kappa)\gamma)}+\|x^h\|^{\tilde q\tilde p_{\gamma}+(1+\kappa)\gamma}}\Big)\nn\\
&~\times\Big(\tau^{2\e}+\tau^{2\l}\sum_{j=\lceil\tau^{-1-(\l-\e)}\rceil}^{k-1}\big(\rho^{\Delta}(t_j)\big)^{\gamma}+1\Big).
\end{align}
Note that \eqref{rho_remainder} implies 
$
\lim_{\e\rightarrow 0}\lim_{\tau\rightarrow 0}\tau^{2\l}\sum_{j=\lceil\tau^{-1-(\l-\e)}\rceil}^{k-1}\big(\rho^{\Delta}(t_j)\big)^{\gamma}=0.
$
This, along with \eqref{I22+} leads to
\begin{align}\label{th6.1.9}
\lim_{\epsilon\rightarrow 0}\lim_{\e\rightarrow 0}\lim_{\tau\rightarrow 0}I_{2,2}(k)=0.
\end{align}
By virtue of the condition \eqref{condip} there exists a positive constant $c$ with $c>\frac{2}{\l}\vee4$ such that 
$c\big(\frac{p\tilde q}{2}+(1+\frac{\kappa}{2})\gamma\big)\leq q$. 
Applying the H\"older inequality, the Chebyshev inequality, and \eqref{Zbound},  we have 
\begin{align}
I_{2,1}(k)\leq&~ \iota^2 e^{\frac{v^2\iota^2}{2}}\tau^{2\l}\sum_{j=0}^{k-1}\big(\E[|\mathcal Z^{x^h,\Delta}_{j+1}|^4]\big)^{\frac{1}{2}}\big(\PP\{|\mathcal Z^{x^h,\Delta}_{j+1}|> \epsilon\tau^{-\l}\}\big)^{\frac{1}{2}}\nn\\
\leq&~\iota^2 e^{\frac{v^2\iota^2}{2}}\frac{1}{\epsilon^{c/2}}\tau^{2\l+\frac{c\l}{2}}\sum_{j=0}^{k-1}\big(\E[|\mathcal Z^{x^h,\Delta}_{j+1}|^{4}]\big)^{\frac{1}{2}}\big(\E[|\mathcal Z^{x^h,\Delta}_{j+1}|^{c}]\big)^{\frac{1}{2}}\nn\\
\leq&~K\iota^2 e^{\frac{v^2\iota^2}{2}}\frac{1}{\epsilon^{c/2}}\tau^{2\l+\frac{c\l}{2}}k\|f\|^{\frac{4+c}{2}}_{p,\gamma}(1+\|x^h\|^{\frac{(4+c)\tilde q}{2}\left(\frac{p\tilde q}{2}+(1+\kappa)\gamma\right)})\nn\\
\leq&~K\iota^2 e^{\frac{v^2\iota^2}{2}}\frac{1}{\epsilon^{c/2}}\tau^{\frac{c\l}{2}-1}\|f\|^{\frac{4+c}{2}}_{p,\gamma}(1+\|x^h\|^{\frac{(4+c)\tilde q}{2}\left(\frac{p\tilde q}{2}+(1+\kappa)\gamma\right)}).\label{th6.1.10}
\end{align}
Plugging \eqref{th6.1.9} and \eqref{th6.1.10} into \eqref{th6.1.5}  yields
\begin{align}\label{th6.1.11}
\lim_{\tau\rightarrow 0}|I_2(k)|\leq &~\lim_{\epsilon\rightarrow 0}\lim_{\tau\rightarrow 0}K\iota^2 e^{\frac{v^2\iota^2}{2}}\frac{1}{\epsilon^{c/2}}\tau^{\frac{c\l}{2}-1}\|f\|^{\frac{4+c}{2}}_{p,\gamma}(1+\|x^h\|^{\frac{(4+c)\tilde q}{2}\left(\frac{p\tilde q}{2}+(1+\kappa)\gamma\right)})\nn\\
&~+\lim_{\epsilon\rightarrow 0}\lim_{\varepsilon\rightarrow 0}\lim_{\tau\rightarrow 0}|I_{2,2}(k)|=0.
\end{align}
\textbf{Estimate of  the term $I_3(k)$.}
To simplify notations, we denote
$\tilde k=\lceil \tau^{-1-\e}\rceil$ and
$M=\lceil k/\tilde k\rceil$. Divide $\{0,1,\ldots, k-1\}$ into $M$ blocks, and denote 
\begin{align*}
\mathbb B_{i}=&~\{(i-1)\tilde k,(i-1)\tilde k +1,\ldots, i\tilde k-1\}~~~~\mbox{for all} ~i\in\{1,\ldots, M-1\},\nn\\
\mathbb B_{M}=&~\{(M-1)\tilde k,\ldots, k-1\}.
\end{align*}
We rewrite the term $I_3(k)$ as
\begin{align}\label{th6.1.12}
I_3(k)
={}&\frac{\iota^2\tau^{2\l}}{2}\sum_{i=1}^{M}\E\Big[\exp{({\bf i} \iota\tau^{\l}\mathscr {M}^{x^h,\Delta}_{(i-1)\tilde k})}\sum_{j\in\mathbb B_{i}}e^{\frac{v^2\iota^2(j+1)}{2k}}\big(\frac{v^2}{k\tau^{2\l}}-(\mathcal Z^{x^h,\Delta}_{j+1})^2\big)\Big]\nn\\
&+\frac{\iota^2\tau^{2\l}}{2}\sum_{i=1}^{M}\sum_{j\in\mathbb B_{i}}\E\Big[\big(\exp{({\bf i} \iota\tau^{\l}\mathscr {M}^{x^h,\Delta}_{j})}-\exp{({\bf i} \iota\tau^{\l}\mathscr {M}^{x^h,\Delta}_{(i-1)\tilde k})}\big)e^{\frac{v^2\iota^2(j+1)}{2k}}\nn\\
&\times\big(\frac{v^2}{k\tau^{2\l}}-(\mathcal Z^{x^h,\Delta}_{j+1})^2\big)\Big]\nn\\
\leq{}& I_{3,1}(k)+I_{3,2}(k),
\end{align}
where
\begin{align*}
I_{3,1}(k):=&~\frac{\iota^2\tau^{2\l}}{2}\sum_{i=1}^{M}\E\Big[\exp{({\bf i} \iota\tau^{\l}\mathscr {M}^{x^h,\Delta}_{(i-1)\tilde k})}\sum_{j\in\mathbb B_{i}}e^{\frac{v^2\iota^2(j+1)}{2k}}\big(\frac{v^2}{k\tau^{2\l}}-(\mathcal Z^{x^h,\Delta}_{j+1})^2\big)\Big],\nn\\
I_{3,2}(k):=&~\frac{\iota^2\tau^{2\l}}{2}e^{\frac{v^2\iota^2}{2}}\sum_{i=1}^{M}\sum_{j\in\mathbb B_{i}}\E\Big[\big|\exp{\big({\bf i} \iota\tau^{\l}(\mathscr {M}^{x^h,\Delta}_{j}-\mathscr {M}^{x^h,\Delta}_{(i-1)\tilde k})\big)}-1\big|\big(\frac{v^2}{k\tau^{2\l}}+(\mathcal Z^{x^h,\Delta}_{j+1})^2\big)\Big].
\end{align*}
It follows from the property of the conditional expectation and the time-homogeneous Markov property of $\{\mathcal Z^{x^h,\Delta}_{k}\}_{k\in\mathbb N}$ that
\begin{align}\label{th6.1.16}
I_{3,1}(k)\leq {}&\frac{\iota^2\tau^{2\l}
}{2}\sum_{i=1}^{M}\E\Big[\Big|\E\Big[\sum_{j\in\mathbb B_{i}}e^{\frac{v^2\iota^2(j+1)}{2k}}\big(\frac{v^2}{k\tau^{2\l}}-(\mathcal Z^{x^h,\Delta}_{j+1})^2\big)\Big|\mathcal F_{t_{(i-1)\tilde k}}\Big]\Big|\Big]\nn\\
={}&\frac{\iota^2\tau^{2\l}
}{2}\sum_{i=1}^{M}\E\Big[\Big|\sum_{j\in\mathbb B_{i}}e^{\frac{v^2\iota^2(j+1)}{2k}}\big(\frac{v^2}{k\tau^{2\l}}-\E[(\mathcal Z^{y^h,\Delta}_{j+1-(i-1)\tilde k})^2]\big)\Big|_{y^h=Y^{x^h,\Delta}_{t_{(i-1)\tilde k}}}\Big|\Big]\nn\\
\leq&~\frac{\iota^2\tau^{2\l-\e}}{2}e^{\frac{v^2\iota^2}{2}}\sum_{i=1}^{M}\E\Big[\Big|\tau^{\e}\sum_{j\in\mathbb B_{1}}\Big|\frac{v^2}{k\tau^{2\l}}-\E[(\mathcal Z^{y^h,\Delta}_{j+1})^2]\Big|\Big|_{y^h=Y^{x^h,\Delta}_{t_{(i-1)\tilde k}}}\Big|\Big].
\end{align}
It is straightforward to see from Proposition \ref{p6.1}  (ii) that
\begin{align}\label{th6.1.13}
&\tau^{\e}\sum_{j\in\mathbb B_{1}}\Big|\frac{v^2}{k\tau^{2\l}}-\E[(\mathcal Z^{y^h,\Delta}_{j+1})^2]\Big|\nn\\
\leq{}& \tau^{\e}\sum_{j=0}^{\tilde k-1}\Big|\frac{v^2}{k\tau^{2\l}}-\mu^{\Delta}(H^{\Delta})\Big|+\tau^{\e}\sum_{j=0}^{\tilde k-1}|\mu^{\Delta}(H^{\Delta})-\E[(\mathcal Z^{y^h,\Delta}_{j+1})^2]|\nn\\
\leq{}&\tau^{\e+1} \tilde k|\frac{v^2}{k\tau^{2\l+1}}-v^2|+\tau^{\e+1}\tilde k|v^2-\frac{\mu^{\Delta}(H^{\Delta})}{\tau}|+\tau^{\e}\sum_{j=0}^{\lceil\tau^{-1-\frac{\e}{4}}\rceil-1}|\mu^{\Delta}(H^{\Delta})-\E[(\mathcal Z^{y^h,\Delta}_{j+1})^2]|\nn\\
&+\tau^{\e}\sum_{j=\lceil\tau^{-1-\frac{\e}{4}}\rceil}^{\tilde k-1}|\mu^{\Delta}(H^{\Delta})-\E[(\mathcal Z^{y^h,\Delta}_{j+1})^2]|\nn\\
\leq{}&\tau^{\e+1}\tilde kv^2|\frac{1}{k\tau^{2\l+1}}-1|+\tau^{\e+1}\tilde k|v^2-\frac{\mu^{\Delta}(H^{\Delta})}{\tau}|+\tau^{\e-1-\frac{\e}{4}}|\mu^{\Delta}(H^{\Delta})|\nn\\
&+\tau^{\e}\sum_{j=0}^{\lceil\tau^{-1-\frac{\e}{4}}\rceil-1}\E[(\mathcal Z^{y^h,\Delta}_{j+1})^2]
+\tau^{\e}\sum_{j=\lceil\tau^{-1-\frac{\e}{4}}\rceil}^{\tilde k-1}|\mu^{\Delta}(H^{\Delta})-P^{\Delta}_{t_j}H^{\Delta}(y^h)|.
\end{align}
Using the same techniques as \eqref{th6.1.14+} and \eqref{th6.1.15+} we derive
\begin{align}\label{th6.1.14}
\sum_{j=0}^{\lceil\tau^{-1-\frac{\e}{4}}\rceil-1}\E[(\mathcal {Z}^{y^h,\Delta}_{j+1})^2]
\leq K\|f\|^2_{p,\gamma}(1+\|y^h\|^{\tilde q(p\tilde q+2(1+\kappa)\gamma)})\tau^{-\frac{\e}{2}} 
\end{align}
and
\begin{align}\label{th6.1.15}
&\tau^{\e}\sum_{j=\lceil\tau^{-1-\frac{\e}{4}}\rceil}^{\tilde k-1}|\mu^{\Delta}(H^{\Delta})-P^{\Delta}_{t_j}H^{\Delta}(y^h)|\nn\\
\leq&~K \|f\|_{p,\gamma}^2(1+\|y^h\|^{\tilde q\tilde p_{\gamma}+(1+\kappa)\gamma})\tau^{\e}\sum_{j=\lceil\tau^{-1-\frac{\e}{4}}\rceil}^{\tilde k-1}\big(\rho^{\Delta}(t_j)\big)^{\gamma}.
\end{align}
Inserting \eqref{th6.1.14} and \eqref{th6.1.15} into \eqref{th6.1.13} and using $|\mu^{\Delta}(H^{\Delta})|\leq K\tau\|f\|^{2}_{p,\gamma}$, we arrive at
\begin{align}\label{th6.1.17}
&\tau^{\e}\sum_{j\in\mathbb B_{1}}|\frac{v^2}{k\tau^{2\l}}-\E[(\mathcal Z^{y^h,\Delta}_{j+1})^2]|\nn\\
\leq&~\tau^{\e+1}\tilde kv^2|\frac{1}{k\tau^{2\l+1}}-1|+\tau^{\e+1}\tilde k|v^2-\frac{\mu^{\Delta}(H^{\Delta})}{\tau}|+\tau^{\frac{3\e}{4}}K\|f\|^{2}_{p,\gamma}\nn\\
&+K\|f\|^2_{p,\gamma}(1+\|y^h\|^{\tilde q(p\tilde q+2(1+\kappa)\gamma)})\tau^{\frac{\e}{2}}\nn\\
&+K \|f\|_{p,\gamma}^2(1+\|y^h\|^{\tilde q\tilde p_{\gamma}+(1+\kappa)\gamma})\tau^{\e}\sum_{j=\lceil\tau^{-1-\frac{\e}{4}}\rceil}^{\tilde k-1}\big(\rho^{\Delta}(t_j)\big)^{\gamma}.
\end{align}
Since $\tilde q(p\tilde q+2(1+\kappa)\gamma)\leq q$ and $\tilde q\tilde p_{\gamma}+(1+\kappa)\gamma\leq q$, plugging \eqref{th6.1.17} into \eqref{th6.1.16}  and applying Assumption \ref{a2} (\romannumeral1), we have
\begin{align*}
I_{3,1}(k)\leq{}&K\iota^2e^{\frac{v^2\iota^2}{2}}\Big(v^2|\frac{1}{k\tau^{2\l+1}}-1|+|v^2-\frac{\mu^{\Delta}(H^{\Delta})}{\tau}|+\tau^{\frac{3\e}{4}}\|f\|^{2}_{p,\gamma}\nn\\
&+\|f\|^2_{p,\gamma}(1+\|x^h\|^{\tilde q^2(p\tilde q+2(1+\kappa)\gamma)})\tau^{\frac{\e}{2}}\nn\\
&+\|f\|_{p,\gamma}^2(1+\|x^h\|^{\tilde q(q\tilde p_{\gamma}+(1+\kappa)\gamma)})\tau^{\e}\sum_{j=\lceil\tau^{-1-\frac{\e}{4}}\rceil}^{\tilde k-1}\big(\rho^{\Delta}(t_j)\big)^{\gamma}\Big),
\end{align*}
where we used
$k=\lceil \tau^{-2\l-1}\rceil$, $\tilde k=\lceil \tau^{-1-\e}\rceil$, and
$M=\lceil k/\tilde k\rceil=\mathcal O(\tau^{-2\l+\e})$. 
It follows from \eqref{rho_remainder} that
$
\lim_{\e\rightarrow 0}\lim_{\tau\rightarrow 0}\tau^{\e}\sum_{j=\lceil\tau^{-1-\frac{\e}{4}}\rceil}^{\tilde k-1}\big(\rho^{\Delta}(t_j)\big)^{\gamma}=0,
$
This, along with 
$
p(1+\tilde q\vee\tilde r)+2(1+\kappa\vee\beta)\gamma\leq q\wedge r
$ and Proposition \ref{p6.2} implies that
\begin{align}\label{th6.1.19}
\lim_{|\Delta|\rightarrow0}I_{3,1}(k)=0.
\end{align}
Furthermore, by $|e^{{\bf i}u}-1|\leq u$ and  the H\"older inequality, one has that for any $\e_1>0$,
\begin{align}\label{th6.1.20}
I_{3,2}(k)\leq
{}&\e_1\iota \frac{\iota^2\tau^{2\l}}{2}e^{\frac{v^2\iota^2}{2}}\sum_{i=1}^{M}\sum_{j\in\mathbb B_{i}}\big(\frac{v^2}{k\tau^{2\l}}+\E[(\mathcal Z^{x^h,\Delta}_{j+1})^2]\big)\nn\\
&+\iota^2\tau^{2\l}e^{\frac{v^2\iota^2}{2}}\sum_{i=1}^{M}\sum_{j\in\mathbb B_{i}}\E\Big[\big(\frac{v^2}{k\tau^{2\l}}+(\mathcal Z^{x^h,\Delta}_{j+1})^2\big)\textbf 1_{\{\tau^{\l}|\mathscr {M}^{x^h,\Delta}_{j}-\mathscr {M}^{x^h,\Delta}_{(i-1)\tilde k}|>\e_1\}}\Big]\nn\\
\leq{}&K\e_1 \iota^3e^{\frac{v^2\iota^2}{2}}+\e_1 \frac{\iota^3\tau^{2\l}}{2}e^{\frac{v^2\iota^2}{2}}\sum_{j=0}^{k-1}\E\big[H^{\Delta}(Y^{x^h,\Delta}_{t_j})\big]\nn\\
&+K\iota^2\tau^{2\l}e^{\frac{v^2\iota^2}{2}}\sum_{i=1}^{M}\sum_{j\in\mathbb B_{i}}\big(v^4+\E[(\mathcal Z^{x^h,\Delta}_{j+1})^4]\big)^{\frac{1}{2}}\big(\PP\{\tau^{\l}|\mathscr {M}^{x^h,\Delta}_{j}-\mathscr {M}^{x^h,\Delta}_{(i-1)\tilde k}|>\e_1\}\big)^{\frac{1}{2}}.
\end{align}
It follows from  \eqref{trmix1} and the definition of $\mathscr {M}^{x^h,\Delta}_{i}$  that for any $j\in\mathbb B_{i}\setminus\{(i-1)\tilde k\}$,
\begin{align*}
&|\mathscr {M}^{x^h,\Delta}_{j}-\mathscr {M}^{x^h,\Delta}_{(i-1)\tilde k}|\nn\\
\leq{}&\tau\Big|\sum_{l=(i-1)\tilde k}^{j-1}\big(f(Y_{t_l}^{x^h,\Delta})-\mu^{\Delta}(f)\big)\Big|+\tau\sum_{l=0}^{\infty}\big|P^{\Delta}_{t_l}f(Y_{t_j}^{x^h,\Delta})-\mu^{\Delta}(f)\big|\nn\\
&+\tau\sum_{l=0}^{\infty}\big|P^{\Delta}_{t_l}f(Y_{t_{(i-1)\tilde k}}^{x^h,\Delta})-\mu^{\Delta}(f)\big|\nn\\
\leq{}&\tau\Big|\!\!\sum_{l=(i-1)\tilde k}^{j-1}\!\!\big(f(Y_{t_l}^{x^h,\Delta})-\mu^{\Delta}(f)\big)\Big|\!+\!
K\|f\|_{p,\gamma}(1\!+\!\|Y^{x^h,\Delta}_{t_j}\|^{\frac{p\tilde q}{2}+(1+\kappa)\gamma}\!+\!\|Y^{x^h,\Delta}_{t_{(i-1)\tilde k}}\|^{\frac{p\tilde q}{2}+(1+\kappa)\gamma}).
\end{align*}
Let $\e\ll1$ such that $\lceil\frac{1+\sqrt\e}{\l}\rceil\leq \frac{1}{\l}+1$.
This, along with the condition \eqref{condip} implies that  $2\lceil 3\vee\frac{1+\sqrt\e}{\l}\rceil\left(\frac{p\tilde q}{2}+(1+\kappa)\gamma\right)\leq q$.
Combining the Chebyshev inequality and Assumption \ref{a2} (i) we yield  
\begin{align}\label{th6.1.21}
&\PP{\{\tau^{\l}|\mathscr {M}^{x^h,\Delta}_{j}-\mathscr {M}^{x^h,\Delta}_{(i-1)\tilde k}|>\e_1\}}\nn\\
\leq{}&(\frac{\tau^{\l}}{\e_1})^{2\lceil 3\vee\frac{1+\sqrt\e}{\l}\rceil}\E\Big[|\mathscr {M}^{x^h,\Delta}_{j}-\mathscr {M}^{x^h,\Delta}_{(i-1)\tilde k}|^{2\lceil 3\vee\frac{1+\sqrt\e}{\l}\rceil}\Big]\nn\\
\leq{}&K(\frac{\tau^{\l+1}}{\e_1})^{2\lceil 3\vee\frac{1+\sqrt\e}{\l}\rceil}\E\Big[\big|\sum_{l=(i-1)\tilde k}^{j-1}\big(f(Y_{t_l}^{x^h,\Delta})-\mu^{\Delta}(f)\big)\big|^{2\lceil 3\vee\frac{1+\sqrt\e}{\l}\rceil}\Big]\nn\\
&+K\|f\|_{p,\gamma}^{{2\lceil 3\vee\frac{1+\sqrt\e}{\l}\rceil}}(\frac{\tau^{\l}}{\e_1})^{2\lceil 3\vee\frac{1+\sqrt\e}{\l}\rceil}\Big(1+\sup_{n\geq0}\E\big[\|Y^{x^h,\Delta}_{t_n}\|^{2\lceil 3\vee\frac{1+\sqrt\e}{\l}\rceil\left(\frac{p\tilde q}{2}+(1+\kappa)\gamma\right)}\big]\Big)\nn\\
\leq{}&K(\frac{\tau^{\l+1}}{\e_1})^{2\lceil 3\vee\frac{1+\sqrt\e}{\l}\rceil}\E\Big[\E\Big[\big|\sum_{l=0}^{j-(i-1)\tilde k-1}\big(f(Y_{t_l}^{y^h,\Delta})-\mu^{\Delta}(f)\big)\big|^{2\lceil 3\vee\frac{1+\sqrt\e}{\l}\rceil}\Big]\Big|_{y^h=Y^{x^h,\Delta}_{t_{(i-1)\tilde k}}}\Big]\nn\\
&+K\|f\|_{p,\gamma}^{{2\lceil 3\vee\frac{1+\sqrt\e}{\l}\rceil}}(\frac{\tau^{\l}}{\e_1})^{2\lceil 3\vee\frac{1+\sqrt\e}{\l}\rceil}(1+\|x^h\|^{2\tilde q\lceil 3\vee\frac{1+\sqrt\e}{\l}\rceil\left(\frac{p\tilde q}{2}+(1+\kappa)\gamma\right)}).
\end{align}
Note that the condition of \eqref{condip} leads to  that  the one on $p$ in Lemma \ref{l5.1} holds with $\varrho=\lceil 3\vee\frac{1+\sqrt\e}{\l}\rceil$. By virtue of Lemma \ref{l5.1} with $\varrho=\lceil 3\vee\frac{1+\sqrt\e}{\l}\rceil$, we obtain
\begin{align*}
&\E\Big[\big|\sum_{l=0}^{j-(i-1)\tilde k-1}\big(f(Y_{t_l}^{y^h,\Delta})-\mu^{\Delta}(f)\big)\big|^{2\lceil 3\vee\frac{1+\sqrt\e}{\l}\rceil}\Big]\nn\\
={}&\big(j-(i-1)\tilde k\big)^{2\lceil 3\vee\frac{1+\sqrt\e}{\l}\rceil}\E\Big[\big|\frac{1}{j-(i-1)\tilde k}\sum_{l=0}^{j-(i-1)\tilde k-1}\big(f(Y_{t_l}^{y^h,\Delta})-\mu^{\Delta}(f)\big)\big|^{2\lceil 3\vee\frac{1+\sqrt\e}{\l}\rceil}\Big]\nn\\
\leq{}& K\|f\|^{2\lceil 3\vee\frac{1+\sqrt\e}{\l}\rceil}_{p,\gamma}\big(j-(i-1)\tilde k\big)^{2\lceil 3\vee\frac{1+\sqrt\e}{\l}\rceil} (1+\|y^h\|^{\lceil 3\vee\frac{1+\sqrt\e}{\l}\rceil\tilde q\left(\frac{p}{2}(1+\tilde q)+(1+\kappa)\gamma\right)})t_{j-(i-1)\tilde k}^{-\lceil 3\vee\frac{1+\sqrt\e}{\l}\rceil}.
\end{align*}
This, along with \eqref{th6.1.21}, $\tilde q\lceil 3\vee\frac{1+\sqrt\e}{\l}\rceil\left(\frac{p}{2}(1+\frac{\tilde q}{q})+(1+\kappa)\gamma\right)\leq q$, and Assumption \ref{a2} (\romannumeral1) implies that 
\begin{align*}
&\sum_{j\in\mathbb B_{i}}\big(\PP{\{\tau^{\l}|\mathscr {M}^{x^h,\Delta}_{j}-\mathscr {M}^{x^h,\Delta}_{(i-1)\tilde k}|>\e_1\}}\big)^{\frac{1}{2}}\nn\\
\leq{}&K\|f\|^{\lceil 3\vee\frac{1+\sqrt\e}{\l}\rceil}_{p,\gamma}\sum_{j\in\mathbb B_{i}}(\frac{\tau^{\l+1}(j-(i-1)\tilde k)}{\e_1})^{\lceil 3\vee\frac{1+\sqrt\e}{\l}\rceil}\big(\E[\|Y^{x^h,\Delta}_{t_{(i-1)\tilde k}}\|^{\tilde q\lceil 3\vee\frac{1+\sqrt\e}{\l}\rceil\left(\frac{p}{2}(1+\tilde q)+(1+\kappa)\gamma\right)}]\nn\\
&+1\big)^{\frac{1}{2}} t_{j-(i-1)\tilde k}^{-\lceil 3\vee\frac{1+\sqrt\e}{\l}\rceil/2}+K\|f\|_{p,\gamma}^{{\lceil 3\vee\frac{1+\sqrt\e}{\l}\rceil}}(\frac{\tau^{\l}}{\e_1})^{\lceil 3\vee\frac{1+\sqrt\e}{\l}\rceil}(1+\|x^h\|^{\tilde q\lceil 3\vee\frac{1+\sqrt\e}{\l}\rceil\left(\frac{p\tilde q}{2}+(1+\kappa)\gamma\right)})\nn\\
\leq{}&K\|f\|^{\lceil 3\vee\frac{1+\sqrt\e}{\l}\rceil}_{p,\gamma}(1+\|x^h\|^{\lceil 3\vee\frac{1+\sqrt\e}{\l}\rceil\frac{\tilde q^2}{2}\left(\frac{p}{2}(1+\tilde q)+(1+\kappa)\gamma\right)})\sum_{j\in\mathbb B_{1}}(\frac{\tau^{\l+1}j}{\e_1})^{\lceil 3\vee\frac{1+\sqrt\e}{\l}\rceil}t_{j}^{-\lceil 3\vee\frac{1+\sqrt\e}{\l}\rceil/2}\nn\\
&+K\|f\|_{p,\gamma}^{{\lceil 3\vee\frac{1+\sqrt\e}{\l}\rceil}}(\frac{\tau^{\l}}{\e_1})^{\lceil 3\vee\frac{1+\sqrt\e}{\l}\rceil}(1+\|x^h\|^{\tilde q\lceil 3\vee\frac{1+\sqrt\e}{\l}\rceil\left(\frac{p\tilde q}{2}+(1+\kappa)\gamma\right)})\nn\\
\leq{}&K\|f\|^{\lceil 3\vee\frac{1+\sqrt\e}{\l}\rceil}_{p,\gamma}(1+\|x^h\|^{\lceil 3\vee\frac{1+\sqrt\e}{\l}\rceil\tilde q^2\left(\frac{p\tilde q}{2}+(1+\kappa)\gamma\right)})\Big((\frac{\tau^{\l}}{\e_1})^{\lceil 3\vee\frac{1+\sqrt\e}{\l}\rceil}\sum_{j=0}^{\lceil\tau^{-1}\rceil-1}t_{j}^{\lceil 3\vee\frac{1+\sqrt\e}{\l}\rceil/2}\nn\\
&+(\frac{\tau^{\l-\e}}{\e_1})^{\lceil 3\vee\frac{1+\sqrt\e}{\l}\rceil}\sum_{j=\lceil\tau^{-1}\rceil}^{\tilde k-1}(\frac{j}{\tau^{-1-\e}})^{\lceil 3\vee\frac{1+\sqrt\e}{\l}\rceil}t_{j}^{-\lceil 3\vee\frac{1+\sqrt\e}{\l}\rceil/2}+(\frac{\tau^{\l}}{\e_1})^{\lceil 3\vee\frac{1+\sqrt\e}{\l}\rceil}\Big)\nn\\
\leq{}&\frac{1}{\e_1^{\lceil 3\vee\frac{1+\sqrt\e}{\l}\rceil}}K\|f\|^{\lceil 3\vee\frac{1+\sqrt\e}{\l}\rceil}_{p,\gamma}(1+\|x^h\|^{\lceil 3\vee\frac{1+\sqrt\e}{\l}\rceil\tilde q^2\left(\frac{p\tilde q}{2}+(1+\kappa)\gamma\right)})\Big(\int_{0}^{2}\zeta^{\lceil 3\vee\frac{1+\sqrt\e}{\l}\rceil/2}\mathrm d\zeta\nn\\
&+\tau^{(\l-\e)\lceil 3\vee\frac{1+\sqrt\e}{\l}\rceil-1}\int_{1}^{\infty}\zeta^{-\lceil 3\vee\frac{1+\sqrt\e}{\l}\rceil/2}\mathrm d\zeta+\tau\Big)\nn\\
\leq{}&\frac{1}{\e_1^{\lceil 3\vee\frac{1+\sqrt\e}{\l}\rceil}}K\|f\|^{\lceil 3\vee\frac{1+\sqrt\e}{\l}\rceil}_{p,\gamma}(1+\|x^h\|^{\lceil 3\vee\frac{1+\sqrt\e}{\l}\rceil\tilde q^2\left(\frac{p\tilde q}{2}+(1+\kappa)\gamma\right)})(1+\tau^{\frac{(\l-\e)(1+\sqrt \e)}{\l}-1}+\tau).
\end{align*}
Inserting the above inequality into \eqref{th6.1.20} leads to
\begin{align*}
I_{3,2}(k)
\leq{}& K\e_1 \iota^3e^{\frac{v^2\iota^2}{2}} +\e_1 \frac{\iota^3\tau^{2\l}}{2}e^{\frac{v^2\iota^2}{2}}\sum_{j=0}^{k-1}\E\big[H^{\Delta}(Y^{x^h,\Delta}_{t_j})\big]\nn\\
&+K\iota^2e^{\frac{v^2\iota^2}{2}}\big(v^4+\sup_{n\geq0}\E[(\mathcal Z^{x^h,\Delta}_{n})^4]\big)^{\frac{1}{2}}\frac{1}{\e_1^{\lceil 3\vee\frac{1+\sqrt\e}{\l}\rceil}}\|f\|^{\lceil 3\vee\frac{1+\sqrt\e}{\l}\rceil}_{p,\gamma}\nn\\
&\times(1+\|x^h\|^{\lceil 3\vee\frac{1+\sqrt\e}{\l}\rceil\tilde q^2\left(\frac{p\tilde q}{2}+(1+\kappa)\gamma\right)})(1+\tau^{\sqrt \e-\frac{\e}{\l}-\frac{\e\sqrt \e}{\l}}+\tau)M\tau^{2\l}.
\end{align*}
Recalling $M=\mathcal O(\tau^{-2\l+\e})$, letting $\e\in(0, 1)$ satisfy $\sqrt \e-\frac{\e}{\l}-\frac{\e\sqrt \e}{\l}>0$, and using \eqref{Zbound} and \eqref{th6.1.7+}--\eqref{th6.1.15+} yield
\begin{align}\label{th6.1.22}
\lim_{\e_1\rightarrow 0}\lim_{\tau\rightarrow 0}I_{3,2}(k)=0.
\end{align}
We conclude from \eqref{th6.1.12}, \eqref{th6.1.19}, and \eqref{th6.1.22}  that
\begin{align}\label{th6.1.23}
\lim_{|\Delta|\rightarrow 0}I_{3}(k)=0.
\end{align}
The desired argument \eqref{claimM} follows from \eqref{th6.1.3}, \eqref{th6.1.4}, \eqref{th6.1.9}, and \eqref{th6.1.23}.

\textit{Step 2: Prove that $\frac{1}{\sqrt{k\tau}}\mathscr R^{x^h,\Delta}_{k}$ converges in probability to $0$.} Recalling $k=\lceil\tau^{-2\l-1}\rceil,$ it is equivalent to show that $\tau^{\l}\mathscr R^{x^h,\Delta}_{\lceil\tau^{-2\l-1}\rceil}$ converges in probability to $0$. Let $\mathscr R^{x^h,\Delta}_{k}=\mathscr R^{x^h,\Delta}_{k,1}+\mathscr R^{x^h,\Delta}_{k,2},$ where 
 $$\mathscr R^{x^h,\Delta}_{k,1}:=-\tau\sum_{i=k}^{\infty}\Big(\E\big[f(Y_{t_i}^{x^h,\Delta})|\mathcal F_{t_{k}}\big]-\mu^{\Delta}(f)\Big)+\tau\sum_{i=0}^{\infty}\Big(\E\big[f(Y_{t_i}^{x^h,\Delta})|\mathcal F_{0}\big]-\mu^{\Delta}(f)\Big)$$ and $$\mathscr R^{x^h,\Delta}_{k,2}:=k\tau(\mu^{\Delta}(f)-\mu(f)).$$ 
By \eqref{trmix1} we deduce
\begin{align*}
|\mathscr R^{x^h,\Delta}_{k,1}|\leq&~\tau\sum_{i=0}^{\infty}\big|P_{t_{i}}^{\Delta}f(Y_{t_k}^{x^h,\Delta})-\mu^{\Delta}(f)\big|
+\tau\sum_{i=0}^{\infty}\big|P_{t_{i}}^{\Delta}f(x^h)-\mu^{\Delta}(f)\big|\nn\\
\leq&~K\|f\|_{p,\gamma}(1+\|Y_{t_k}^{x^h,\Delta}\|^{\frac{p\tilde q}{2}+(1+\kappa)\gamma})\tau\sum_{i=0}^{\infty}(\rho^{\Delta}(t_i))^{\gamma}\nn\\
&+K\|f\|_{p,\gamma}(1+\|x^h\|^{\frac{p\tilde q}{2}+(1+\kappa)\gamma})\tau\sum_{i=0}^{\infty}(\rho^{\Delta}(t_i))^{\gamma}\nn\\
\leq&~K\|f\|_{p,\gamma}(1+\|Y_{t_k}^{x^h,\Delta}\|^{\frac{p\tilde q}{2}+(1+\kappa)\gamma}+\|x^h\|^{\frac{p\tilde q}{2}+(1+\kappa)\gamma}).
\end{align*}
Taking expectation on both sides of $|\mathscr R^{x^h,\Delta}_{k,1}|$ and using Assumption \ref{a2} (\romannumeral1), we derive 
\begin{align*}
\E[|\mathscr R^{x^h,\Delta}_{k,1}|]\leq&~ K\|f\|_{p,\gamma}(1+\sup_{k\geq 0}\E[\|Y^{x^h,\Delta}_{t_{k}}\|^{\frac{p\tilde q}{2}+(1+\kappa)\gamma}]+\|x^h\|^{\frac{p\tilde q}{2}+(1+\kappa)\gamma})\nn\\
\leq&~ K\|f\|_{p,\gamma}(1+\|x^h\|^{\tilde q(\frac{p\tilde q}{2}+(1+\kappa)\gamma)}).
\end{align*}
This, along with $\frac{p\tilde q}{2}+(1+\kappa)\gamma\leq q$ and Assumption \ref{a2} (i) leads to that 
\begin{align}\label{R+}
\lim_{|\Delta|\rightarrow 0}\frac{1}{\sqrt{k\tau}}\E[|\mathscr R^{x^h,\Delta}_{k,1}|]\leq \lim_{|\Delta|\rightarrow 0}K\|f\|_{p,\gamma}(1+\|x^h\|^{\tilde q(\frac{p\tilde q}{2}+(1+\kappa)\gamma)})\tau^{\l}=0.
\end{align}
It follows from \eqref{t3.1} that
\begin{align*}
\frac{1}{\sqrt{k\tau}}|\mathscr R^{x^h,\Delta}_{k,2}|=\sqrt{k\tau}\big|\mu^{\Delta}(f)-\mu(f)\big|
\leq K\|f\|_{p,\gamma}\sqrt{k\tau}(h^{\a_1\gamma}+\tau^{\a_2\gamma}).
\end{align*}
Since 
 $h=\tau^{\a_2/\a_1}$ when  $h\neq0$,
 one has
$$
K\|f\|_{p,\gamma}\sqrt{k\tau}(h^{\a_1\gamma}+\tau^{\a_2\gamma})
=K\|f\|_{p,\gamma}\tau^{\a_2\gamma-\lambda}\stackrel{\tau\rightarrow0}{\longrightarrow}0
$$
by letting $\lambda\in(0,\a_2\gamma)$. 

Combining \textit{Steps 1--2} finishes the proof.
\end{proof}

\section{Proofs of Propositions in Section \ref{sec_CLT}}\label{proof_prop} 
This section gives proofs of Propositions \ref{Prop4.1}--\ref{p6.2}.
\begin{proof}[\textbf{Proof of Proposition \ref{Prop4.1}}.]
It follows from the definition of $\mathscr M^{x^h,\Delta}_{k}$ (see \eqref{mar_eq}) that
\begin{align}\label{p6.1.1}
\mathscr M^{x^h,\Delta}_{k}=&~\tau\sum_{i=0}^{k-1}\big(f(Y_{t_i}^{x^h,\Delta})-\mu^{\Delta}(f)\big) +\tau\sum_{i=k}^{\infty}\Big(\E\big[f(Y_{t_i}^{x^h,\Delta})|\mathcal F_{t_{k}}\big]-\mu^{\Delta}(f)\Big)\nn\\
&-\tau\sum_{i=0}^{\infty}\Big(\E\big[f(Y_{t_i}^{x^h,\Delta})|\mathcal F_{0}\big]-\mu^{\Delta}(f)\Big)\nn\\
=&~\tau\sum_{i=0}^{k-1}\big(f(Y_{t_i}^{x^h,\Delta})-\mu^{\Delta}(f)\big)+\tau\sum_{i=0}^{\infty}\big(P^{\Delta}_{t_i}f(Y_{t_k}^{x^h,\Delta})-\mu^{\Delta}(f)\big)\nn\\
&-\tau\sum_{i=0}^{\infty}\big(P^{\Delta}_{t_i}f(x^h)-\mu^{\Delta}(f)\big).
\end{align}
By  \eqref{trmix1},  we have 
\begin{align*}
&|\mathscr M^{x^h,\Delta}_{k}|\nn\\
\leq&~\tau\big|\sum_{i=0}^{k-1}\big(f(Y_{t_i}^{x^h,\Delta})-\mu^{\Delta}(f)\big)\big|+K\|f\|_{p,\gamma}(1+\|Y_{t_k}^{x^h,\Delta}\|^{\frac{p\tilde q}{2}+(1+\kappa)\gamma})\tau\sum_{i=0}^{\infty}(\rho^{\Delta}(t_i))^{\gamma}\nn\\
&+K\|f\|_{p,\gamma}(1+\|x^h\|^{\frac{p\tilde q}{2}+(1+\kappa)\gamma})\tau\sum_{i=0}^{\infty}(\rho^{\Delta}(t_i))^{\gamma}.
\end{align*}
Using the conditions $\gamma\in[\gamma_2,1]$ and $\frac{p\tilde q}{2}+(1+\kappa)\gamma\leq q$, and  Assumption \ref{a2} leads to
\begin{align*}
\E[|\mathscr M^{x^h,\Delta}_{k}|]
\leq&{} K\|f\|_{p,\gamma}k(1+\sup_{i\geq0}\E[\|Y_{t_i}^{x^h,\Delta}\|^{\frac{p}{2}}])+K\|f\|_{p,\gamma}(1+\E[\|Y_{t_k}^{x^h,\Delta}\|^{\frac{p\tilde q}{2}+(1+\kappa)\gamma}])\nn\\
&+K\|f\|_{p,\gamma}(1+\|x^h\|^{\frac{p\tilde q}{2}+(1+\kappa)\gamma})
\leq~K\|f\|_{p,\gamma}k(1+\|x^h\|^{\tilde q(\frac{p\tilde q}{2}+(1+\kappa)\gamma)})<\infty.
\end{align*}
In addition, by \eqref{p6.1.1}, we have
\begin{align}\label{p6.1.3}
\mathscr M^{x^h,\Delta}_{k}-\mathscr M^{x^h,\Delta}_{k-1}=&~\tau\big(f(Y_{t_{k-1}}^{x^h,\Delta})-\mu^{\Delta}(f)\big)+\tau\sum_{i=0}^{\infty}\big(P^{\Delta}_{t_i}f(Y_{t_k}^{x^h,\Delta})-\mu^{\Delta}(f)\big)\nn\\
&-\tau\sum_{i=0}^{\infty}\big(P^{\Delta}_{t_i}f(Y_{t_{k-1}}^{x^h,\Delta})-\mu^{\Delta}(f)\big),
\end{align}
which gives
\begin{align}\label{p6.1.2}
&\E\Big[(\mathscr M^{x^h,\Delta}_{k+j}-\mathscr M^{x^h,\Delta}_{k})|\mathcal F_{t_k}\Big]\nn\\
=&~\tau\sum_{i=k}^{k+j-1}\Big(\E\big[f(Y_{t_{i}}^{x^h,\Delta})|\mathcal F_{t_k}\big]-\mu^{\Delta}(f)\Big)+\tau\E\Big[\sum_{i=0}^{\infty}\big(P^{\Delta}_{t_i}f(Y_{t_{k+j}}^{x^h,\Delta})-\mu^{\Delta}(f)\big)\big|\mathcal F_{t_k}\Big]\nn\\
&-\tau\sum_{i=0}^{\infty}\big(P^{\Delta}_{t_i}f(Y_{t_{k}}^{x^h,\Delta})-\mu^{\Delta}(f)\big)\nn\\
=&~\tau\sum_{i=0}^{j-1}\big(P^{\Delta}_{t_i}f(Y_{t_{k}}^{x^h,\Delta})-\mu^{\Delta}(f)\big)+\tau P^{\Delta}_{t_j}\sum_{i=0}^{\infty}\big(P^{\Delta}_{t_i}f-\mu^{\Delta}(f)\big)(Y_{t_{k}}^{x^h,\Delta})\nn\\
&-\tau\sum_{i=0}^{\infty}\big(P^{\Delta}_{t_i}f(Y_{t_{k}}^{x^h,\Delta})-\mu^{\Delta}(f)\big).
\end{align}
Applying the dominated convergence theorem yields
\begin{align*}
\tau P^{\Delta}_{t_j}\sum_{i=0}^{\infty}\big(P^{\Delta}_{t_i}f-\mu^{\Delta}(f)\big)(Y_{t_{k}}^{x^h,\Delta})=&~\tau\sum_{i=0}^{\infty}\big(P^{\Delta}_{t_{i+j}}f-\mu^{\Delta}(f)\big)(Y_{t_{k}}^{x^h,\Delta})\nn\\
=&~\tau\sum_{i=j}^{\infty}\big(P^{\Delta}_{t_{i}}f(Y_{t_{k}}^{x^h,\Delta})-\mu^{\Delta}(f)\big).
\end{align*}
Inserting the above equality into \eqref{p6.1.2} leads to 
$\E\Big[(\mathscr M^{x^h,\Delta}_{k+j}-\mathscr M^{x^h,\Delta}_{k})|\mathcal F_{t_k}\Big]=0.$ This finishes the proof.
\end{proof}

\begin{proof}[\textbf{Proof of Proposition \ref{p6.1}}.]
(\romannumeral1) 
It follows from \eqref{trmub}, \eqref{trmix1},  \eqref{p6.1.3},  Assumption \ref{a2}, and the condition $\gamma\in[\gamma_2,1]$ that
\begin{align}\label{Z++}
|\mathcal Z^{x^h,\Delta}_{k}|\leq&~\tau\big|f(Y_{t_{k-1}}^{x^h,\Delta})-\mu^{\Delta}(f)\big|+\tau\sum_{i=0}^{\infty}\big|P^{\Delta}_{t_i}f(Y_{t_k}^{x^h,\Delta})-\mu^{\Delta}(f)\big|\nn\\
&+\tau\sum_{i=0}^{\infty}\big|P^{\Delta}_{t_i}f(Y_{t_{k-1}}^{x^h,\Delta})-\mu^{\Delta}(f)\big|\nn\\
\leq&~K\|f\|_{p,\gamma}\tau\big(1+\|Y_{t_k}^{x^h,\Delta}\|^{\frac{p}{2}}\big)+K\|f\|_{p,\gamma}\Big(1+\|Y_{t_k}^{x^h,\Delta}\|^{\frac{p\tilde q}{2}+(1+\kappa)\gamma}\nn\\
&+\|Y^{x^h,\Delta}_{t_{k-1}}\|^{\frac{p\tilde q}{2}+(1+\kappa)\gamma}\Big).
\end{align}
Taking any constant $c$ satisfying
$c\big(\frac{p\tilde q}{2}+(1+\kappa)\gamma\big)\leq q$ and using Assumption \ref{a2} (i)
finish the proof of (\romannumeral1).

(\romannumeral2) The condition $p\tilde q+2(1+\kappa)\gamma\leq q$ coincides with the one in (\romannumeral1) with $c=2,$ thus $\E\big[|\mathcal Z^{x^h,\Delta}_{l+1}|^2|\mathcal F_{t_{l}}\big]$ is well-defined. 
It follows from \eqref{p6.1.3} that
\begin{align}\label{Vrewri}
\E\big[|\mathcal Z^{x^h,\Delta}_{l+1}|^2|\mathcal F_{t_{l}}\big]
=&~\E\Big[\big|\tau\big(f(y^h)-\mu^{\Delta}(f)\big)+\tau\sum_{i=0}^{\infty}\big(P^{\Delta}_{t_i}f(Y_{t_1}^{y^h,\Delta})-\mu^{\Delta}(f)\big)\nn\\
&~-\tau\sum_{i=0}^{\infty}\big(P^{\Delta}_{t_i}f(y^h)-\mu^{\Delta}(f)\big)\big|^2\Big]\Big|_{y^h=Y^{x^h,\Delta}_{t_{l}}}=:G^{\Delta}(Y^{x^h,\Delta}_{t_{l}}).
\end{align}
For any $u\in E^h$,
\begin{align}
 G^{\Delta}(u)
=&~\tau^2|f(u)-\mu^{\Delta}(f)|^2+\tau^2 \Big|\sum_{i=0}^{\infty}\big(P^{\Delta}_{t_i}f(u)-\mu^{\Delta}(f)\big)\Big|^2\nn\\
&+\tau^2\E\Big[\Big|\sum_{i=0}^{\infty}\big(P^{\Delta}_{t_i}f(Y^{u,\Delta}_{t_1})-\mu^{\Delta}(f)\big)\Big|^2\Big]\nn\\
&+2\tau^2\big(f(u)-\mu^{\Delta}(f)\big) P^{\Delta}_{t_1}\Big(\sum_{i=0}^{\infty}\big(P^{\Delta}_{t_i}f-\mu^{\Delta}(f)\big)\Big)(u)\nn\\
&-2\tau^2\big( f(u)-\mu^{\Delta}(f)\big)\sum_{i=0}^{\infty}\big(P^{\Delta}_{t_i}f(u)-\mu^{\Delta}(f)\big)\nn\\
&-2\tau^2\sum_{i=0}^{\infty}\big(P^{\Delta}_{t_i}f(u)-\mu^{\Delta}(f)\big)P^{\Delta}_{t_1}\Big(\sum_{j=0}^{\infty}\big(P^{\Delta}_{t_j}f-\mu^{\Delta}(f)\big)\Big)(u).\label{+l6.1.1}
\end{align}
Applying  the dominated convergence theorem yields 
\begin{align}\label{l6.1+1}
&\big(f(u)-\mu^{\Delta}(f)\big) P^{\Delta}_{t_1}\Big(\sum_{i=0}^{\infty}\big(P^{\Delta}_{t_i}f-\mu^{\Delta}(f)\big)\Big)(u)\nn\\
=&~\big(f(u)-\mu^{\Delta}(f)\big)\sum_{i=0}^{\infty}\big(P^{\Delta}_{t_{i+1}}f(u)-\mu^{\Delta}(f)\big)\nn\\
=&~\big(f(u)-\mu^{\Delta}(f)\big)\sum_{i=0}^{\infty}\big(P^{\Delta}_{t_{i}}f(u)-\mu^{\Delta}(f)\big)-|f(u)-\mu^{\Delta}(f)|^2.
\end{align}
Similarly,
\begin{align}\label{l6.1+2}
&\sum_{i=0}^{\infty}\big(P^{\Delta}_{t_i}f(u)-\mu^{\Delta}(f)\big)P^{\Delta}_{t_1}\Big(\sum_{i=0}^{\infty}\big(P^{\Delta}_{t_i}f-\mu^{\Delta}(f)\big)\Big)(u)\nn\\
=&~\big|\sum_{i=0}^{\infty}\big(P^{\Delta}_{t_i}f(u)-\mu^{\Delta}(f)\big)\big|^2-\big(f(u)-\mu^{\Delta}(f)\big)\sum_{i=0}^{\infty}\big(P^{\Delta}_{t_i}f(u)-\mu^{\Delta}(f)\big).
\end{align}
Inserting \eqref{l6.1+1} and \eqref{l6.1+2} into \eqref{+l6.1.1} implies $G^{\Delta}=H^{\Delta}$, which finishes the proof of (\romannumeral2).

(\romannumeral3) For any $u\in E^h$, 
using \eqref{trmub},  \eqref{trmix1}, the conditions $\gamma\in[\gamma_2,1]$  and $p\tilde q+2(1+\kappa)\gamma\leq q$,  and Assumption \ref{a2},  we deduce
\begin{align*}
&|H^{\Delta}(u)|\nn\\
\leq &\;K\|f\|^2_{p,\gamma}\tau^2(1+\|u\|^{p})+K\|f\|^2_{p,\gamma}\big(1+\E[\|Y^{u,\Delta}_{t_1}\|^{p\tilde q+2(1+\kappa)\gamma}]\big)\Big(\tau\sum_{i=0}^{\infty}\big(\rho^{\Delta}(t_{i})\big)^{\gamma}\Big)^2\nn\\
&+K\|f\|^2_{p,\gamma}\big(1+\|u\|^{p\tilde q+2(1+\kappa)\gamma}\big)\Big(\tau\sum_{i=0}^{\infty}\big(\rho^{\Delta}(t_{i})\big)^{\gamma}\Big)^2\nn\\
&+K\|f\|^2_{p,\gamma}\tau(1+\|u\|^{\frac{p}{2}}) \big(1+\|u\|^{\frac{p}{2}(1+\tilde q)+(1+\kappa)\gamma}\big)\Big(\tau\sum_{i=0}^{\infty}\big(\rho^{\Delta}(t_{i})\big)^{\gamma}\Big)\nn\\
\leq&~K\|f\|^2_{p,\gamma}(1+\|u\|^{\tilde q(p\tilde q+2(1+\kappa)\gamma)})\leq K\|f\|^2_{p,\gamma}(1+\|u\|^{\tilde p_{\gamma}}),
\end{align*}
which implies
\begin{align}\label{l6.1.2}
\sup_{h\in[0,\tilde h], \tau\in(0,\tilde \tau]}\sup_{u\in E^h}\frac{|H^{\Delta}(u)|}{1+\|u\|^{\tilde p_{\gamma}}}\leq K\|f\|_{p,\gamma}^{2}.
\end{align}
By the definition of $H^{\Delta},$ we deduce that for any $u_1,u_2\in E^h$,
\begin{align}\label{l6.1.11}
H^{\Delta}(u_1)-H^{\Delta}(u_2)=\sum_{l=1}^{4}I_{l}(u_1,u_2),
\end{align} 
where
\begin{align*}
I_1(u_1,u_2):=&~\tau^2\big(|f(u_2)-\mu^{\Delta}(f)|^2-|f(u_1)-\mu^{\Delta}(f)|^2\big),\nn\\
I_2(u_1,u_2):=&~\tau^2 \E\Big[\Big|\sum_{i=0}^{\infty}\big(P^{\Delta}_{t_i}f(Y^{u_1,\Delta}_{t_1})-\mu^{\Delta}(f)\big)\Big|^2\Big]\nn\\
&-\tau^2 \E\Big[\Big|\sum_{i=0}^{\infty}\big(P^{\Delta}_{t_i}f(Y^{u_2,\Delta}_{t_1})-\mu^{\Delta}(f)\big)\Big|^2\Big],\nn\\
I_3(u_1,u_2):=&~-\tau^2 \Big|\sum_{i=0}^{\infty}\big(P^{\Delta}_{t_i}f(u_1)-\mu^{\Delta}(f)\big)\Big|^2+\tau^2 \Big|\sum_{i=0}^{\infty}\big(P^{\Delta}_{t_i}f(u_2)-\mu^{\Delta}(f)\big)\Big|^2,\nn\\
I_4(u_1,u_2):=&~2\tau^2\big(f(u_1)-\mu^{\Delta}(f)\big) \sum_{i=0}^{\infty}\big(P^{\Delta}_{t_i}f(u_1)-\mu^{\Delta}(f)\big)\nn\\
&-2\tau^2\big(f(u_2)-\mu^{\Delta}(f)\big) \sum_{i=0}^{\infty}\big(P^{\Delta}_{t_i}f(u_2)-\mu^{\Delta}(f)\big).
\end{align*}
By the definition of $\mathcal C_{p,\gamma}$ and \eqref{trmub}, we have
\begin{align}\label{l6.1.6}
|I_1(u_1,u_2)|
=&~\tau^2|f(u_1)+f(u_2)-2\mu^{\Delta}(f)||f(u_1)-f(u_2)|\nn\\
\leq&~K\|f\|^2_{p,\gamma}\tau^2(1+\|u_1\|^{\frac{p}{2}}+\|u_2\|^{\frac{p}{2}})(1\wedge\|u_1-u_2\|^{\gamma})(1+\|u_1\|^{p}+\|u_2\|^{p})^{\frac{1}{2}}\nn\\
\leq&~ K\|f\|^2_{p,\gamma}\tau^2(1\wedge\|u_1-u_2\|^{\gamma})(1+\|u_1\|^{2p}+\|u_2\|^{2p})^{\frac{1}{2}}.
\end{align}
It follows from \eqref{trmix1} and \eqref{Yattr1} that
\begin{align}\label{l6.1.4}
|I_3(u_1,u_2)|
=&~\tau^2\big|\sum_{i=0}^{\infty}\big(P^{\Delta}_{t_i}f(u_1)-\mu^{\Delta}(f)\big)\nn\\
&+\sum_{i=0}^{\infty}\big(P^{\Delta}_{t_i}f(u_2)-\mu^{\Delta}(f)\big)\big|\big|\sum_{i=0}^{\infty}\big(P^{\Delta}_{t_i}f(u_2)-P^{\Delta}_{t_i}f(u_1)\big)\big|\nn\\
\leq&~K\|f\|^2_{p,\gamma}(1+\|u_1\|^{\frac{p\tilde q}{2}+(1+\kappa)\gamma}+\|u_2\|^{\frac{p\tilde q}{2}+(1+\kappa)\gamma})(\tau\sum_{i=0}^{\infty}(\rho^{\Delta}(t_i))^{\gamma})^2\nn\\
&\times\|u_1-u_2\|^{\gamma}(1+\|u_1\|^{\frac{p\tilde q}{2}+\gamma\kappa}+\|u_2\|^{\frac{p\tilde q}{2}+\gamma\kappa})\nn\\
\leq&~K\|f\|^2_{p,\gamma}\|u_1-u_2\|^{\gamma}(1+\|u_1\|^{p\tilde q+(1+2\kappa)\gamma}+\|u_2\|^{p\tilde q+(1+2\kappa)\gamma}).
\end{align}
Noting
\begin{align*}
I_2(u_1,u_2)=-\E\big[I_3(Y^{u_1,\Delta}_{t_1},Y^{u_2,\Delta}_{t_1})\big],
\end{align*}
by \eqref{l6.1.4}, $2p\tilde q+2(1+2\kappa)\gamma\leq q$, the  H\"older inequality, and Assumption \ref{a2}, we obtain
\begin{align}\label{l6.1.10}
&|I_2(u_1,u_2)|\nn\\
\leq{}&K\|f\|^2_{p,\gamma}\E\Big[\|Y^{u_1,\Delta}_{t_1}-Y^{u_2,\Delta}_{t_1}\|^{\gamma}(1+\|Y^{u_1,\Delta}_{t_1}\|^{p\tilde q+(1+2\kappa)\gamma}+\|Y^{u_2,\Delta}_{t_1}\|^{p\tilde q+(1+2\kappa)\gamma})\Big]\nn\\
\leq{}&K\|f\|^2_{p,\gamma}\big(\E[\|Y^{u_1,\Delta}_{t_1}-Y^{u_2,\Delta}_{t_1}\|^2]\big)^{\frac{\gamma}{2}}
\nn\\
&\times \Big(\E\big[1+\|Y^{u_1,\Delta}_{t_1}\|^{2p\tilde q+2(1+2\kappa)\gamma}+\|Y^{u_2,\Delta}_{t_1}\|^{2p\tilde q+2(1+2\kappa)\gamma}\big]\Big)^{\frac{1}{2}}\nn\\
\leq{}&K\|f\|^2_{p,\gamma}\|u_1-u_2\|^{\gamma}\big(1+\|u_1\|^{\tilde q(p\tilde q+(1+2\kappa)\gamma)+\gamma\kappa}+\|u_2\|^{\tilde q(p\tilde q+(1+2\kappa)\gamma)+\gamma\kappa}\big).
\end{align}
The similar arguments as those of $I_1$ and $I_3$ lead to
\begin{align}\label{l6.1.9}
|I_4(u_1,u_2)|\leq &~2\tau^2|f(u_1)-f(u_2)||\sum_{i=0}^{\infty}\big(P^{\Delta}_{t_i}f(u_1)-\mu^{\Delta}(f)\big)|\nn\\
&+2\tau^2|f(u_2)-\mu^{\Delta}(f)||\sum_{i=0}^{\infty}\big(P^{\Delta}_{t_i}f(u_1)-P^{\Delta}_{t_i}f(u_2)\big)|\nn\\
\leq&~K\|f\|^{2}_{p,\gamma}\tau\|u_1-u_2\|^{\gamma}(1+\|u_1\|^{\frac{p}{2}(1+\tilde q)+(1+\kappa)\gamma}+\|u_2\|^{\frac{p}{2}(1+\tilde q)+(1+\kappa)\gamma}).
\end{align}
According to $\tilde q\geq 1$ and $p\geq1,$ one has
$p\tilde q^2\geq \frac{p}{2}(1+\tilde q),$
which implies
$$\tilde p_{\gamma}:=\tilde q\big(p\tilde q+(2+3\kappa)\gamma\big)\geq\frac{p}{2}(1+\tilde q)+(1+\kappa)\gamma.$$
Plugging \eqref{l6.1.6}--\eqref{l6.1.9} into \eqref{l6.1.11} deduces
\begin{align*}
\sup_{h\in[0,\tilde h], \tau\in(0,\tilde \tau]}\sup_{\substack{u_1,u_2\in E^h\\u_1\neq u_2}}\frac{|H^{\Delta}(u_1)-H^{\Delta}(u_2)|}{\|u_1-u_2\|^{\gamma}(1+\|u_1\|^{2\tilde p_{\gamma}}+\|u_2\|^{2\tilde p_{\gamma}})^{\frac{1}{2}}}\leq K\|f\|^2_{p,\gamma}.
\end{align*}
This, along with \eqref{l6.1.2}  implies that
\begin{align*}
&\sup_{h\in[0,\tilde h], \tau\in(0,\tilde \tau]}\sup_{\substack{u_1,u_2\in E^h\\u_1\neq u_2}}\frac{|H^{\Delta}(u_1)-H^{\Delta}(u_2)|}{(1\wedge\|u_1-u_2\|^{\gamma})(1+\|u_1\|^{2\tilde p_{\gamma}}+\|u_2\|^{2\tilde p_{\gamma}})^{\frac{1}{2}}}\nn\\
\leq&~ K\|f\|^2_{p,\gamma}+ \sup_{h\in[0,\tilde h], \tau\in(0,\tilde \tau]}\sup_{\substack{u_1,u_2\in E^h\\u_1\neq u_2}}\frac{|H^{\Delta}(u_1)-H^{\Delta}(u_2)|\textbf 1_{\{\|u_1-u_2\|>1\}}}{(1+\|u_1\|^{2\tilde p_{\gamma}}+\|u_2\|^{2\tilde p_{\gamma}})^{\frac{1}{2}}}\nn\\
\leq&~K\|f\|^2_{p,\gamma}+ \sup_{h\in[0,\tilde h], \tau\in(0,\tilde \tau]}\sup_{\substack{u_1,u_2\in E^h\\u_1\neq u_2}}\frac{|H^{\Delta}(u_1)|+|H^{\Delta}(u_2)|}{(1+\|u_1\|^{2\tilde p_{\gamma}}+\|u_2\|^{2\tilde p_{\gamma}})^{\frac{1}{2}}}\nn\\
\leq &~K\|f\|^2_{p,\gamma}.
\end{align*}
The proof is completed.
\end{proof}

\begin{proof}[\textbf{Proof of Proposition \ref{p6.2}}.]
\textit{Step 1: Proof of \eqref{p6.2.1}}. 
The condition of $p$ in Proposition \ref{p6.2} implies that the one of Proposition \ref{p6.1} (ii) holds and then $H^{\Delta}$ is well-defined.
Moreover,
$H^{\Delta}$ can be rewritten as 
\begin{align*}
H^{\Delta}(u)=&-\tau^2|f(u)-\mu^{\Delta}(f)|^2+\tau^2 P^{\Delta}_{t_1}\Big|\sum_{i=0}^{\infty}\big(P^{\Delta}_{t_i}f-\mu^{\Delta}(f)\big)\Big|^2(u)-\tau^2 \Big|\sum_{i=0}^{\infty}\big(P^{\Delta}_{t_i}f(u)\nn\\
&-\mu^{\Delta}(f)\big)\Big|^2+2\tau^2\big(f(u)-\mu^{\Delta}(f)\big) \sum_{i=0}^{\infty}\big(P^{\Delta}_{t_i}f(u)-\mu^{\Delta}(f)\big),
\end{align*}
which together with $(P^{\Delta}_{t_1})^*\mu^{\Delta}=\mu^{\Delta}$ gives 
\begin{align*}
\mu^{\Delta}(H^{\Delta})
=-\tau^{2}\mu^{\Delta}\big(|f-\mu^{\Delta}(f)|^2\big)+2\tau\mu^{\Delta}\Big(\big(f-\mu^{\Delta}(f)\big)\sum_{i=0}^{\infty}\big(P^{\Delta}_{t_i}f-\mu^{\Delta}(f)\big)\tau\Big).
\end{align*}

\textit{Step 2: Proof of $|\mu^{\Delta}(H^{\Delta})|\leq K\tau\|f\|_{p,\gamma}^2$}. According to the Young inequality, we arrive at
\begin{align}\label{p6.2.8}
|\mu^{\Delta}(H^{\Delta})|\leq 2\tau\mu^{\Delta}\big(|f-\mu^{\Delta}(f)|^2\big)+\tau\mu^{\Delta}\Big(\big(\sum_{i=0}^{\infty}\big(P^{\Delta}_{t_i}f-\mu^{\Delta}(f)\big)\big)^2\tau^2\Big).
\end{align}
The fact $f\in\mathcal C_{p,\gamma}$ implies that $f^2\in\mathcal C_{2p,\gamma}$ and 
$\|f^2\|_{2p,\gamma}\leq K\|f\|_{p,\gamma}^2$.
Together with \eqref{trmub}, we derive
\begin{align}\label{p6.2.3}
&\quad \mu^{\Delta}(|f-\mu^{\Delta}(f)|^2)
=\mu^{\Delta}(f^2)-|\mu^{\Delta}(f)|^2\nn\\
&\leq K\|f\|_{p,\gamma}^2\mu^{\Delta}(1+\|\cdot\|^{p})+K\big(\|f\|_{p,\gamma}\mu^{\Delta}(1+\|\cdot\|^{\frac{p}{2}})\big)^2\leq K\|f\|_{p,\gamma}^2.
\end{align}
It follows from \eqref{trmix1}, $p\tilde q+2(1+\kappa)\gamma\leq q$, and \eqref{trmub} that
\begin{align}\label{p6.2.9}
&\mu^{\Delta}\Big((\sum_{i=0}^{\infty}\big(P^{\Delta}_{t_i}f-\mu^{\Delta}(f)\big)\big)^2\tau^2\Big)\nn\\
\leq&~\mu^{\Delta}\Big(
K\|f\|^2_{p,\gamma}(1+\|\cdot\|^{p\tilde q+2(1+\kappa)\gamma})\big(\sum_{i=0}^{\infty}\big(\rho^{\Delta}(t_i)\big)^{\gamma}\tau\big)^{2}\Big)\nn\\
\leq&~K\|f\|^2_{p,\gamma}\mu^{\Delta}
(1+\|\cdot\|^{p\tilde q+2(1+\kappa)\gamma})\leq K\|f\|^2_{p,\gamma}.
\end{align}
Inserting \eqref{p6.2.3} and \eqref{p6.2.9} into \eqref{p6.2.8}
yields
$|\mu^{\Delta}(H^{\Delta})|\leq K\tau\|f\|^2_{p,\gamma}.$

\textit{Step 3: Proof of 
$
\lim_{|\Delta|\rightarrow 0}\frac{\mu^{\Delta}(H^{\Delta})}{\tau}=v^2.
$}
Without loss of generality, below we  assume that $\mu(f)=0$. Otherwise, we let $\tilde f:=f-\mu(f)$ and consider $\tilde f$ instead of $f$.
The condition on $p$ in Proposition \ref{p6.2} leads to
$\frac{p}{2}(1+\tilde q\vee\tilde r)+(1+\beta\vee\kappa)\gamma\leq r$, which along with $f\in\mathcal C_{p,\gamma}$, \eqref{trmix1}, and \eqref{Xmix} implies that
$$\mu\Big(f\sum_{i=0}^{\infty}\big(P^{\Delta}_{t_i}f-\mu^{\Delta}(f)\big)\tau\textbf 1_{\{E^h\}}\Big)<\infty\quad\text{and}\quad \mu\Big(f\int_{0}^{\infty}P^{}_{t}f\mathrm dt\Big)<\infty.$$
Together with \eqref{p6.2.1} we arrive at
\begin{align*}
&\Big|\frac{\mu^{\Delta}(H^{\Delta})}{\tau}-v^2\Big|\nn\\
\leq&~\tau\mu^{\Delta}\big(|f-\mu^{\Delta}(f)|^2\big)+2|\mu^{\Delta}(f)|\mu^{\Delta}\Big(\sum_{i=0}^{\infty}\big|P^{\Delta}_{t_i}f-\mu^{\Delta}(f)\big|\tau\Big)\nn\\
&+2\Big|\mu^{\Delta}\Big(f\sum_{i=0}^{\infty}\big(P^{\Delta}_{t_i}f-\mu^{\Delta}(f)\big)\tau\textbf 1_{\{E^h\}}\Big)-\mu\Big(f\sum_{i=0}^{\infty}\big(P^{\Delta}_{t_i}f-\mu^{\Delta}(f)\big)\tau\textbf 1_{\{E^h\}}\Big)\Big|\nn\\
&+2\Big|\mu\Big(f\sum_{i=0}^{\infty}\big(P^{\Delta}_{t_i}f-\mu^{\Delta}(f)\big)\tau\textbf 1_{\{E^h\}}\Big)-\mu\Big(f\int_{0}^{\infty}P^{}_{t}f\mathrm dt\Big)\Big|=:\sum_{i=1}^4II_i.
\end{align*}
Due to \eqref{p6.2.3}, we have $II_1\to0$ as $|\Delta|\to0.$
By \eqref{trmub}, \eqref{trmix1}, and \eqref{t3.1},  we arrive at
\begin{align}\label{p6.2.5}
II_2\leq&~K|\mu^{\Delta}(f)|\mu^{\Delta}\Big(\|f\|_{p,\gamma}(1+\|\cdot\|^{\frac{p\tilde q}{2}+(1+\kappa)\gamma})\sum_{i=0}^{\infty}\big(\rho^{\Delta}(t_i)\big)^{\gamma}\tau\Big)\nn\\
\leq&~K\|f\|_{p,\gamma}|\mu^{\Delta}(f)-\mu(f)|
\leq~K\|f\|_{p,\gamma}^2|\Delta^{\a}|^{\gamma},
\end{align}
where we used $\mu(f)=0$.
Denote
$$F^{\Delta}:=f\sum_{i=0}^{\infty}\big(P^{\Delta}_{t_i}f-\mu^{\Delta}(f)\big)\tau\textbf 1_{\{E^h\}}.$$
Using the similar techniques as the proof of Proposition \ref{p6.1}, we obtain 
$$F^{\Delta}\in \mathcal C_{p(1+\tilde q)+2(1+\kappa)\gamma,\gamma}~~~\mbox{and}~~~\sup_{h\in[0,\tilde h], \tau\in(0,\tilde \tau]}\|F^{\Delta}\|_{p(1+\tilde q)+2(1+\kappa)\gamma,\gamma}\leq K \|f\|_{p,\gamma}^2.$$
Then by $p(1+\tilde q)+2(1+\kappa)\gamma\leq r\wedge q$ and \eqref{t3.1}, we deduce 
\begin{align}\label{p6.2.6}
II_3\leq&~K\|F^{\Delta}\|_{p(1+\tilde q)+2(1+\kappa)\gamma,\gamma}|\Delta^{\a}|^{\gamma}
\leq K\|f\|_{p,\gamma}^2|\Delta^{\a}|^{\gamma}.
\end{align}
For the term $II_4,$ we have
\begin{align*}
II_4\leq&~2\mu\Big(\textbf 1_{\{E^h\}}\int_{0}^{\infty}\big|f\big(P^{\Delta}_{t}f-P_{t}f\big)+f\big(\mu(f)-\mu^{\Delta}(f)\big)\big|\mathrm dt\Big)\nn\\
&+2\mu\Big(f\int_{0}^{\infty}P^{}_{t}f\mathrm dt\textbf 1_{\{E\backslash E^h\}}\Big),
\end{align*}
where $E\backslash E^h$ represents the complementary set of $E^h$ and we used $\mu(f)=0$.
It follows from \eqref{5.4} and Assumptions \ref{a2}--\ref{a3} that for any $u\in E^h$, 
\begin{align}\label{add_eq2}
|P_{t}f(u)-P^{\Delta}_{t}f(u)|
\leq&\;\|f\|_{p,\gamma}\big(1+\E[\|X^{u}_{t}\|^{p}]+\E[\|Y^{u}_{t}\|^{p}]\big)^{\frac{1}{2}}\big(\mathbb W_{2}(\mu^{u}_{t},\mu^{u,\Delta}_{t})\big)^{\gamma}\nn\\
\leq&\;\|f\|_{p,\gamma}\big(1+\E[\|X^{u}_{t}\|^{p}]+\E[\|Y^{u}_{t}\|^{p}]\big)^{\frac{1}{2}}\Big(\E\big[\|X^{u}_{t}-Y^{u}_{t}\|^2\big]\Big)^{\frac{\gamma}{2}}\nn\\
\leq&\;K\|f\|_{p,\gamma}(1+\|u\|^{(\frac{p}{2}+\gamma)(\tilde r\vee\tilde q)})|\Delta^{\a}|^{\gamma}.
\end{align}
This, along with \eqref{t3.1} implies that
\begin{align*}
&\textbf 1_{\{u\in E^h\}}\big|f(x)\big(P^{\Delta}_{t}f(u)-P_{t}f(u)\big)+f(u)\big(\mu(f)-\mu^{\Delta}(f)\big)\big|\nn\\
\leq&~\textbf 1_{\{u\in E^h\}}|f(u)|\Big(\big|P^{\Delta}_{t}f(u)-P_{t}f(u)\big|+K\|f\|_{p,\gamma}|\Delta^{\a}|^{\gamma}\Big)\nn\\
\leq &~K\|f\|^2_{p,\gamma}(1+\|u\|^{(\frac{p}{2}+\gamma)(\tilde r\vee\tilde q)+\frac{p}{2}}\textbf 1_{\{u\in E^h\}})|\Delta^{\a}|^{\gamma}.
\end{align*}
Applying the Fubini theorem and  the Fatou lemma leads to
\begin{align*}
&\varlimsup_{|\Delta|\rightarrow 0}\mu\Big(\textbf 1_{\{E^h\}}\int_{0}^{\infty}\big|f\big(P^{\Delta}_{t}f-P_{t}f\big)-f\big(\mu^{\Delta}(f)-\mu(f)\big)\big|\mathrm dt\Big)\nn\\
\leq&\int_{0}^{\infty}\mu\Big(\varlimsup_{|\Delta|\rightarrow 0}\textbf 1_{\{E^h\}}\big|f\big(P^{\Delta}_{t}f-P_{t}f\big)+f\big(\mu(f)-\mu^{\Delta}(f)\big)\big|\Big)\mathrm dt=0,
\end{align*}
which together with $\lim_{|\Delta|\rightarrow0}\mu\big(f\int_{0}^{\infty}P^{}_{t}f\mathrm dt\textbf 1_{\{E\backslash E^h\}}\big)=0$
gives $II_4\to0$ as $|\Delta|\to0.$
Combining terms $II_i,i=1,2,3,4$ completes the proof.
\end{proof}

\begin{appendix}\label{sec5}
\section{Proof of Proposition \ref{Limit_exact}}\label{appen_exact}
Let Assumption \textup{\ref{a1}} hold. Similar to the proof of Proposition \ref{prop_IM}, it can be shown that the time-homogeneous Markov process $\{X^x_t\}_{t\ge 0}$ admits a unique invariant measure $\mu,$ whose proof is omitted. First, we give some properties of $\mu$.
  It is straightforward to see that
\begin{align}\label{mubound}
\mu\big(\|\cdot\|^r\big)
=&{}\lim_{t\rightarrow\infty}\int_{E} \|u\|^r\mathrm d\mu^{\textbf 0}_t(u)
\leq \sup_{t\geq0}\E [\|X^{\textbf 0}_t\|^r]\leq L_{1}.
\end{align}
Let  $1\leq p\leq r$ and $\gamma\in(0,1]$.
By \eqref{5.4}, \eqref{W}, \eqref{mubound}, and Assumption \textup{\ref{a1}} \textup{(i)}, we obtain that for any $x\in E$ and $g\in \mathcal C_{p,\gamma}$,
\begin{align}\label{Xmix}
|P_{t}g(x)-\mu(g)|
\leq&~ K\|g\|_{p,\gamma} (1+\E[\|X^{x}_{t}\|^{p}]+\mu(\|\cdot\|^p))^{\frac{1}{2}}\big(\mathbb W_{2}(\mu^{x}_t,\mu)\big)^{\gamma}\nn\\
\leq&~ K\|g\|_{p,\gamma} (1+\|x\|^{\frac{p\tilde r}{2}})\big(\mathbb W_{2}(\mu^{x}_t,\mu)\big)^{\gamma}\nn\\
\leq&~ K\|g\|_{p,\gamma}(1+\|x\|^{\frac{p\tilde r}{2}+(1+\beta)\gamma})\rho^{\gamma}(t),
\end{align}
which implies that $P_t$ has strong mixing for $g\in\mathcal C_{p,\gamma}$. 
Moreover, using the same technique as \eqref{Xmix},  we obtain that for any $x,y\in E$,
\begin{align}\label{Xattr}
|P_{t}g(x)-P_{t}g(y)|
\leq&{} K\|g\|_{p,\gamma}\|x-y\|^{\gamma}(1+\|x\|^{\frac{p\tilde r}{2}+\gamma\beta}+\|y\|^{\frac{p\tilde r}{2}+\gamma\beta})\rho^{\gamma}(t).
\end{align}

The following lemma plays an important role in investigating the strong LLN and the CLT of the Markov process, whose proof is based on \cite[Proposition 2.6]{Shirikyan} and the strong mixing \eqref{Xmix}.
\begin{lemma}\label{exactle}
Let Assumption \textup{\ref{a1}} hold, $\gamma\in[\gamma_1,1],$
 and $\bar\varrho\in \mathbb N,\,p
\ge 1$ satisfy 
$\bar\varrho\big(\frac{p}{2}(1+\tilde r)+(1+\beta)\gamma\big)\leq r.$
Then for any $f\in\mathcal C_{p,\gamma}$, 
$x\in E$, 
we have 
\begin{align*}
\E \Big[\Big|\frac{1}{T}\int_0^Tf(X^x_s)\mathrm ds-\mu(f)\Big|^{2\bar\varrho}\Big]\leq K\|f\|^{2\varrho}_{p,\gamma}(1+\|x\|^{\bar\varrho\tilde r(\frac{p}{2}(1+\tilde r)+(1+\beta)\gamma)})T^{-\bar\varrho} ~~~\forall~T>0.
\end{align*}
\end{lemma}

Next, with this lemma in hand, we give the proof of the strong LLN and the CLT in Proposition \ref{Limit_exact}.
\begin{proof}
\textbf{Proof of the strong LLN.}
For any $T\geq 1$, one has
\begin{align}\label{10072}
\Big|\frac1T\int_0^Tf(X^x_s)\mathrm ds-\mu(f)\Big|&\leq\Big|\frac1T\int_{\lfloor T\rfloor}^Tf(X^x_s)\mathrm ds\Big|
 +\Big|\frac{1}{\lfloor T\rfloor}\int_0^{\lfloor T\rfloor}f(X^x_s)\mathrm ds-\mu(f)\Big|\nn\\
 &\quad +\Big|\Big(\frac 1T-\frac{1}{\lfloor T\rfloor}\Big)\int_0^{\lfloor T\rfloor}f(X^x_s)\mathrm ds\Big|.
\end{align}
We first estimate the term $\big|\frac1T\int_{\lfloor T\rfloor}^Tf(X^x_s)\mathrm ds\big|$.
For any $k\in\mathbb N_+$, denote $$\mathscr D_k^{x}:=\Big\{\omega\in\Omega:\sup_{T\in[k,k+1)}\Big|\int_k^Tf(X^x_s)\mathrm ds\Big|>\|f\|_{p,\gamma}k^{\frac{3}{4}}\Big\}.$$
It follows from the Chebyshev inequality, the H\"older inequality, and Assumption \ref{a1} (i) that
\begin{align*}
\sum_{k=1}^{\infty}\mathbb P(\mathscr D^{x}_k)
&\leq \sum_{k=1}^{\infty}\frac{\E\Big[\sup_{T\in[k, k+1)}\big|\int_{k}^{T}f(X^{x}_s)\mathrm ds\big|^2\Big]}{\|f\|_{p,\gamma}^2k^{\frac{3}{2}}}\nn\\
&\leq \sum_{k=1}^{\infty}\frac{\E\Big[\sup_{T\in[k, k+1)}\int_{k}^{T}|f(X^{x}_s)|^2\mathrm ds(T-k)\Big]}{\|f\|_{p,\gamma}^2k^{\frac{3}{2}}}\nn\\
&\leq K\sum_{k=1}^{\infty}  k^{-\frac{3}{2}} \int_{k}^{k+1}(1+\E[\|X^x_s\|^{p}])\mathrm ds\nn\\
&\leq K(1+\|x\|^{\tilde rp})\sum_{k=1}^{\infty}k^{-\frac{3}{2}}\leq K(1+\|x\|^{\tilde rp}).
\end{align*}
Define $\tilde N^x(\omega):=\inf\{j\in\mathbb N_+:\sup_{T\in[k,k+1)}|\int_k^Tf(X^x_s)\mathrm ds|\leq\|f\|_{p,\gamma}k^{\frac{3}{4}}\text{ for all }k\ge j\}$. Applying the Borel--Cantelli lemma leads to  that for a.s. $\omega\in\Omega,$ $\tilde N^x(\omega)<\infty$, and   when $k\ge \tilde N^x(\omega),$
\begin{align*}
\sup_{T\in[k,k+1)}\Big|\int_k^Tf(X^x_s)\mathrm ds\Big|\leq\|f\|_{p,\gamma}k^{\frac{3}{4}},
\end{align*}
which implies 
\begin{align}\label{10073}
\Big|\frac1 T\int_{\lfloor T\rfloor}^T f(X_{s}^x)\mathrm ds\Big|\overset{a.s.}\longrightarrow0\quad\text{ as }T\to\infty.
\end{align}
Next, we estimate the second  term $\Big|\frac{1}{\lfloor T\rfloor}\int_0^{\lfloor T\rfloor}f(X^x_s)\mathrm ds-\mu(f)\Big|$.
For any $N\in\mathbb N_+$ and $\delta\in(0, \frac14)$, denote $\mathcal A^x_N:=\{\omega\in\Omega:|\frac1N\int_0^Nf(X^x_s)\mathrm ds-\mu(f)|>\|f\|_{p,\gamma}N^{-\delta}\}.$ This, together with Lemma \ref{exactle} for $\bar \varrho=2$ gives 
\begin{align*}
\sum_{N=1}^{\infty}\mathbb P(\mathcal A^x_N)\leq K(1+\|x\|^{\tilde r(p(1+\tilde r)+2(1+\beta)\gamma)})\sum_{N=1}^{\infty}N^{4\delta-2}\leq K(1+\|x\|^{\tilde r(p(1+\tilde r)+2(1+\beta)\gamma)}).
\end{align*}
Define $\mathcal K^x(\omega):=\inf\{j\in\mathbb N_+:|\frac1N\int_0^Nf(X^x_s)\mathrm ds-\mu(f)|(\omega)\leq \|f\|_{p,\gamma}N^{-\delta}\text{ for all }N\ge j\}$. Applying the Borel--Cantelli lemma yields that for a.s. $\omega\in\Omega,$ when $N\ge \mathcal K^x(\omega),$
\begin{align}\label{10071}
\Big|\frac1N\int_0^Nf(X^N_s)\mathrm ds-\mu(f)\Big|\leq \|f\|_{p,\gamma}N^{-\delta},
\end{align}
which implies 
\begin{align}\label{10075}
\Big|\frac{1}{\lfloor T\rfloor}\int_0^{\lfloor T\rfloor}f(X^x_s)\mathrm ds-\mu(f)\Big|\overset{a.s.}\longrightarrow0\text{ as }T\to\infty.
\end{align}
For the third term $\Big|\Big(\frac 1T-\frac{1}{\lfloor T\rfloor}\Big)\int_0^{\lfloor T\rfloor}f(X^x_s)\mathrm ds\Big|$, by \eqref{10071} we deduce that 
\begin{align}\label{10074}
\Big|\Big(\frac 1T-\frac{1}{\lfloor T\rfloor}\Big)\int_0^{\lfloor T\rfloor}f(X^x_s)\mathrm ds\Big|
&\leq\frac{1}{T\lfloor T\rfloor}\Big|\int_0^{\lfloor T\rfloor}f(X^x_s)\mathrm ds\Big|\nn\\
&\leq \frac{1}{T}\big(|\mu(f)|+\|f\|_{p,\gamma}\lfloor T\rfloor^{-\delta}\big)
\overset{a.s.}\longrightarrow0\text{ as }T\to\infty.
\end{align}
Therefore, plugging \eqref{10073}, \eqref{10075}, and \eqref{10074} into \eqref{10072} we derive  the strong LLN of $\{X^x_t\}_{t\ge 0}$.

\textbf{Proof of the CLT.}
For any $T\geq 1$, note that
\begin{align*}
&\frac{1}{\sqrt T}\int_0^T\big(f(X^x_s)-\mu(f)\big)\mathrm ds\nn\\
=&~\frac{1}{\sqrt T}\int_{\lfloor T\rfloor}^T\big(f(X^x_s)-\mu(f)\big)\mathrm ds
 +\Big(\frac {1}{\sqrt{T}}-\frac{1}{\sqrt{\lfloor T\rfloor}}\Big)\int_0^{\lfloor T\rfloor}\big(f(X^x_s)-\mu(f)\big)\mathrm ds\nn\\
 &~+\frac{1}{\sqrt {\lfloor T\rfloor}}\int_0^{\lfloor T\rfloor}\big(f(X^x_s)-\mu(f)\big)\mathrm ds.
\end{align*}
It follows from \eqref{mubound} and Assumption \ref{a1} (i) that
\begin{align*}
\E\Big[\Big|\frac{1}{\sqrt T}\int_{\lfloor T\rfloor}^T\big(f(X^x_s)-\mu(f)\big)\mathrm ds\Big|\Big]
\leq &~\frac{1}{\sqrt T}\int_{\lfloor T\rfloor}^T K \|f\|_{p,\gamma}\big(1+\sup_{t\geq 0}\E[\|X^x_t\|^{\frac{p}{2}}]\big)\mathrm ds\nn\\
\leq&~ K(1+\|x\|^{\frac{\tilde rp}{2}})\frac{1}{\sqrt T}\rightarrow 0\quad\text{as }T\rightarrow \infty.
\end{align*}
Making use of  Lemma  \ref{exactle} and $(\sqrt T-\sqrt{\lfloor T\rfloor})^2\leq T-\lfloor T\rfloor\leq 1$, we deduce  
 \begin{align*}
 &\E\Big[\Big|\big(\frac {1}{\sqrt{T}}-\frac{1}{\sqrt{\lfloor T\rfloor}}\big)\int_0^{\lfloor T\rfloor}\big(f(X^x_s)-\mu(f)\big)\mathrm ds\Big|^2\Big]\nn\\
 \leq&\frac{\lfloor T\rfloor}{T}\E\Big[\Big|\frac{1}{\lfloor T\rfloor}\int_0^{\lfloor T\rfloor}\big(f(X^x_s)-\mu(f)\big)\mathrm ds\Big|^2\Big]\nn\\
 \leq &~K\|f\|^2_{p,\gamma} (1+\|x\|^{\tilde r(\frac{p}{2}(1+\tilde r)+(1+\beta)\gamma)})T^{-1}  \rightarrow 0\quad\text{as } T\rightarrow \infty.
 \end{align*}
Those imply that it suffices to prove the CLT for the case of  $T=N\in\mathbb N_+$.   We have the decomposition
\begin{align*}
\frac{1}{\sqrt {N}}\int_{0}^{N}\big(f(X^{x}_{t})-\mu(f)\big)\mathrm dt =\frac{1}{\sqrt N}\mathscr M^{x}_{N}+\frac{1}{\sqrt {N}}\mathscr R^{x}_{N},
\end{align*}
where 
\begin{align*}
\mathscr M^{x}_{N}:=&\int_{0}^{\infty}\Big(\E\big[f(X_{t}^{x})|\mathcal F_{N}\big]-\mu(f)\Big)\mathrm dt-\int_{0}^{\infty}\Big(\E\big[f(X_{t}^{x})|\mathcal F_{0}\big]-\mu(f)\Big)\mathrm dt,\nn\\
\mathscr R^{x}_{N}:=&-\int_{N}^{\infty}\Big(\E\big[f(X_{t}^{x})|\mathcal F_{N}\big]-\mu(f)\Big)\mathrm dt+\int_{0}^{\infty}\Big(\E\big[f(X_{t}^{x})|\mathcal F_{0}\big]-\mu(f)\Big)\mathrm dt.
\end{align*}
By a similar proof to Proposition \ref{Prop4.1}, it can be shown that  $\{\mathscr M^{x}_{N}\}_{N\in\mathbb N}$ is an $\{\mathcal F_{N}\}_{N\in\mathbb N}$-adapted  martingale with $\mathscr M^x_0=0.$ Define the martingale difference sequence of $\{\mathscr M^{x}_{N}\}_{N\in\mathbb N}$ by
\begin{align*}
\mathcal Z^{x}_{N}:=\mathscr M^{x}_{N}-\mathscr M^{x}_{N-1}~~~\forall N\in\mathbb N_+,
~~~~\mathcal Z^{x}_{0}=0.
\end{align*}
Since
\begin{align}
\mathcal Z^{x}_{N+1}
&=\int_{N}^{N+1}\Big(f(X^{x}_t)-\mu(f)\Big)\mathrm dt+\int_{N+1}^{\infty}\Big(\E[f(X^{x}_t)|\mathcal F_{N+1}\big]-\mu(f)\Big)\mathrm dt\nn\\
&-\int_{N}^{\infty}\Big(\E[f(X^{x}_t)|\mathcal F_{N}\big]-\mu(f)\Big)\mathrm dt\nn\\
&=\int_{N}^{N+1}\Big(f(X^{x}_t)-\mu(f)\Big)\mathrm dt+\int_{0}^{\infty}\Big(P_{t}f(X^{x}_{N+1})-\mu(f)\Big)\mathrm dt\nn\\
&-\int_{0}^{\infty}\Big(P_{t}f(X^{x}_N)-\mu(f)\Big)\mathrm dt,
\end{align}
we deduce that $\sup_{N\in\mathbb N}\mathbb E[|\mathcal Z^x_{N}|^{c}]\leq K$ when $c(\frac{p\tilde r}{2}+(1+\beta)\gamma)\leq r$, whose proof is similar to Proposition \ref{p6.1} (\romannumeral1). Moreover,  
\begin{align}
\E\big[|\mathcal Z^{x}_{N+1}|^2|\mathcal F_{N}\big]
={}&\E\Big[\big|\int_{0}^{1}\big(f(X^{y}_t)-\mu(f)\big)\mathrm dt+\int_{0}^{\infty}\big(P_{t}f(X_{1}^{y})-\mu(f)\big)\mathrm dt\nn\\
&-\int_{0}^{\infty}\big(P_{t}f(y)-\mu(f)\big)\mathrm dt\big|^2\Big]\Big|_{y=X^{x}_{N}}
=: H(X^{x}_{N}).
\end{align}
Without loss of generality, below we  assume that $\mu(f)=0$. Otherwise, we let $\tilde f:=f-\mu(f)$ and consider $\tilde f$ instead of $f$. 

\textit{Claim 1: $H\in\mathcal C_{2\bar p_{\gamma},\gamma}$ and $\|H\|_{2\tilde p_{\gamma},\gamma}\leq K\|f\|^2_{p,\gamma}$ with $\bar p_{\gamma}:=\tilde r(p\tilde r+(2+3\beta)\gamma)$.} 
It can be calculated that 
\begin{align}\label{eqH}
H(u)=&~\E\Big[\big|\int_{0}^{1}f(X^{u}_t)\mathrm dt\big|^2\Big]+\E\Big[\big|\int_{0}^{\infty}P_{t}f(X_{1}^{u})\mathrm dt\big|^2\Big]+\big|\int_{0}^{\infty}P_{t}f(u)\mathrm dt\big|^2\nn\\
&+2\E\Big[\int_{0}^{1}f(X^{u}_t)\mathrm dt \int_{0}^{\infty}P_{t}f(X_{1}^{u})\mathrm dt\Big]-2\E\Big[\int_{0}^{1}f(X^{u}_t)\mathrm dt\Big] \int_{0}^{\infty}P_{t}f(u)\mathrm dt\nn\\
&-2\E\Big[\int_{0}^{\infty}P_{t}f(X_{1}^{u})\mathrm dt\Big]\int_{0}^{\infty}P_{t}f(u)\mathrm dt.
\end{align}
By the property of the semigroup $P_t$, we have  
\begin{align}
&\E\Big[\int_{0}^{1}f(X^{u}_t)\mathrm dt\Big] \int_{0}^{\infty}P_{t}f(u)\mathrm dt
=\int_{0}^{1}P_{t}f(u)\mathrm dt \int_{0}^{\infty}P_{t}f(u)\mathrm dt\label{eq6.36}
\end{align}
and 
\begin{align}\label{eq6.37}
&\E\Big[\int_{0}^{\infty}P_{t}f(X_{1}^{u})\mathrm dt\Big]\int_{0}^{\infty}P_{t}f(u)\mathrm dt
=\int_{0}^{\infty}P_{t+1}f(u)\mathrm dt\int_{0}^{\infty}P_{t}f(u)\mathrm dt\nn\\
=&~\big(\int_{0}^{\infty}P_{t}f(u)\mathrm dt\big)^2-\int_{0}^{1}P_{t}f(u)\mathrm dt\int_{0}^{\infty}P_{t}f(u)\mathrm dt.
\end{align}
Combining  the property of the conditional expectation  gives 
\begin{align}
&\quad \E\Big[\big|\int_{0}^{1}f(X^{u}_t)\mathrm dt\big|^2\Big]
=2\E\Big[\int_{0}^{1}f(X^{u}_t)\int_{t}^{1}f(X^{u}_s)\mathrm ds\mathrm dt\Big]\nn\\
&=2\int_{0}^{1}\E\Big[f(X^{u}_t)\E\Big[\int_{t}^{1}f(X^{u}_s)\mathrm ds\Big|\mathcal F_{t}\Big]\Big]\mathrm dt\nn\\
&=2\int_{0}^{1}\int_{0}^{1-t}P_{t}\big(fP_{s} f\big)(u)\mathrm ds\mathrm dt\label{eq6.34}
\end{align}
and 
\begin{align}
&\E\Big[\int_{0}^{1}f(X^{u}_t)\mathrm dt \int_{0}^{\infty}P_{t}f(X_{1}^{u})\mathrm dt\Big]
=\;\int_{0}^{1}\E\Big[f(X^{u}_t)\E\big[\int_{0}^{\infty}P_{s}f(X_{1}^{u})\mathrm ds|\mathcal F_{t}\big]\Big]\mathrm dt\nn\\
=&~\int_{0}^{1}\E\Big[f(X^{u}_t)\int_{0}^{\infty}P_{1-t+s}f(X_{t}^{u})\mathrm ds\Big]\mathrm dt\nn\\
=&~\int_{0}^{1}\E\Big[f(X^{u}_t)\int_{0}^{\infty}P_{s}f(X_{t}^{u})\mathrm ds\Big]\mathrm dt-\int_{0}^{1}\E\Big[f(X^{u}_t)\int_{0}^{1-t}P_{s}f(X_{t}^{u})\mathrm ds\Big]\mathrm dt\nn\\
=&~\int_{0}^{1}\E\Big[f(X^{u}_t)\int_{0}^{\infty}P_{s}f(X_{t}^{u})\mathrm ds\Big]\mathrm dt-\int_{0}^{1}\int_{0}^{1-t}P_{t}\big(fP_{s}f\big)(u)\mathrm ds\mathrm dt.\label{eq6.35}
\end{align}
Plugging \eqref{eq6.36}--\eqref{eq6.35} into  \eqref{eqH} yields 
\begin{align*}
H(u)=&~
\E\Big[\big|\int_{0}^{\infty}P_{t}f(X_{1}^{u})\mathrm dt\big|^2\Big]-\big|\int_{0}^{\infty}P_{t}f(u)\mathrm dt\big|^2+2\E\Big[\int_{0}^{1}f(X^{u}_t) \int_{0}^{\infty}P_{s}f(X_{t}^{u})\mathrm ds\mathrm dt\Big].
\end{align*}
We derive from $\tilde r\geq 1$, $p\tilde r+2(1+\beta)\gamma\leq r$, Assumption \ref{a1} (i), and \eqref{Xmix} that 
\begin{align}
|H(u)|&\leq K\E\Big[\Big|\int_{0}^{\infty}\|f\|_{p,\gamma}(1+\|X^{u}_1\|^{\frac{p\tilde r}{2}+(1+\beta)\gamma})(\rho(t))^{\gamma}\mathrm dt\Big|^2\Big]\nn\\
&+K\Big|\int_{0}^{\infty}\|f\|_{p,\gamma}(1+\|u\|^{\frac{p\tilde r}{2}+(1+\beta)\gamma})(\rho(t))^{\gamma}\mathrm dt\Big|^2\nn\\
&+K\E\Big[\Big|\int_{0}^{1}\int_{0}^{\infty}\|f\|^2_{p,\gamma}(1+\|X^{u}_t\|^{\frac{p}{2}(1+\tilde r)+(1+\beta)\gamma})(\rho(s))^{\gamma}\mathrm ds\mathrm dt\Big]\nn\\
&\leq K\|f\|^2_{p,\gamma}(1+\|u\|^{\tilde r(p\tilde r+2(1+\beta)\gamma)}),
\end{align}
and hence 
\begin{align*}
\sup_{u\in E}\frac{|H(u)|}{1+\|u\|^{\bar p_{\gamma}}}\leq K\|f\|^2_{p,\gamma}.
\end{align*}
For any $u_1,u_2\in E$,
\begin{align*}
H(u_1)-H(u_2)=\sum_{i=1}^{3}\tilde I_{i}(u_1,u_2),
\end{align*}
where
\begin{align*}
 \tilde I_{1}(u_1,u_2)&=\E\Big[\big|\int_{0}^{\infty}P_{t}f(X_{1}^{u_1})\mathrm dt\big|^2\Big]-\E\Big[\big|\int_{0}^{\infty}P_{t}f(X_{1}^{u_2})\mathrm dt\big|^2\Big],\nn\\
 \tilde I_{2}(u_1,u_2)&=-\big|\int_{0}^{\infty}P_{t}f(u_1)\mathrm dt\big|^2+\big|\int_{0}^{\infty}P_{t}f(u_2)\mathrm dt\big|^2,\nn\\
  \tilde I_{3}(u_1,u_2)&=2\E\Big[\int_{0}^{1}f(X^{u_1}_t) \int_{0}^{\infty}P_{s}f(X_{t}^{u_1})\mathrm ds\mathrm dt\Big]-2\E\Big[\int_{0}^{1}f(X^{u_2}_t) \int_{0}^{\infty}P_{s}f(X_{t}^{u_2})\mathrm ds\mathrm dt\Big].
\end{align*}
Similarly to the proof of \eqref{l6.1.4},  using \eqref{Xmix}--\eqref{Xattr} and the  H\"older inequality leads to 
\begin{align*}
\tilde I_2(u_1,u_2)\leq K\|f\|^2_{p,\gamma}\|u_1-u_2\|^{\gamma}(1+\|u_1\|^{p\tilde r+(1+2\beta)\gamma}+\|u_2\|^{p\tilde r+(1+2\beta)\gamma}). 
\end{align*}
Similarly to the proof of \eqref{l6.1.10}, by the relation $\tilde I_1(u_1,u_2)=-\mathbb E[\tilde I_2(X^{u_1}_1,X^{u_2}_1)]$, $2p\tilde r+2(1+2\beta)\gamma\leq r$, and Assumption \ref{a1}, we obtain 
\begin{align*}
\tilde I_1(u_1,u_2)&\leq K\|f\|^2_{p,\gamma}\mathbb E\Big[\|X^{u_1}_1-X^{u_2}_1\|^{\gamma}(1+\|X^{u_1}_1\|^{p\tilde r+(1+2\beta)\gamma}+\|X^{u_2}_1\|^{p\tilde r+(1+2\beta)\gamma})\Big]\\
& \leq K\|f\|^2_{p,\gamma}\|u_1-u_2\|^{\gamma}(1+\|u_1\|^{\tilde r(p\tilde r+(1+2\beta)\gamma)+\gamma\beta}+\|u_2\|^{\tilde r(p\tilde r+(1+2\beta)\gamma)+\gamma\beta}).
 \end{align*}
Similarly,  $\tilde I_3$ can be estimated as follows:
\begin{align*}
\tilde I_3(u_1,u_2)&\leq K\|f\|^2_{p,\gamma}\!\int_0^1\mathbb E\Big[\|X^{u_1}_t\!\!-\!\!X^{u_2}_t\|^{\gamma}(1+\|X^{u_1}_t\|^p+\|X^{u_2}_t\|^p)^{\frac12}(1+\|X^{u_1}_t\|^{\frac{p\tilde r}{2}+\gamma(1+\beta)})\\
&\quad +\|X^{u_1}_t-X^{u_2}_t\|^{\gamma}(1+\|X^{u_1}_t\|^{\frac{p\tilde r}{2}+\gamma\beta}+\|X^{u_2}_t\|^{\frac{p\tilde r}{2}+\gamma\beta})(1+\|X^{u_2}_t\|^{\frac p2})\Big]\mathrm dt\\
&\leq K\|f\|^2_{p,\gamma}\|u_1-u_2\|^{\gamma}(1+\|u_1\|^{\frac {\tilde rp}{2}(1+\tilde r)+\tilde r(1+\beta)\gamma+\beta\gamma}+\|u_2\|^{\frac {\tilde rp}{2}(1+\tilde r)+\tilde r(1+\beta)\gamma+\beta\gamma}),
\end{align*}
where we used \eqref{Xmix}--\eqref{Xattr},  the H\"older inequality, $p(1+\tilde r)+2(1+\beta)\gamma\leq r$, and Assumption \ref{a1}. This, combining the definition of $\bar p_{\gamma}$ gives that
$H\in\mathcal C_{2\bar p_{\gamma},\gamma}$ and $\|H\|_{2\bar p_{\gamma},\gamma}\leq K\|f\|^2_{p,\gamma},$ which finishes the proof of \textit{Claim 1}.

\textit{Claim 2: $\mu(H)=v^2.$} In fact, $H$ can be rewritten as
\begin{align*}
H=P_1\Big|\int_0^{\infty}P_tf\mathrm dt\Big|^2-\Big|\int_0^{\infty}P_tf\mathrm dt\Big|^2+2\int_0^1P_t\Big(\int_0^{\infty}fP_sf\mathrm ds\Big)\mathrm dt,
\end{align*}
which together with the fact that $P_t^*\mu=\mu$ for $t>0$ finishes the proof of \textit{Claim 2}.

\textit{Claim 3: $\frac{1}{\sqrt{N}}\mathscr R^x_N\overset{\mathbb P}\longrightarrow 0$ as $N\to\infty$.} 
Combining Assumption \ref{a1} and  \eqref{Xmix} leads to
\begin{align*}
\mathbb E\Big[\Big|\frac{1}{\sqrt N}\mathscr R^x_N\Big|\Big]\leq\frac{1}{\sqrt N}\mathbb E\Big[\Big|\int_0^{\infty}P_tf(X^x_N)\mathrm dt\Big|+\Big|\int_0^{\infty}P_tf(x)\mathrm dt\Big|\Big]\to0\quad\text{ as }N\to\infty.
\end{align*}
 
\textit{Claim 4: $\frac{1}{\sqrt{N}}\mathscr M^x_N\overset{d}\longrightarrow\mathscr N(0,v^2)$ as $N\to\infty$.}
The proof is similar to that of  Theorem \ref{CLTle}. The aim is to show that
\begin{align*}
\lim_{N\to \infty}\E \big[\exp{\big({\bf i}\iota N^{-\frac12}\mathscr M^{x}_{N}\big)}\big]=e^{-\frac{v^2\iota^2}{2}}\quad\forall\;\iota \in\RR. 
\end{align*}
Similar to \eqref{th6.1.3}, we have
\begin{align*}
e^{\frac{v^2\iota^2}{2}}\E \big[\exp{\big({\bf i}\iota N^{-\frac12}\mathscr M^{x}_{N}\big)}\big]-1=\tilde I_1(N)+\tilde I_2(N)+\tilde I_3(N),
\end{align*}
where 
\begin{align*}
\tilde I_1(N):=&{}\sum_{j=0}^{N-1}e^{\frac{v^2\iota^2(j+1)}{2N}}(1-e^{-\frac{v^2\iota^2}{2N}}-\frac{v^2\iota^2}{2N})\E[\exp{({\bf i} \iota N^{-\frac12}\mathscr {M}^{x}_{j})}],\nn\\
\tilde I_2(N):=&{}-\iota^2N^{-1}\sum_{j=0}^{N-1}e^{\frac{v^2\iota^2(j+1)}{2N}}\E\Big[\exp{({\bf i} \iota N^{-\frac12}\mathscr {M}^{x}_{j})}(\mathcal Z^{x}_{j+1})^2Q(\iota N^{-\frac12}\mathcal Z^x_{j+1})\Big],\nn\\
\tilde I_3(N):=&{}\frac{\iota^2}{2N}\sum_{j=0}^{N-1}e^{\frac{v^2\iota^2(j+1)}{2N}}\E\Big[\exp{({\bf i} \iota N^{-\frac12}\mathscr {M}^{x}_{j})}\big(v^2-(\mathcal Z^{x}_{j+1})^2\big)\Big].
\end{align*}
A similar proof to \eqref{th6.1.4} gives $|\tilde I_1(N)|\to0$ as $N\to\infty.$

For the estimate of $\tilde I_2(N),$ we arrive at that for  $\epsilon\ll1$ and $c_1>2,$
\begin{align*}
|\tilde I_2(N)|\leq& \iota ^2e^{\frac{v^2\iota ^2}{2}}N^{-1}\Big[\sup_{|\zeta|\leq \iota \epsilon}|Q(\zeta)|\sum_{j=0}^{N-1}\mathbb E[(\mathcal Z^x_{j+1})^2]\nn\\
&+\sum_{j=0}^{N-1}(\mathbb E\big[(\mathcal Z^x_{j+1})^{c_1}\big])^{\frac{2}{c_1}}\big(\mathbb P(|\mathcal Z^x_{j+1}|>\epsilon\sqrt N)\big)^{\frac{c_1-2}{c_1}}\Big].
\end{align*}
 Since the condition on $p$ in Proposition \ref{Limit_exact} implies that there exists $c_1>2$ such that $c_1(\frac{p\tilde r}{2}+(1+\beta)\gamma)\leq r,$ 
 we have $\sup_{j\in\mathbb N}\mathbb E[(\mathcal Z^x_{j})^{c_1}]\leq K$.
 Thus $|\tilde I_2(N)|$ tends to $0$ as $N\to\infty$ and $\epsilon\rightarrow 0$.

For the estimate of $\tilde I_3(N),$  similar to that of $I_3(k)$ in Theorem \ref{CLTle}, we define $\widetilde N:=N^{1-\epsilon_0}$ with some $\epsilon_0\ll1$ and $\widetilde M:=\lceil N/\widetilde N\rceil$, 
and denote 
\begin{align*}
\widetilde {\mathbb B}_{i}=&~\{(i-1)\widetilde N,(i-1)\widetilde N +1,\ldots, i\widetilde N-1\}~~~~\mbox{for all} ~i\in\{1,\ldots, \widetilde M-1\},\nn\\
\widetilde{\mathbb B}_{\widetilde M}=&~\{(\widetilde M-1)\widetilde N,\ldots, N-1\}.
\end{align*}
Then     the term $\tilde I_3(N)$ can be estimated  as
\begin{align*}
\tilde I_3(k)&\leq \frac{\iota^2}{2N}\sum_{i=1}^{\widetilde M}\E\Big[\exp{({\bf i} \iota N^{-\frac12}\mathscr {M}^{x}_{(i-1)\widetilde N})}\sum_{j\in\widetilde{\mathbb B}_{i}}e^{\frac{v^2\iota^2(j+1)}{2N}}\big(v^2-(\mathcal Z^{x}_{j+1})^2\big)\Big],\nn\\
&\quad +\frac{\iota^2}{2N}e^{\frac{v^2\iota^2}{2}}\sum_{i=1}^{\widetilde M}\sum_{j\in\widetilde{\mathbb B}_{i}}\E\Big[\big|\exp{\big({\bf i} \iota N^{-\frac12}(\mathscr {M}^{x}_{j}-\mathscr {M}^{x}_{(i-1)\widetilde N})\big)}-1\big|\big(v^2+(\mathcal Z^{x}_{j+1})^2\big)\Big]\nn\\
&=:\tilde I_{3,1}(N)+\tilde I_{3,2}(N).
\end{align*}
The condition of $p$ in Proposition \ref{exactle} implies $2\bar p_{\gamma}\leq r$. 
Then, using the property of the conditional expectation, \textit{Claims 1--2}, \eqref{Xmix}, and Assumption \ref{a1} yields that 
\begin{align*}
\tilde I_{3,1}(N)&\leq \frac{\iota^2}{2N}\sum_{i=1}^{\widetilde M}\E\Big[\Big|\E\Big[\sum_{j\in\widetilde{\mathbb B}_{i}}e^{\frac{v^2\iota^2(j+1)}{2N}}\big(v^2-(\mathcal Z^{x}_{j+1})^2\big)\Big|\mathcal F_{(i-1)\widetilde N}\Big]\Big|\Big]\\
&\leq \frac{\iota^2}{2N}e^{\frac{v^2\iota^2}{2}}\sum_{i=1}^{\widetilde M}\mathbb E\Big[\sum_{j\in\widetilde {\mathbb B}_1}\big|v^2-\mathbb E[(\mathcal Z^y_{j+1})^2]\big|\Big|_{y=X^x_{(i-1)\widetilde N}}\Big]\\
&\leq \frac{\iota^2}{2N}e^{\frac{v^2\iota^2}{2}}\sum_{i=1}^{\widetilde M}\mathbb E\Big[\sum_{j\in\widetilde {\mathbb B}_1}|\mu(H)-P_jH(y)|\Big|_{y=X^x_{(i-1)\widetilde N}}\Big]\to0\quad\text{ as }N\to\infty.
\end{align*}
 The term $\tilde I_{3,2}(N)$ can be estimated as
\begin{align*}
\tilde I_{3,2}&\leq \frac{\iota ^2}{2N}e^{\frac{v^2\iota^2}{2}}\sum_{i=1}^{\widetilde M}\sum_{j\in\widetilde {\mathbb B}_i}\Big[\iota\tilde\epsilon _0\mathbb E[v^2\!+\!(\mathcal Z^x_{j+1})^2]+2\mathbb E\Big[(v^2\!+\!(\mathcal Z^x_{j+1})^2)
\textbf 1_{\{N^{-\frac12}|\mathscr {M}^{x}_{j}-\mathscr {M}^{x}_{(i-1)\widetilde N}|>\tilde\epsilon_0\}}
\Big]\Big]\\
&\leq K\tilde\epsilon_0+K\frac{\iota^2}{N}e^{\frac{v^2\iota^2}{2}}\sum_{i=1}^{\widetilde M}\sum_{j\in\widetilde {\mathbb B}_i}(\mathbb E[v^{c_1}+(\mathcal Z^x_{j+1})^{c_1}])^{\frac{2}{c_1}}\big(\mathbb E[|\mathscr {M}^{x}_{j}-\mathscr {M}^{x}_{(i-1)\widetilde N}|^2](\tilde\epsilon_0 ^2N)^{-1}\big)^{\frac{c_1-2}{c_1}}\\
&\leq K\tilde \epsilon_0+K\big(\widetilde N/(\tilde \epsilon_0^2 N)\big)^{\frac{c_1-2}{c_1}}\to0\quad\text{ as }N\to\infty ~\mbox{and}~ \tilde \epsilon_0\rightarrow 0,
\end{align*}
where we used 
\begin{align*}
\E\big[|\mathscr {M}^{x}_{j}-\mathscr {M}^{x}_{(i-1)\widetilde N}|^2\big]
=\sum_{l=(i-1)\widetilde N}^{j}\E\big[(\mathcal Z^x_{l})^2\big]
\leq \sum_{l=0}^{\widetilde N-1}\sup_{j\in\mathbb N}\mathbb E[(\mathcal Z^x_{j})^{2}]\leq K\widetilde N.
\end{align*}

Combining \textit{Claims 1--4} completes the proof of the CLT.
\end{proof}

\section{Proof of $\eqref{ex_spde_order}$}\label{spdeorder}

\begin{proof}
By the  H\"older continuity $\|X^x_t-X^x_{\lfloor  \frac{t}{\tau}\rfloor\tau}\|_{L^2(\Omega;E)}\leq K(t-\lfloor \frac{t}{\tau}\rfloor\tau )^{\frac{\beta_1}{2}}\leq K\tau^{\frac{\beta_1}{2}}$ (see e.g. \cite[Section 2.2]{cui2021}) and the definition of $\{Y^{x^h,\Delta}_t\}_{t\ge 0}$, it suffices to prove that $\sup_{k\ge 0}\|X^x_{t_k}-Y^{x^h,\Delta}_{t_k}\|_{L^2(\Omega;E)}\leq K(\tau^{\frac{\beta_1}{2}}+N^{-\beta_1}).$ 
Introduce the auxiliary process
\begin{align}\label{au_spde}
\widetilde Y^{x^h,\Delta}_{t_{k+1}}=\widetilde Y^{x^h,\Delta}_{t_k}+A^h\widetilde Y^{x^h,\Delta}_{t_{k+1}}\tau+P^hF(X^x_{t_{k+1}})\tau+\delta W^h_k.
\end{align}
 We can rewritten \eqref{scheme_spde} and \eqref{au_spde} as
 \begin{align*}
 Y^{x^h,\Delta}_{t_k}&=S^k_{h,\tau}x^h+\tau\sum_{i=0}^{k-1}S^{k-i}_{h,\tau}P^hF(Y^{x^h,\Delta}_{t_{i+1}})+W^k_{A^h},\\
 \widetilde Y^{x^h,\Delta}_{t_k}&=S^k_{h,\tau}x^h+\tau\sum_{i=0}^{k-1}S^{k-i}_{h,\tau}P^hF(X^x_{t_{i+1}})+W^k_{A^h},
 \end{align*}
 where $S_{h,\tau}:=(\mathrm {Id}-A^h\tau)^{-1},$ $W^k_{A^h}:=\sum_{i=0}^{k-1}S_{h,\tau}^{k-i}\delta W^h_i.$
 
 We split the error as $\|X^x_{t_k}-Y^{x^h,\Delta}_{t_k}\|_{L^2(\Omega;E)}\leq \|X^x_{t_k}-P^hX^x_{t_k}\|_{L^2(\Omega;E)}+\|P^hX^x_{t_k}-\widetilde Y^{x^h,\Delta}_{t_k}\|_{L^2(\Omega;E)}+\|\widetilde Y^{x^h,\Delta}_{t_k}-Y^{x^h,\Delta}_{t_k}\|_{L^2(\Omega;E)}.$
 It is clear that $\|X^x_{t_k}-P^hX^x_{t_k}\|_{L^{\tilde q}(\Omega;E)}\leq K\lambda_N^{-\frac{\beta_1}{2}}\|X^x_{t_k}\|_{L^{\tilde q}(\Omega;\dot{E}^{\beta_1})}$ for $\tilde q\ge 2$, where $\lambda_N\sim N^2$ is the $N$-th eigenvalue of $-A$,  and $\dot{E}^{\beta_1}$ is the Sobolev  space generated by the fractional power of $-A$ (see e.g. \cite[Section B.2]{Kruse}). 
 
 For the term $\|P^hX^x_{t_k}-\widetilde Y^{x^h,\Delta}_{t_k}\|_{L^2(\Omega;E)},$ we have 
 \begin{align*}
 &\|P^hX^x_{t_k}-\widetilde Y^{x^h,\Delta}_{t_k}\|_{L^2(\Omega;E)}\nn\\
 \leq&\; \|(e^{A^ht_k}-S^k_{h,\tau})x^h\|_{L^2(\Omega;E)}\nn\\
&+\Big\|\int_0^{t_k}e^{A^h(t_k-s)}P^hF(X^x_s)\mathrm ds-\tau\sum_{i=0}^{k-1}S^{k-i}_{h,\tau}P^hF(X^x_{t_{i+1}})\Big\|_{L^2(\Omega;E)}\\
 &+\Big\|\int_0^{t_k}e^{A^h(t_k-s)}\mathrm dW(s)-\sum_{i=0}^{k-1}S^{k-i}_{h,\tau}\delta W^h_i\Big\|_{L^2(\Omega;E)}\nn\\
 =:&\;JJ_1+JJ_2+JJ_3.
 \end{align*}
 The term $JJ_1$ can be estimated as $
 JJ_1\leq K\tau^{\frac{\beta_1}{2}}
$ (see e.g. \cite[Lemma 3.12]{Kruse}). 

For the term $JJ_2$, noting that similar to the proof of \cite[Lemma 2]{cui2021}, it can be shown that $$\sup_{s\ge 0}\|F(X^x_s)\|_{L^2(\Omega;E)}\leq K\sup_{s\ge 0}\|X^x_s\|_{L^4(\Omega;E)}\sup_{s\ge 0}(1+\|X^x_s\|^2_{L^8(\Omega;\mathcal C((0,1);\mathbb R))})\leq K,$$ which 
combining  the H\"older continuity  $\|X^x_s-X^x_{\lceil  \frac{s}{\tau}\rceil\tau}\|_{L^4(\Omega;E)}\leq K(\lceil  \frac{s}{\tau}\rceil\tau -s)^{\frac{\beta_1}{2}}$ (see e.g. \cite[Section 2.2]{cui2021}) and  $\|(-A)^{-\frac{\beta_1}{2}}(\mathrm{Id}-e^{A^h(s-\lfloor \frac{s}{\tau}\rfloor\tau)})\|_{\mathcal L(E)}\leq K(s-\lfloor \frac{s}{\tau}\rfloor\tau )^{\frac{\beta_1}{2}}$ (see e.g. \cite[Lemma B.9]{Kruse}) leads to 
 \begin{align*}
 JJ_2&\leq \Big\|\int_0^{t_k}e^{A^h(t_k-s)}P^h[F(X^x_{s})-F(X^x_{\lceil \frac{s}{\tau}\rceil\tau})]\mathrm ds\Big\|_{L^2(\Omega;E)}\\
 &\quad+\Big\|\int_0^{t_k}(e^{A^h(t_k-s)}-e^{A^h(t_k-\lfloor \frac{s}{\tau}\rfloor\tau)})P^hF(X^x_{\lceil \frac{s}{\tau}\rceil\tau})\mathrm ds\Big\|_{L^2(\Omega;E)}\\
 &\quad  +\Big\|\sum_{i=0}^{k-1}\int_{t_i}^{t_{i+1}}(e^{A^h(t_k-\lfloor \frac{s}{\tau}\rfloor\tau)}-S^{k-i}_{h,\tau})P^hF(X^x_{\lceil \frac{s}{\tau}\rceil\tau})\mathrm ds\Big\|_{L^2(\Omega;E)}\\
 &\leq \int_0^{t_k}e^{-\lambda_1(t_k-s)}\|F(X^x_s)-F(X^x_{\lceil \frac{s}{\tau}\rceil\tau})\|_{L^2(\Omega;E)}\mathrm ds\\
 &\quad+\!\!\int_0^{t_k}\!
(t_k\!-\!s)^{-\frac{\beta_1}{2}}e^{-\frac{\lambda_1}{2}(t_k-s)}\|(-A)^{-\frac{\beta_1}{2}}(\mathrm {Id}\!-\!e^{A^h(s-\lfloor \frac{s}{\tau}\rfloor\tau)})\|_{\mathcal L(E)}\|F(X^x_{\lceil \frac{s}{\tau}\rceil\tau})\|_{L^2(\Omega;E)}\mathrm ds\\
 &\quad +\sum_{i=0}^{k-1}\int_{t_i}^{t_{i+1}}\|(e^{A^h(t_k-t_i)}-S^{k-i}_{h,\tau})P^hF(X^x_{\lceil \frac{s}{\tau}\rceil\tau})\|_{L^2(\Omega;E)}\mathrm ds\\
 &\leq K\tau^{\frac{\beta_1}{2}}+\sum_{i=0}^{k-1}\int_{t_i}^{t_{i+1}}\|(e^{A^h(t_k-t_i)}-S^{k-i}_{h,\tau})P^hF(X^x_{\lceil \frac{s}{\tau}\rceil\tau})\|_{L^2(\Omega;E)}\mathrm ds. 
 \end{align*}
 Note that the inequalities $\ln (1+\zeta)\ge \frac12\zeta$ and  $\zeta-\ln (1+\zeta)\leq K\zeta^2$ for $\zeta\in(0,1]$ give 
 \begin{align*}
 &\quad \|(e^{A^ht_i}-S^i_{h,\tau})x\|^2\nn\\
&=\sum_{\lambda_j\tau\leq 1}(e^{-\lambda_jt_i}-(1+\tau\lambda_j)^{-i})^2x_j^2+\sum_{\lambda_j\tau>1}(e^{-\lambda_jt_i}-(1+\tau\lambda_j)^{-i})^2x_j^2\\
&= \sum_{\lambda_j\tau\leq 1}e^{-2i\ln (1+\tau\lambda_j)}(1-e^{-i(\tau\lambda_j-\ln (1+\tau\lambda_j))})^2x_j^2\nn\\
&\quad+\sum_{\lambda_j\tau>1}\frac{1}{t_i}\Big[(\lambda_ji\tau)^{\frac12}(e^{-\lambda_jt_i}-(1+\tau\lambda_j)^{-i})\Big]^2\lambda_j^{-1}x_j^2\\
&\leq K\sum_{\lambda_j\tau\leq1}e^{-t_i\lambda_j}(i\tau^2\lambda_j^2)^2x^2_j\nn\\
&\quad+K\tau\sum_{\lambda_j\tau>1}\frac{1}{t_i}(1+\tau\lambda_1)^{-i}\Big[(\lambda_ji\tau)^{\frac 12}(e^{-\lambda_jt_i}(1+\tau\lambda_j)^{\frac i2}-(1+\tau\lambda_j)^{-\frac i2})\Big]^2x^2_j.
 \end{align*}
 Based on $e^{-\zeta^2}\leq K\zeta^{-s}$ for $s>0,\zeta>0$, we obtain  $e^{-\frac12 t_i\lambda_j}\leq K(t_i\lambda_j)^{-3}$. And due to $\zeta^{\frac12}e^{-\frac12\zeta}\leq K$, $1+\zeta\leq e^{\zeta}$, and $(1+\zeta)^i\ge 1+i\zeta,i\ge 1$ for   $\zeta>0$, we derive  $$(\lambda_ji\tau)^{\frac12}e^{-\lambda_jt_i}(1+\tau\lambda_j)^{\frac i2}+(\lambda_ji\tau)^{\frac12}(1+\tau\lambda_j)^{-\frac i2}\leq Ke^{-\frac12\lambda_jt_i}(1+\tau\lambda_j)^{\frac{i}{2}}+\big(\frac{\lambda_ji\tau}{1+i\tau\lambda_j}\big)^{\frac12}\leq K.$$ 
Thus,
 \begin{align*}
 \|(e^{A^ht_i}-S^i_{h,\tau})x\|^2
&\leq K\sum_{\lambda_j\tau\leq 1}e^{-\frac12t_i\lambda_1}(t_i\lambda_j)^{-3}t_i^2\tau^2\lambda_j^4x_j^2+K\tau t_i^{-1}(1+\tau\lambda_1)^{-i}\sum_{\lambda_j\tau>1}x^2_j\\
&\leq K\tau t_i^{-1}(e^{-\frac12t_i\lambda_1}\vee (1+\tau\lambda_1)^{-i})\|x\|^2.
 \end{align*}
Hence, we have
\begin{align*}
 JJ_2&\leq K\tau^{\frac{\beta_1}{2}}+ K\tau^{\frac12}\sum_{i=0}^{k-1}\int_{t_i}^{t_{i+1}}(e^{-\frac{\lambda_1}{4} (t_k-t_i)}\vee (1+\tau\lambda_1)^{-\frac12(k-i)})t_{k-i}^{-\frac12}\|F(X^x_{\lceil \frac{s}{\tau}\rceil\tau})\|_{L^2(\Omega;E)}\mathrm ds\nn\\
 &\leq K\tau^{\frac{\beta_1}{2}}.
\end{align*}
 
For the term $JJ_3$, 
\begin{align*}
 JJ_3&\leq \Big\|\int_0^{t_k}e^{A^h(t_k-s)}(\mathrm{Id}-e^{A^h(s-\lfloor \frac{s}{\tau}\rfloor\tau)})\mathrm dW(s)\Big\|_{L^2(\Omega;E)}\nn\\
 &\quad+\Big\|\int_0^{t_k}(e^{A^h(t_k-\lfloor \frac{s}{\tau}\rfloor\tau)}-S^{k-\lfloor \frac{s}{\tau}\rfloor}_{h,\tau})\mathrm dW(s)\Big\|_{L^2(\Omega;E)}\\&=:JJ_{3,1}+JJ_{3,2}.
\end{align*}
It follows from the It\^o isometry and \cite[Lemma B.9]{Kruse} that the term $JJ_{3,1}$ can be estimated as
\begin{align*}
 (JJ_{3,1})^2=\int_0^{t_k}\|(-A)^{\frac{1}{2}}e^{A^h(t_k-s)}(-A)^{-\frac{\beta_1}{2}}(\mathrm {Id}-e^{A^h(s-\lfloor \frac{s}{\tau}\rfloor\tau)})(-A)^{\frac{\beta_1-1}{2}}Q^{\frac12}\|^2_{\mathcal L_2}\mathrm ds\leq K\tau^{\beta_1}.
\end{align*}
Applying  the It\^o isometry  again, we derive
\begin{align*}
& (JJ_{3,2})^2
 = \int_0^{t_k}\sum_{j=0}^{\infty}\|(-A)^{-\frac{\beta_1-1}{2}}(e^{A^h(t_k-\lfloor \frac{s}{\tau}\rfloor\tau)}-S^{k-\lfloor \frac{s}{\tau}\rfloor}_{h,\tau})(-A)^{\frac{\beta_1-1}{2}}Q^{\frac12}e_j\|^2\mathrm ds\\
 &\leq K\!\int_0^{t_k}\!\Big[\!\!\sum_{\lambda_j\tau\leq 1}\lambda_j^{-\beta_1+1}e^{-2(k-\lfloor \frac{s}{\tau}\rfloor)\ln (1+\tau\lambda_j)}(1\!-\!e^{-(k-\lfloor \frac{s}{\tau}\rfloor)(\tau\lambda_j-\ln (1+\tau\lambda_j))})^2\|(-A)^{\frac{\beta_1-1}{2}}Q^{\frac12}e_j\|^2\\
&\quad +\sum_{\lambda_j\tau>1}2\lambda_j^{-\beta_1}\lambda_j\Big((1+\tau\lambda_j)^{-2(k-\lfloor \frac{s}{\tau}\rfloor)}+e^{-2(t_k-\lfloor \frac{s}{\tau}\rfloor\tau)\lambda_j}\Big)\|(-A)^{\frac{\beta_1-1}{2}}Q^{\frac12}e_j\|^2\Big]\mathrm ds\\
&\leq K\int_0^{t_k}\Big[\sum_{\lambda_j\tau\leq 1}\lambda_j^{1-\beta_1}e^{-\frac12(k-\lfloor \frac{s}{\tau}\rfloor)\tau\lambda_j}((t_k-\lfloor \frac{s}{\tau}\rfloor\tau)\lambda_j)^{-2}(k-\lfloor \frac{s}{\tau}\rfloor)^2(\tau\lambda_j)^4\|(-A)^{\frac{\beta_1-1}{2}}Q^{\frac12}e_j\|^2\\
&\quad +\sum_{\lambda_j\tau>1}\tau^{\beta_1}\lambda_j(e^{-2(k-\lfloor \frac{s}{\tau}\rfloor)\ln(1+\tau\lambda_j)}+e^{-2(t_k-\lfloor \frac{s}{\tau}\rfloor\tau)\lambda_j})\|(-A)^{\frac{\beta_1-1}{2}}Q^{\frac12}e_j\|^2\Big]\mathrm ds\\
&\leq K\|(-A)^{\frac{\beta_1-1}{2}}Q^{\frac12}\|_{\mathcal L_2}^2\tau^{\beta_1},
 \end{align*}
 where in the last step we used $$\lambda_j^{-\beta_1}((t_k-\lfloor \frac{s}{\tau}\rfloor\tau)\lambda_j)^{-2}(k-\lfloor \frac{s}{\tau}\rfloor)^2(\tau\lambda_j)^4\leq \lambda^{-\beta_1}_j\tau^2\lambda_j^2\leq (\lambda_j\tau)^{2-\beta_1}\tau^{\beta_1}\leq \tau^{\beta_1}$$ for $\lambda_j\tau\leq 1$ and $$\sup_{k,j}\int_0^{t_k}\lambda_j(e^{-\zeta(t_k-\lfloor \frac{s}{\tau}\rfloor\tau)\frac{\ln (1+\tau\lambda_j)}{\tau}}+e^{-\zeta(t_k-\lfloor \frac{s}{\tau}\rfloor\tau)\lambda_j})\mathrm ds\leq K$$ for $\zeta>0$.
Moreover, by using the Burkholder--Davis--Gundy inequality (see e.g.  \cite[Proposition 2.12]{Kruse}), one can derive that for $\tilde q_0\ge 2,$ 
$\|P^hX^x_{t_k}-\widetilde Y^{x^h,\Delta}_{t_k}\|_{L^{\tilde q_0}(\Omega;E)}\leq K\tau^{\frac{\beta_1}{2}}.$
 
 For the term $\|\widetilde Y^{x^h,\Delta}_{t_k}-Y^{x^h,\Delta}_{t_k}\|_{L^2(\Omega;E)}$, denoting $\tilde e_k:=\widetilde Y^{x^h,\Delta}_{t_k}-Y^{x^h,\Delta}_{t_k},$ we have 
 \begin{align*}
 \tilde e_k-\tilde e_{k-1}=A^h\tilde e_k\tau+\tau P^h(F(X^x_{t_{k}})-F(Y^{x^h,\Delta}_{t_k})).
 \end{align*}
 Hence, by taking $\langle \cdot,\,\tilde e_k\rangle$ on both sides, we obtain
 \begin{align*}
 \frac12(\|\tilde e_k\|^2-\|\tilde e_{k-1}\|^2)&\leq \tau\langle A^h\tilde e_k,\tilde e_k\rangle +\tau\langle P^h(F(X^x_{t_{k}})-F(Y^{x^h,\Delta}_{t_k})),\tilde e_k\rangle\\
 &\leq -(\lambda_1-\lambda_F)\|\tilde e_k\|^2\tau+\tau\langle P^h(F(X^x_{t_k})-F(\widetilde Y^{x^h,\Delta}_{t_k})),\tilde e_k\rangle\\
 &\leq -\frac{\lambda_1-\lambda_F}{2}\|\tilde e_k\|^2\tau+\frac{\tau}{2(\lambda_1-\lambda_F)}\|F(X^x_{t_k})-F(\widetilde Y^{x^h,\Delta}_{t_k})\|^2,
 \end{align*}
 which gives
 \begin{align*}
  e^{\epsilon t_k}\|\tilde e_k\|^2
&\leq  \sum_{i=0}^{k-1}e^{\epsilon t_{i}}\|\tilde e_i\|^2\Big(\frac{e^{\epsilon\tau}}{1+(\lambda_1-\lambda_F)\tau}-1\Big)\\
&\quad +\frac{1}{1+(\lambda_1-\lambda_F)\tau}\sum_{i=0}^{k-1}e^{\epsilon t_{i+1}}\frac{\tau}{(\lambda_1-\lambda_F)}\|F(X^x_{t_{i+1}})-F(\widetilde Y^{x^h,\Delta}_{t_{i+1}})\|^2.
 \end{align*}
 Take $\epsilon\ll1$ so that $\frac{e^{\epsilon \tau}}{1+(\lambda_1-\lambda_F)\tau}\leq 1$. 
Similar to the proof of \cite[Lemmas 2 and 4]{cui2021}, it can be shown that $\sup_{k\ge 0}\mathbb E[\|X^x_{t_k}\|^8_{\mathcal C((0,1);\mathbb R)}+\|\widetilde Y^{x^h,\Delta}_{t_k}\|^8_{\mathcal C((0,1);\mathbb R)}]\leq K,$ which implies 
\begin{align*}
&\quad \sup_{k\ge 0}\mathbb E[\|F(X^x_{t_k})-F(\widetilde Y^{x^h,\Delta}_{t_k})\|^2]\\
&\leq K\sup_{k\ge 0}\|X^x_{t_k}-\widetilde Y^{x^h,\Delta}_{t_k}\|^2_{L^4(\Omega;E)}\sup_{k\ge 0}(\mathbb E[1+\|X^x_{t_k}\|^8_{\mathcal C((0,1);\mathbb R)}+\|\widetilde Y^{x^h,\Delta}_{t_k}\|^8_{\mathcal C((0,1);\mathbb R)})])^{\frac12}\\
&\leq K(\tau^{\beta_1}+N^{-2\beta_1}).
\end{align*}
Hence, we obtain  $\mathbb E[\|\tilde e_k\|^2]\leq K(\tau^{\beta_1}+N^{-2\beta_1})$, which finishes the proof.
\end{proof}

\section{Proofs of $\eqref{sfdebound}$ and $\eqref{sfdecov}$}\label{sfdeproof}
We first propose the hypotheses of coefficients $b$, $\sigma$, and the initial datum $x(\cdot)$.

\textbf{(H1).}
Assume that there are positive constants $\bar L_1,\bar L_2$, and the probability measure $\nu_1$ on $[-\delta_0, 0]$ such that  
\begin{align*}
|\sigma(\phi_1)-\sigma(\phi_2)|^2\leq \bar L_1\Big(|\phi_1(0)-\phi_2(0)|^2+\int_{-\delta_0}^{0}|\phi_1(\theta)-\phi_2(\theta)|^2\mathrm d\nu_1(\theta)\Big)\quad\forall\;\phi_1,\phi_2\in E
\end{align*}
and 
\begin{align*}
|\sigma(\phi)|\leq \bar L_2\quad\forall\; \phi\in E.
\end{align*}

\textbf{(H2).}
Assume that there are positive constants $\bar L_3,\bar L_4, \bar L_5$ with $\bar L_3 >\bar L_1+\bar L_4$ and the probability measure $\nu_2$ on $[-\delta_0, 0]$ such that for any $\phi_1,\phi_2\in E,$
\begin{align*}
\big\langle \phi_1(0)-\phi_2(0), b(\phi_1)-b(\phi_2) \big\rangle\leq -\bar L_3|\phi_1(0)-\phi_2(0)|^2+\bar L_4\int_{-\delta_0}^{0}|\phi_1(\theta)-\phi_2(\theta)|^2\mathrm d\nu_2(\theta)
\end{align*}
and
\begin{align*}
|b(\phi_1)-b(\phi_2)|^2\leq \bar L_5\Big(|\phi_1(0)-\phi_2(0)|^2+\int_{-\delta_0}^{0}|\phi_1(\theta)-\phi_2(\theta)|^2\mathrm d\nu_2(\theta)\Big).
\end{align*}

\textbf{(H3).}
Assume that there are positive constants $\bar L_6$ and $\ell\geq 1/2$ such that
\begin{align*}
|x(\theta_1)-x(\theta_2)|\leq \bar L_6|\theta_1-\theta_2|^{\ell}\quad\forall\;\theta_1,\theta_2\in [-\delta_0, 0].
\end{align*}

\begin{proof}
\textbf{Proof of \eqref{sfdebound}.}
To simplify notations we introduce the following abbreviations
\begin{align*}
&y(t_k)=y^{x,\tau}(t_k), ~~Y_{t_k}=Y^{x^h,\Delta}_{t_k},
~~b_{k}=b(Y^{x^h,\Delta}_{t_k}),~~\sigma_k=\sigma(Y^{x^h,\Delta}_{t_k}).
\end{align*}
It follows from \eqref{thL}, $y(t_k)=Y_{t_k}(0)$, and the hypotheses (H1)--(H2) that for any $\epsilon\in(0,1)$,  
\begin{align*}
&\;|y(t_{k+1})|^2=|y(t_k)|^2+|b_k|^2\tau^2+|\sigma_k\delta W_k|^2
+2\langle y(t_k), b_k\rangle\tau+\mathcal M_k\nn\\
\leq&~|y(t_k)|^2+(1+\frac{1}{\epsilon})\tau^2|b({\textbf 0})|^2+(1+\epsilon)\tau^2|b_k-b({\textbf 0})|^2+\bar L_2^2|\delta W_k|^2+
2\tau\langle y(t_k), b({\textbf 0})\rangle\nn\\
&~-2\bar L_3\tau|y(t_k)|^2+2\bar L_4\tau\int_{-\delta_0}^0|Y_{t_k}(\theta)|^2\mathrm d\nu_2(\theta)+\mathcal M_{k}\nn\\
\leq&~|y(t_k)|^2+K(1+\frac{1}{\epsilon})\tau+\bar L_2^2|\delta W_k|^2-\Big(2\bar L_3-(1+\epsilon)\tau\bar L_5-\epsilon\Big)\tau|y(t_k)|^2\nn\\
&+\Big(2\bar L_4+(1+\epsilon)\tau\bar L_5\Big)\tau\int_{-\delta_0}^0|Y_{t_k}(\theta)|^2\mathrm d\nu_2(\theta)+
\mathcal M_{k},
\end{align*}
where the martingale $\{\mathcal M_{k}\}_{k=0}^{\infty}$ is defined by
\begin{align*}
\mathcal M_k:=2\langle y(t_k), \sigma_k\delta W_k\rangle+2\langle b_k, \sigma_k\delta W_k\rangle\tau.
\end{align*}
Then for any $q_0\in\mathbb N$ and $\tilde\e>0$,
\begin{align}\label{1p2.1}
e^{\tilde \e t_{k+1}}|y(t_{k+1})|^{2q_0}=&\sum_{i=0}^k\big( e^{\tilde \e t_{i+1}}|y(t_{i+1})|^{2q_0}-e^{\tilde \e t_{i}}|y(t_{i})|^{2q_0}\big)+|x(0)|^{2q_0}\nn\\
\leq&~|x(0)|^{2q_0}+\sum_{i=0}^{k}(1-e^{-\tilde \e \tau})e^{\tilde \e t_{i+1}}|y(t_{i})|^{2q_0}+\sum_{l=1}^{q_0}\sum_{i=0}^kC_{q_0}^{l} I_{l,i},
\end{align}
where positive constants $C_{q_0}^l$ are binomial coefficients,
and 
\begin{align*}
I_{l,i}&:=e^{\tilde \e t_{i+1}}|y(t_{i})|^{2(q_0-l)}\Big(K(1+\frac{1}{\epsilon})\tau+\bar L_2^2|\delta W_i|^2-\big(2\bar L_3-(1+\epsilon)\tau\bar L_5-\epsilon\big)\tau|y(t_i)|^2\nn\\
&\quad +\big(2\bar L_4+(1+\epsilon)\tau\bar L_5\big)\tau\int_{-\delta_0}^0|Y_{t_i}(\theta)|^2\mathrm d\nu_2(\theta)+
\mathcal M_{i}\Big)^{l}.
\end{align*}
Applying the Young inequality and the H\"older inequality yields
\begin{align*}
I_{1,i}&\leq\frac{e^{\tilde \e t_{i+1}}K}{(\epsilon\tau)^{q_0-1}}\Big((1+\frac{1}{\epsilon})^{q_0}\tau^{q_0}+\bar L_2^{2q_0}|\delta W_i|^{2q_0}\Big)-\big(2\bar L_3\!-\!(1+\epsilon)\tau\bar L_5-3\epsilon\big)\tau e^{\tilde \e t_{i+1}}|y(t_{i})|^{2q_0}\nn\\
&\quad+\big(2\bar L_4+(1+\epsilon)\tau\bar L_5\big)\tau e^{\tilde \e t_{i+1}}\Big(\frac{q_0-1}{q_0}|y(t_{i})|^{2q_0}+\frac{1}{q_0}\int_{-\delta_0}^0|Y_{t_i}(\theta)|^{2q_0}\mathrm d\nu_2(\theta)\Big)\nn\\
&\quad+e^{\tilde \e t_{i+1}}|y(t_{i})|^{2(q_0-1)}\mathcal M_{i},
\end{align*}
and 
\begin{align*}
|I_{l,i}|&\leq K6^{^{l-1}}e^{\tilde \e t_{i+1}}|y(t_{i})|^{2(q_0-l)}\Big((1+\frac{1}{\epsilon})^l\tau^l+\bar L_2^{2l}|\delta W_i|^{2l}+\tau^l|y(t_i)|^{2l}\nn\\
&\quad +\tau^{l}\int_{-\delta_0}^0|Y_{t_i}(\theta)|^{2l}\mathrm d\nu_2(\theta)+
|y(t_i)|^l|\sigma_i\delta W_i|^l+\tau^{l}|b_i|^l|\sigma_i\delta W_i|^l\Big)\nn\\
&\leq 3\epsilon\tau e^{\tilde \e t_{i+1}}|y(t_{i})|^{2q_0}+
 \frac{Ke^{\tilde \e t_{i+1}}}{(\epsilon\tau)^{\frac{q_0-l}{l}}}\Big((1+\frac{1}{\epsilon})^{q_0}\tau^{q_0}+\bar L_2^{2q_0}|\delta W_i|^{2q_0}\Big)+K \tau^{l}e^{\tilde \e t_{i+1}}|y(t_{i})|^{2q_0}\nn\\
&\quad +K\tau^{l}\int_{-\delta_0}^0|Y_{t_i}(\theta)|^{2q_0}\mathrm d\nu_2(\theta)+\frac{Ke^{\tilde \e t_{i+1}}}{(\epsilon\tau)^{\frac{2q_0-l}{l}}}|\delta W_i|^{2q_0} +K\tau^{l}|b_i|^{q_0}|\delta W_{i}|^{q_0}
\end{align*}
for the integer $l\in[2,q_0]$.
Inserting the above inequalities into \eqref{1p2.1} and using (H2), we obtain
\begin{align}\label{10032}
&e^{\tilde \e t_{k+1}}|y(t_{k+1})|^{2q_0}\nn\\
\leq&~|x(0)|^{2q_0}+\sum_{i=0}^{k}(1-e^{-\tilde \e \tau})e^{\tilde \e t_{i+1}}|y(t_{i})|^{2q_0}+
K_{\epsilon}\tau\sum_{i=0}^ke^{\tilde \e t_{i+1}}+K_{\epsilon}\sum_{i=0}^k\frac{|\delta W_i|^{2q_0}}{\tau^{q_0-1}}e^{\tilde \e t_{i+1}}\nn\\
&+K\tau^{2}\sum_{i=0}^k\mathscr E_{i}-\Big(q_0\big(2\bar L_3-(1+\epsilon)\tau\bar L_5\big)-3\epsilon\cdot 2^{q_0}\Big)\tau \sum_{i=0}^k e^{\tilde \e t_{i+1}}|y(t_{i})|^{2q_0}\nn\\
&+q_0\big(2\bar L_4+(1+\epsilon)\tau\bar L_5\big)\tau \sum_{i=0}^{k}e^{\tilde \e t_{i+1}}\Big(\frac{q_0-1}{q_0}|y(t_{i})|^{2q_0}+\frac{1}{q_0}\int_{-\delta_0}^0|Y_{t_i}(\theta)|^{2q_0}\mathrm d\nu_2(\theta)\Big)\nn\\
&+q_0\sum_{i=0}^{k}e^{\tilde \e t_{i+1}}|y(t_{i})|^{2(q_0-1)}\mathcal M_{i},
\end{align}
where 
\begin{align*}
\mathscr E_{i}:=e^{\tilde \e t_{i+1}}|y(t_{i})|^{2q_0}+e^{\tilde \e t_{i+1}}\int_{-\delta_0}^0|Y_{t_i}(\theta)|^{2q_0}\mathrm d\nu_2(\theta)+e^{\tilde \e t_{i+1}}|\delta W_{i}|^{2q_0}.
\end{align*}
It follows from \eqref{EMsfde} and the convex property of $|\cdot|^{2q_0}$ that
\begin{align}\label{5p2.1}
&\sum_{i=0}^{k}e^{\tilde \e t_{i+1}}\int_{-\delta_0}^{0}|Y_{t_i}(\theta)|^{2q_0}\mathrm d\nu_2(\theta)\nn\\
=&\sum_{j=-N}^{-1}\sum_{i=0}^{k}e^{\tilde \e t_{i+1}}\int_{t_j}^{t_{j+1}}
|\frac{t_{j+1}-\theta}{\tau}y(t_{i+j})+\frac{\theta-t_{j}}{\tau}y(t_{i+j+1})|^{2q_0}\mathrm d\nu_2(\theta)\nn\\
\leq{}&e^{\tilde \e\delta_0}\sum_{j=-N}^{-1}\sum_{i=0}^{k}\int_{t_j}^{t_{j+1}}
\frac{t_{j+1}-\theta}{\tau}\mathrm d\nu_2(\theta)e^{\tilde \e t_{i+j+1}}|y(t_{i+j})|^{2q_0}\nn\\
&+e^{\tilde \e\delta_0}\sum_{j=-N}^{-1}\sum_{i=0}^{k}\int_{t_j}^{t_{j+1}}
\frac{\theta-t_{j}}{\tau}\mathrm d\nu_2(\theta)e^{\tilde \e t_{i+j+2}}|y(t_{i+j+1})|^{2q_0}\nn\\
\leq{}&e^{\tilde \e\delta_0}\sum_{j=-N}^{-1}\sum_{l=-N}^{k}\int_{t_j}^{t_{j+1}}
\frac{t_{j+1}-\theta}{\tau}\mathrm d\nu_2(\theta)e^{\tilde \e t_{l+1}}|y(t_{l})|^{2q_0}\nn\\
&+e^{\tilde \e\delta_0}\sum_{j=-N}^{-1}\sum_{l=-N}^{k}\int_{t_j}^{t_{j+1}}
\frac{\theta-t_{j}}{\tau}\mathrm d\nu_2(\theta)e^{\tilde \e t_{l+1}}|y(t_{l})|^{2q_0}\nn\\
={}&e^{\tilde \e\delta_0}\sum_{j=-N}^{-1}\sum_{l=-N}^{k}\int_{t_j}^{t_{j+1}}
\mathrm d\nu_2(\theta)|y(t_{l})|^{2q_0}
\leq Ne^{\tilde \e\delta_0}\|x\|^{2q_0}+e^{\tilde \e\delta_0}\sum_{i=0}^{k}e^{\tilde \e t_{i+1}}|y(t_{i})|^{2q_0}.
\end{align}
Plugging \eqref{5p2.1} into \eqref{10032}  and using $1-e^{-\tilde\e\tau}\leq \tilde \e\tau$ leads to
\begin{align}\label{10033}
&e^{\tilde \e t_{k+1}}|y(t_{k+1})|^{2q_0}\nn\\
\leq&~K e^{\tilde \e\delta_0}\|x\|^{2q_0}+K_{\epsilon}\tau\sum_{i=0}^ke^{\tilde \e t_{i+1}}+K_{\epsilon}\sum_{i=0}^k\frac{|\delta W_i|^{2q_0}}{\tau^{q_0-1}}e^{\tilde \e t_{i+1}}+K\tau^{2}\sum_{i=0}^ke^{\tilde \e t_{i+1}}|\delta W_i|^{2q_0}\nn\\
&-\Big(q_0\big(2\bar L_3-(1+\epsilon)\tau\bar L_5\big)-3\epsilon\cdot 2^{q_0}-\tilde \e-K\tau-q_0\big(2\bar L_4+(1+\epsilon)\tau\bar L_5\big)e^{\tilde\e\delta_0}\Big)\tau\nn\\
&\times \sum_{i=0}^k e^{\tilde \e t_{i+1}}|y(t_{i})|^{2q_0}+q_0\sum_{i=0}^{k}e^{\tilde \e t_{i+1}}|y(t_{i})|^{2(q_0-1)}\mathcal M_{i}.
\end{align}
Owing to $\bar L_3\geq \bar L_1+\bar L_4$, choose a $\tilde \tau\in(0,1]$ sufficiently small such that
$$q_0\big(2\bar L_3-\tilde\tau\bar L_5\big)-K\tilde\tau-q_0\big(2\bar L_4+\tilde\tau\bar L_5\big)> 0.$$
Let $\tilde \e_0, \epsilon_0\in(0,1]$ sufficiently small such that 
$$q_0\big(2\bar L_3-(1+\epsilon_0)\tilde\tau\bar L_5\big)-3\epsilon_0\cdot 2^{q_0}-\tilde\e_0-K\tilde\tau-q_0\big(2\bar L_4+(1+\epsilon_0)\tilde\tau\bar L_5\big)e^{\tilde\e_0\delta_0}\geq0.$$
This, along with \eqref{10033} implies that for any $\tau\in(0,\tilde \tau]$,
\begin{align}\label{10035}
e^{\tilde \e_0 t_{k+1}}|y(t_{k+1})|^{2q_0}
\leq&~K e^{\tilde \e_0\delta_0}\|x\|^{2q_0}+
K_{\epsilon_0}\tau\sum_{i=0}^ke^{\tilde \e_0 t_{i+1}}+K_{\epsilon_0}\sum_{i=0}^k\frac{|\delta W_i|^{2q_0}}{\tau^{q_0-1}}e^{\tilde \e_0 t_{i+1}}\nn\\
&~+K\tau^{2}\sum_{i=0}^ke^{\tilde \e_0 t_{i+1}}|\delta W_i|^{2q_0}+q_0\sum_{i=0}^{k}e^{\tilde \e_0 t_{i+1}}|y(t_{i})|^{2(q_0-1)}\mathcal M_{i}.
\end{align}
Taking expectations in the above inequality we deduce
\begin{align*}
e^{\tilde \e_0 t_{k+1}}\E[|y(t_{k+1})|^{2q_0}]
\leq K e^{\tilde \e_0\delta_0}\|x\|^{2q_0}+K_{\epsilon_0}e^{\tilde \e_0 t_{k+1}},
\end{align*}
which implies 
\begin{align}\label{10034}
\sup_{k\geq0}\E[|y(t_{k})|^{2q_0}]\leq K(1+\|x\|^{2q_0}).
\end{align}
By virtue of \eqref{10035} and \eqref{10034} we conclude that for any $\tau\in(0,\tilde \tau]$,
\begin{align}\label{10041}
&e^{\tilde \e_0 (t_{k+1}-\delta_0)}\E\big[\sup_{(k-N)\vee 0\leq l\leq k}|y(t_{l+1})|^{2q_0}\big]\nn\\
\leq&~\E\big[\sup_{(k-N)\vee 0\leq l\leq k}e^{\tilde \e_0 t_{l+1}}|y(t_{l+1})|^{2q_0}\big]\nn\\
\leq&~K e^{\tilde \e_0\delta_0}\|x\|^{2q_0}+K_{\epsilon_0}\tau\sum_{i=0}^ke^{\tilde \e_0 t_{i+1}}\nn\\
&+q_0\E\big[\sup_{(k-N)\vee0\leq l\leq k}\sum_{i=(k-N)\vee 0}^{l}e^{\tilde \e_0 t_{i+1}}|y(t_{i})|^{2(q_0-1)}\mathcal M_{i}\big].
\end{align}
Applying the Burkholder--Davis--Gundy inequality, the condition (H1), and the Young inequality yields
\begin{align*}
&\E\big[\sup_{(k-N)\vee0\leq l\leq k}\sum_{i=(k-N)\vee 0}^{l}e^{\tilde \e_0 t_{i+1}}|y(t_{i})|^{2(q_0-1)}\mathcal M_{i}\big]\nn\\
\leq&~ K\E \Big[\Big(\sum_{i=(k-N)\vee 0}^{k}e^{2\tilde \e_0 t_{i+1}}|y(t_{i})|^{4q_0-2}|\sigma_i|^{2}\tau\Big)^{\frac{1}{2}}\Big]\nn\\
&+ K\E \Big[\Big(\sum_{i=(k-N)\vee 0}^{k}e^{2\tilde \e_0 t_{i+1}}|y(t_{i})|^{4q_0-4}|b_i|^{2}|\sigma_i|^2\tau^{\frac32}\Big)^{\frac{1}{2}}\Big]\nn\\
\leq &~\frac{1}{4q_0}\E\Big[\sup_{(k-N)\vee0\leq i\leq k}e^{\tilde\e_0 t_{i+1}}|y(t_{i})|^{2q_0}\Big]+K\Big(\sum_{i=(k-N)\vee 0}^ke^{\frac{1}{q_0} \tilde\e_0 t_{i+1}}\tau\Big)^{q_0}\nn\\
&+K\E\Big[\Big(\sum_{i=(k-N)\vee 0}^ke^{\frac{2}{q_0} \tilde\e_0 t_{i+1}}|b_i|^2\tau^{\frac32}\Big)^{\frac{q_0}{2}}\Big]\nn\\
\leq &~\frac{1}{4q_0}\E\Big[e^{\tilde \e_0\tau}\Big(e^{\tilde\e_0 t_{(k-N)\vee0}}|y(t_{(k-N)\vee0})|^{2q_0}+\sup_{(k-N)\vee0\leq i\leq k}e^{\tilde\e_0 t_{i+1}}|y(t_{i+1})|^{2q_0}\Big)\Big]\nn\\
&+Ke^{\tilde \e_0 t_{k+1}}+K\E\Big[\sum_{i=(k-N)\vee0}^k|b_i|^{q_0}\tau^{\frac{q_0}{4}+1}e^{\tilde \e_0 t_{i+1}}\Big]\nn\\
\leq &~\frac{1}{2q_0}\E\Big[\sup_{(k-N)\vee0\leq i\leq k}e^{\tilde\e_0 t_{i+1}}|y(t_{i+1})|^{2q_0}\Big]+K(1+\|x\|^{2q_0})e^{\tilde \e_0 t_{k+1}},
\end{align*}
where in the last step we used $e^{\zeta}\leq 1+2\zeta\leq 2$ for $\zeta\in(0,\frac12],$ the hypothesis (H2), and \eqref{10034}.
Therefore, 
\begin{align*}
\E\big[\sup_{(k-N)\vee 0\leq l\leq k}|y(t_{l+1})|^{2q_0}\big]
\leq&K(1+\|x\|^{2q_0}).
\end{align*} 
Combining the definition of $\{Y^{x^h,\Delta}_{t_k}\}_{k\in\mathbb N}$ (see \eqref{thL}) implies \eqref{sfdebound}.

\textbf{Proof of \eqref{sfdecov}.}
Introduce the continuous version of $\{y^{x,\tau}(t_k)\}_{k\in\mathbb N}$ as
\begin{align}\label{con_y(t)}
y^{x,\tau}(t)&=y^{x,\tau}(t_k)+b(Y^{x^{h},\Delta}_{t_k})(t-t_k)\nn\\
&\quad+\sigma(Y^{x^{h},\Delta}_{t_k}) \big(W(t)-W(t_k)\big)\quad\forall\; t\in[t_k,t_{k+1}).
\end{align}
Recalling $$Y^{x^h,\Delta}_{t}=\sum_{k=0}^{\infty}Y^{x^h,\Delta}_{t_k} \textbf 1_{[t_k, t_{k+1})}(t),$$
it follows from \eqref{sfde} and \eqref{con_y(t)} that
\begin{align*}
X^{x}(t)-y^{x,\tau}(t)=\int_{0}^{t}(b(X^{x}_s)-b(Y^{x^h,\Delta}_s))\mathrm ds+\int_{0}^{t}(\sigma(X^{x}_s)-\sigma(Y^{x^h,\Delta}_s))\mathrm dW_s.
\end{align*}
By the It\^o formula and the Young inequality we deduce that for any $\epsilon\in(0,1)$,
\begin{align*}
&e^{\epsilon t}|X^{x}(t)-y^{x,\tau}(t)|^2\nn\\
=&\int_{0}^{t}e^{\epsilon s}\Big(\epsilon|X^{x}(s)-y^{x,\tau}(s)|^2
+2\<X^{x}(s)-y^{x,\tau}(s),b(X^{x}_s)-b(Y^{x^h,\Delta}_s)\>\nn\\&+|\sigma(X^{x}_s)-\sigma(Y^{x^h,\Delta}_s)|^2\Big)\mathrm ds+\mathcal N_t\nn\\
\leq&\int_{0}^{t}e^{\epsilon s}\Big(\epsilon|X^{x}(s)-y^{x,\tau}(s)|^2
+2\<X^{x}(s)-y^{x,\tau}(s),b(X^{x}_s)-b(y^{x,\tau}_s)\>\nn\\
&+(1+\epsilon)|\sigma(X^{x}_s)-\sigma(y^{x,\tau}_s)|^2\Big)\mathrm ds+\int_{0}^{t}\epsilon e^{\epsilon s}|X^{x}(s)-y^{x,\tau}(s)|^2\mathrm ds+\mathcal E_t+\mathcal N_t,
\end{align*}
where  $\mathcal N_{t}$ is a martingale defined by
\begin{align*}
\mathcal N_{t}=2\int_{0}^{t}e^{\epsilon s}\<X^{x}(s)-y^{x,\tau}(s),(\sigma(X^{x}_s)-\sigma(Y^{x^h,\Delta}_s))\mathrm dW_s\>,
\end{align*}
 the $E$-valued stochastic process $\{y^{x,\tau}_t\}_{t\geq 0}$ is defined by
$$y^{x,\tau}_t(\theta)=y^{x,\tau}(t+\theta)~~~\forall\theta\in[-\delta_0,0],$$
and
\begin{align*}
\mathcal E_t:=\int_{0}^{t}e^{\epsilon s}\Big(\frac{1}{\epsilon}|b(y^{x,\tau}_s)-b(Y^{x^h,\Delta}_s)|^2+(1+\frac{1}{\epsilon})|\sigma(y^{x,\tau}_s)-\sigma(Y^{x^h,\Delta}_s)|^2\Big)\mathrm ds.
\end{align*}
Using  (H1)--(H2) yields
\begin{align*}
&e^{\epsilon t}|X^{x}(t)-y^{x,\tau}(t)|^2\nn\\
\leq&\int_{0}^{t}\Big(-\big(2\bar L_3-2\epsilon-(1+\epsilon)\bar L_1\big)e^{\epsilon s}|X^{x}(s)-y^{x,\tau}(s)|^2\nn\\
&+2\bar L_4\int_{-\delta_0}^{0}e^{\epsilon s}|X^{x}_s(r)-y^{x,\tau}_s(r)|^2\mathrm d\nu_2(r)+(1+\epsilon)\bar L_1\int_{-\delta_0}^{0}e^{\epsilon s}|X^{x}_s(r)-y^{x,\tau}_s(r)|^2\mathrm d\nu_1(r)\Big)\mathrm ds\nn\\
&+\mathcal E_t+\mathcal N_t\nn\\
\leq&-\Big(2\bar L_3-2\epsilon-(1+\epsilon)\bar L_1-(2\bar L_4 +(1+\epsilon)\bar L_1)e^{\epsilon\delta_0}\Big)\int_{0}^{t}e^{\epsilon s}|X^{x}(s)-y^{x,\tau}(s)|^2\mathrm ds\nn\\
&+(2\bar L_4 +(1+\epsilon)\bar L_1)e^{\epsilon\delta_0}\int_{-\delta_0}^{0}e^{\epsilon s}|X^x(s)-y^{x,\tau}(s)|^2\mathrm ds+\mathcal E_t+\mathcal N_t\nn\\
=&-\Big(2\bar L_3-2\epsilon-(1+\epsilon)\bar L_1-(2\bar L_4 +(1+\epsilon)\bar L_1)e^{\epsilon\delta_0}\Big)\int_{0}^{t}e^{\epsilon s}|X^{x}(s)-y^{x,\tau}(s)|^2\mathrm ds+\mathcal E_t+\mathcal N_t,
\end{align*}
where we used 
\begin{align*}
&\int_{0}^{t}\int_{-\delta_0}^{0}e^{\epsilon s}|X^{x}_s(r)-y^{x,\tau}_s(r)|^2\mathrm d\nu_i(r)\mathrm ds\nn\\
\leq &~e^{\epsilon \delta_0}\int_{0}^{t}\int_{-\delta_0}^{0}e^{\epsilon (s+r)}|X^{x}(s+r)-y^{x,\tau}(s+r)|^2\mathrm d\nu_i(r)\mathrm ds\nn\\
\leq&~ e^{\epsilon \delta_0}\int_{-\delta_0}^{0}e^{\epsilon s}|X^{x}(s)-y^{x,\tau}(s)|^2\mathrm ds+e^{\epsilon \delta_0}\int_{0}^{t}e^{\epsilon s}|X^{x}(s)-y^{x,\tau}(s)|^2\mathrm ds\nn\\
=&~e^{\epsilon \delta_0}\int_{0}^{t}e^{\epsilon s}|X^{x}(s)-y^{x,\tau}(s)|^2\mathrm ds~~~\mbox{for}~i=1,2.
\end{align*}
In view of $\bar L_3>\bar L_1+\bar L_4$, choose  $\epsilon\in(0,1)$ sufficiently small such that
$$2\bar L_3-2\epsilon-(1+\epsilon)\bar L_1-(2\bar L_4 +(1+\epsilon)\bar L_1)e^{\epsilon \delta_0}\geq0.$$
This implies 
\begin{align}\label{l4.5.2}
e^{\epsilon t}|X^{x}(t)-y^{x,\tau}(t)|^2
\leq \mathcal E_t+\mathcal N_t.
\end{align}
Making use of (H1)--(H3), 
 similarly to proofs of \cite[Lemma 3.3]{Mao2003} and \cite[Lemma 3.6]{MLS}, one has
 \begin{align}\label{1001}
 \sup_{t\geq0}\E[\|y^{x,\tau}_t-Y^{x^h,\Delta}_t\|^2]\leq K(1+\|x\|^2)\tau.
 \end{align}
This, along with  \eqref{l4.5.2} and  (H1) implies that
\begin{align}\label{l4.5.3}
\E[|X^{x}(t)-y^{x,\tau}(t)|^2]\leq{}& e^{-\epsilon t}\E[ \mathcal E_t]
\leq K\int_{0}^{t}e^{\epsilon (s-t)}
\sup_{\theta\in[-\delta_0,0]}\E[|y^{x,\tau}_s(\theta)-Y^{x^h,\Delta}_s(\theta)|^2]\mathrm ds \nn\\
\leq {}&K(1+\|x\|^{2})\tau.
\end{align}
Furthermore, according to \eqref{l4.5.2}, and using the Burkholder--Davis--Gundy inequality, we have
\begin{align*}
&\E\Big[\sup_{(t-\delta_0)\vee 0 \leq u\leq t}e^{\epsilon u}|X^{x}(u)-y^{x,\tau}(u)|^2\Big]
\leq \E[\mathcal E_t]+\E \Big[\sup_{(t-\delta_0)\vee 0 \leq u\leq t}\mathcal N_u\Big]\nn\\
\leq&K(1+\|x\|^{2})e^{\epsilon t}\tau+K\E\Big[\Big(\int_{(t-\delta_0)\vee 0}^{t}e^{2\epsilon s}|X^{x}(s)-y^{x,\tau}(s)|^2|\sigma(X^{x}_s)-\sigma(Y^{x^h,\Delta}_s)|^2\mathrm ds\Big)^{\frac{1}{2}}\Big]\nn\\
\leq&K(1+\|x\|^{2})e^{\epsilon t}\tau+K\E\Big[\big(\sup_{(t-\delta_0)\vee 0 \leq u\leq t}e^{\epsilon u}|X^{x}(u)-y^{x,\tau}(u)|^2\big)^{\frac12}\nn\\
&\times\Big(\int_{(t-\delta_0)\vee 0}^{t}e^{\epsilon s}|\sigma(X^{x}_s)-\sigma(Y^{x^h,\Delta}_s)|^2\mathrm ds\Big)^{\frac{1}{2}}\Big].
\end{align*}
Applying the Young inequality, and then by (H1) and \eqref{1001}  we arrive at
\begin{align*}
&\E\Big[\sup_{(t-\delta_0)\vee 0 \leq u\leq t}e^{\epsilon u}|X^{x}(u)-y^{x,\tau}(u)|^2\Big]\nn\\
\leq{}&K(1+\|x\|^{2})e^{\epsilon t}\tau+\frac{1}{2}\E\Big[\sup_{(t-\delta_0)\vee 0 \leq u\leq t}e^{\epsilon u}|X^{x}(u)-y^{x,\tau}(u)|^2\Big]\nn\\
&+K\E\Big[\int_{(t-\delta_0)\vee 0}^{t}e^{\epsilon s}|\sigma(X^{x}_s)-\sigma(Y^{x^h,\Delta}_s)|^2\mathrm ds\Big]\nn\\
\leq&K(1+\|x\|^{2})e^{\epsilon t}\tau+\frac{1}{2}\E\Big[\sup_{(t-\delta_0)\vee 0 \leq u\leq t}e^{\epsilon u}|X^{x}(u)-y^{x,\tau}(u)|^2\Big]\nn\\
&+K\int_{(t-\delta_0)\vee 0}^{t}e^{\epsilon s}\E\big[|X^{x}(s)-Y^{x^h,\Delta}(s)|^2\big]\mathrm ds\nn\\
&+\int_{(t-\delta_0)\vee 0}^{t}e^{\epsilon s}\int_{-\delta_0}^{0}\E\big[|X^{x}_s(r)-Y^{x^h,\Delta}_s(r)|^2\big]\mathrm d \nu_{1}(r)\mathrm ds.
\end{align*}
Making use of  \eqref{1001} and \eqref{l4.5.3} leads to
\begin{align*}
&\frac{1}{2}e^{\epsilon(t-\delta_0)\vee 0}\E\Big[\sup_{(t-\delta_0)\vee 0 \leq u\leq t}|X^{x}(u)-y^{x,\tau}(u)|^2\Big]\nn\\
\leq {}&\frac{1}{2}\E\Big[\sup_{(t-\delta_0)\vee 0 \leq u\leq t}e^{\epsilon u}|X^{x}(u)-y^{x,\tau}(u)|^2\Big]\nn\\
\leq{}& K(1+\|x\|^{2})e^{\epsilon t}\tau,
\end{align*}
which implies
$$\E\Big[\sup_{(t-\tau)\vee 0 \leq u\leq t}|X^{x}(u)-y^{x,\tau}(u)|^2\Big]\leq {}K(1+\|x\|^{2})\tau.$$
Combining \eqref{1001} finishes the proof of the desired assertion \eqref{sfdecov}.
\end{proof}

\end{appendix}

\bibliographystyle{plain}
\bibliography{ldp.bib}

\end{document}